\documentclass[USenglish]{article}

\usepackage[utf8]{inputenc}
\usepackage[small]{dgruyter}
\usepackage{microtype}
\usepackage{amssymb}
\usepackage{bm}
\usepackage{pgfplots}
\usepackage{algorithm}
\usepackage{algorithmic}
\usepackage{mathtools}

\newcommand{\x}{\ensuremath{\boldsymbol{x}}}
\newcommand{\y}{\ensuremath{\boldsymbol{y}}}
\newcommand{\mupar}{\ensuremath{\boldsymbol{\mu}}}
\newcommand{\etapar}{\ensuremath{\boldsymbol{\eta}}}
\newcommand{\peim}{\mathcal{P}_{\mathrm{EIM}}}

\newcommand{\spn}[1]{\mathrm{span}\{#1\}}

\newcommand*{\argmax}{\operatornamewithlimits{\arg\max}\limits}
\newcommand*{\argsup}{\operatornamewithlimits{arg\,sup}\limits}
\DeclareMathOperator*{\dive}{\mathrm{div}}

\providecommand{\norm}[1]{\lVert#1\rVert}
\graphicspath{{figures/}}

\begin{document}


  \author[1]{Gianluigi Rozza}
  \author[2]{Martin Hess}
  \author[3]{Giovanni Stabile}
  \author[4]{Marco Tezzele}
  \author[5]{Francesco Ballarin}
  \runningauthor{G.~Rozza et al.}
  \affil[1]{Mathematics Area, SISSA mathLab, Trieste, Italy, \texttt{gianluigi.rozza@sissa.it}}
  \affil[2]{Mathematics Area, SISSA mathLab, Trieste, Italy, \newline \texttt{martin.hess@sissa.it}}
  \affil[3]{Mathematics Area, SISSA mathLab, Trieste, Italy, \texttt{giovanni.stabile@sissa.it}}
  \affil[4]{Mathematics Area, SISSA mathLab, Trieste, Italy, \texttt{marco.tezzele@sissa.it}}
  \affil[5]{Mathematics Area, SISSA mathLab, Trieste, Italy, \texttt{francesco.ballarin@sissa.it}}

  \title{Basic Ideas and Tools for Projection-Based Model Reduction of Parametric Partial Differential Equations}
  \runningtitle{Basic Ideas and Tools for MOR of Parametric PDEs}

  \abstract{We provide first the functional analysis background required for reduced order modeling and present the underlying concepts of reduced basis model reduction. 
The projection-based model reduction framework under affinity assumptions, offline-online decomposition and error estimation is introduced. 
Several tools for geometry parametrizations, such as free form deformation, radial basis function interpolation and inverse distance weighting interpolation are explained.
The empirical interpolation method is introduced as a general tool to deal with non-affine parameter dependency and non-linear problems. 
The discrete and matrix versions of the empirical interpolation are considered as well.
Active subspaces properties are discussed to reduce high-dimensional parameter spaces as a pre-processing step.
Several examples illustrate the methodologies.} 
   

  \startpage{1}

\maketitle




\newtheorem{thm}{Theorem}[section]
\newtheorem{defn}[thm]{Definition}
\newtheorem{lemma}[thm]{Lemma}
\theoremstyle{dgthm}

\section*{Introduction}

Parametric model order reduction techniques have been developed in recent decades to deal with increasingly complex computational tasks.
The ability to compute how quantities of interest change with respect to parameter variations provides insight and understanding, which is vital in all areas of science and engineering.
Model reduction thus allows to deal with optimization or inverse problems of a whole new scale.
Each chapter of the handbook gives an in-depth view of a model order reduction (MOR, or reduced order modeling (ROM)) method, a particular application area, analytical, 
numerical or technical aspects of software frameworks for model reduction.

There exists a large number of MOR techniques used in many areas of science and engineering
to improve computational performances and contain costs in a repetitive computational environment such as many-query and real-time computing~\cite{Schilders2008}.
We assume a given parametrized partial differential equation (PDE) as starting point of the model reduction procedure.
Typical parameters of interest are material coefficients, corresponding to physical qualities of the media which constitute the 
domain where the PDE is solved. 
Also a variable geometry can be of special interest in a task to find the optimal device configuration.
Physical states such as the temperature might be considered an input parameter. 
It is a task of the mathematical modeling to identify the parameters of interest and how they enter the PDE. 
Once a parametrized model is identified, the MOR techniques described in this and the following chapters can be used either in 
a `black-box' fashion (non-intrusive way) or by intrusive means, which will be explained in detail, whenever this is necessary.

The particular numerical method to solve a PDE is most often not relevant to the model reduction procedure. 
We will therefore assume there is a numerical method available, which solves the problem to any required accuracy, and
move seamlessly from the continuous form to the discretized form.

This chapter covers briefly the functional analysis framework relevant to many, but not all, MOR methods. 
Presented is the starting point of PDE-oriented MOR techniques, which can be found in the various chapters of the handbook.

In particular, section~\ref{sec:basicNotions} provides what is needed for the projection-based ROM. 
Starting from the setting of the classical Lax-Milgram theorem for elliptic PDEs in subsection~\ref{chap_1_1:subsec_ppde} and subsection~\ref{chap_1_1:subsec_variational}, 
a numerical discretisation is introduced in subsection~\ref{chap_1_1:subsubsec_numdisc}.
Due to brevity of representation, many concepts of functional analysis and theory of PDEs are only ouched upon.
Many references to the literature for further reading are given.

Projection-based ROM is presented in subsection~\ref{chap_1_1:subsec:MOR_base_concepts}, with the following topics covered in detail: 
proper orthogonal decomposition in subsection~\ref{chap_1_1:subsubsec_POD}, the greedy algorithm in subsection~\ref{subsec:greedy}, 
the projection framework in subsection~\ref{chap_1_1:subsubsec_projectROM}, affine parameter dependency in subsection~\ref{subsec:affine}, 
the offline-online decomposition in subsection~\ref{chap_1_1:subsubsec_offon_decomp} and basic error estimation in subsection~\ref{subsec:error_bounds}.

Section~\ref{sec:geom} introduces efficient techniques for geometric parametrizations, arising from a reference domain approach, such as 
free form deformation in subsection~\ref{sec:ffd}, radial basis function interpolation in subsection~\ref{sec:rbf} and inverse distance weighting in subsection~\ref{sec:idw}.

A widely used method to generate an approximate affine parameter dependency is the \emph{empirical interpolation method}. 
The original \emph{empirical interpolation method} is presented in section~\ref{sec:nonaffinity} as well as 
the \emph{discrete empirical interpolation method} in subsection~\ref{subsec:DEIM} and further options in subsection~\ref{subsec:further_options}.
Several numerical examples show the use of the \emph{empirical interpolation method} in subsection~\ref{subsec:example}.

Section~\ref{sec:active} introduces active subspaces as a pre-processing step to reduce the parameter space dimension.
Corresponding examples are provided in subsection~\ref{chap_1_4:some_examples} and also nonlinear dimensionality reduction is briefly discussed in subsection~\ref{sec:nonlinear}.

A brief conclusion and outlook at the handbook is given in section~\ref{chap1:sec:conclude}.


 \section{Basic Notions and Tools}
\label{sec:basicNotions}
 
We briefly cover a few main results of linear functional analysis and the analysis of partial differential equations (PDEs).
This material serves as a reminder of the underlying concepts of model reduction but can not replace a textbook on these subjects.
For a more thorough background, we refer to the literature on functional analysis~\cite{Ciarlet2014, Yosida1995},
partial differential equations~\cite{Adams1075, Evans1998, Renardy2004, Rudin1976},
and numerical methods~\cite{AINSWORTH19971, BABUSKA1970/71, ciarlet2002, GRATSCH2005235, quarteroni2017, Verfuerth2013}. 



\subsection{Parametrized Partial Differential Equations}
\label{chap_1_1:subsec_ppde}


Let $\Omega \subset \mathbb{R}^d$ denote a spatial domain in $d = 1, 2$ or $3$ dimensions with boundary $\partial \Omega$.
A Dirichlet boundary $\Gamma_D \subset \partial \Omega$ is given, where essential boundary conditions on the field of interest are prescribed.
Introduce a Hilbert space $V(\Omega)$ equipped with inner product $(\cdot, \cdot)_V$ and induced norm $\| \cdot \|_V$.
A Hilbert space $V(\Omega)$ is a function space, i.e., a function $u \in V(\Omega)$ is seen as a point in the vector space $V$, as is common in functional analysis.
Each $u \in V(\Omega)$ defines a mapping $x \in \Omega \mapsto u(x) \in \mathbb{R}$ or $x \in \Omega \mapsto u(x) \in \mathbb{C}$, depending on whether a real or complex Hilbert space is considered.
In many applications, $V$ is a subset of the Sobolev space $H^1(\Omega)$ as $V(\Omega) = \{v \in H^1(\Omega)\colon v\vert_{\Gamma_D} = 0 \} $.
Vector-valued Hilbert spaces can be constructed using the Cartesian product of $V(\Omega)$.
Given a parameter domain $\mathcal{P} \subset \mathbb{R}^p$, a particular parameter point is denoted by the $p$-tuple $\boldsymbol\mu  = ( \mu_1, \mu_2, \ldots, \mu_p )$.
The set of all linear and continous forms on $V$ defines the dual space $V'$ and let $L \in \mathcal{L}(V, V')$ denote a linear differential operator.

A field variable $u \in V : \Omega \rightarrow \mathbb{R}$ is defined implicitly as the solution to a parametrized linear partial differential equation (PDE)
through the operator $L : V \times \mathcal{P} \rightarrow V'$ with $L(\cdot;\boldsymbol\mu) \in \mathcal{L}(V, V')$ and load vector $f_L(\boldsymbol\mu) \in V'$ 
for each fixed $\boldsymbol\mu$, as

\begin{equation}
 L(u; \boldsymbol\mu) = f_L(\boldsymbol\mu) .
\label{chap_1_1:Lu}
\end{equation}

As in the case of function spaces, operators between function spaces form vector spaces themselves, such as $L(\cdot;\boldsymbol\mu) \in \mathcal{L}(V, V')$, with $\mathcal{L}(V, V')$ being the 
space of operators mapping from the vector space $V$ to $V'$. 

Typical examples of scalar-valued linear PDEs are the Poisson equation, Heat equation or Wave equation, while typical examples of vector-valued linear PDEs are
Maxwells equations or Stokes equations. 
The nonlinear case will be addressed in various chapters as well: examples of nonlinear PDEs are the Navier-Stokes system or the equations describing nonlinear elasticity.




\subsection{Parametrized Variational Formulation}
\label{chap_1_1:subsec_variational}


The variational or weak form of a parametrized linear PDE in the continuous setting is given as 

\begin{equation}
 a(u(\boldsymbol\mu),v;\boldsymbol\mu) = f(v;\boldsymbol\mu) \quad \forall v \in V,
\label{chap_1_1:cont_weak}
\end{equation}

\noindent with bilinear form $a : V \times V \times \mathcal{P} \rightarrow \mathbb{R}$ and linear form $f : V \times \mathcal{P} \rightarrow \mathbb{R}$.
In many application scenarios, a particular output of interest is sought, given by the linear form $l : V \times \mathcal{P} \rightarrow \mathbb{R}$ as

\begin{equation}
 s(\boldsymbol\mu) = l(u(\boldsymbol\mu);\boldsymbol\mu) .
\end{equation}

In the case that $a(\cdot, \cdot; \boldsymbol\mu)$ is symmetric and $l = f$, the problem is called compliant.

For each $\boldsymbol\mu \in \mathcal{P}$ assume coercivity and continuity of the bilinear form $a(\cdot,\cdot;\boldsymbol\mu)$, i.e.,

\begin{eqnarray}
 a(w,w;\boldsymbol\mu) &\geq& \alpha(\boldsymbol\mu) \|w\|^2_V , \\
 a(w,v;\boldsymbol\mu) &\leq& \gamma(\boldsymbol\mu) \| w \|_V \| v \|_V ,
\end{eqnarray}

\noindent and continuity of the linear form $f(\cdot;\boldsymbol\mu)$, 

\begin{equation}
 f(w;\boldsymbol\mu) \leq \delta(\boldsymbol\mu) \|w\|_V ,
\end{equation}

\noindent with parameter-independent bounds, which satisfy $0 < \alpha \leq \alpha(\boldsymbol\mu)$, $\gamma(\boldsymbol\mu) \leq \gamma < \infty$ and $\delta(\boldsymbol\mu) \leq \delta < \infty$.
To do actual computations, the biliner form is discretized into a linear equation. The coercivity property means that the matrix discretizing the bilinear form will be positive definite.

For fixed parameter the well-posedness of \eqref{chap_1_1:cont_weak} is then established by the Lax-Milgram theorem:

\begin{thm}

\textbf{Lax-Milgram Theorem}

Let $a : V \times V \rightarrow \mathbb{R}$ be a continuous and coercive bilinear form over a Hilbert space $V$
and $f \in V'$ a continuous linear form. Then the variational problem 

\begin{equation}
 a(u,v) = f(v) \quad \forall v \in V,
\label{chap_1_1:cont_weak_LaxMilgram}
\end{equation}

\noindent has a unique solution $u \in V$ and it holds

\begin{equation}
 \|u\|_V \leq \frac{1}{\alpha} \|f \|_V ,
\end{equation}

\noindent with the coercivity constant $\alpha > 0$ of the bilinear form.

\end{thm}

Thus, in the parametric setting, the $\boldsymbol\mu$-dependence also carries over to the coercivity constant as $\alpha = \alpha(\boldsymbol\mu)$.

The function space in which the field variable resides is called the ansatz space, while the second function space is called the test space, i.e., where a test function $v$ resides. 
If the test space is distinct from the ansatz space then the bilinear form is defined over $a : V \times W \times \mathcal{P} \rightarrow \mathbb{R}$
for $V$ and $W$ Hilbert spaces. With $f \in W'$ and for fixed $\boldsymbol\mu$, the well-posedness is then established through the Banach-Ne\v cas-Babu\v ska theorem.

\begin{thm}

\textbf{Banach-Ne\v cas-Babu\v ska Theorem}
Let $V$ and $W$ denote Hilbert spaces, $a : V \times W \rightarrow \mathbb{R}$ a continuous bilinear form and $f \in W'$. Then the variational problem

\begin{equation}
  a(u,v) = f(v) \quad \forall v \in W,
\end{equation}

\noindent has a unique solution if and only if 

(i) the inf-sup condition holds, i.e, 

\begin{align*}
 \exists \beta > 0 \text{ , s.t., } \beta \leq \adjustlimits \inf_{v \in V \setminus \{0\}} \sup_{w \in W \setminus \{0\}} \frac{a(v,w)}{ \| v \|_V  \| w \|_W },
\end{align*}

(ii) $\forall w \in W:$ 

\begin{align*}
 \{ a(v,w) = 0 \quad \forall v \in V \} \implies w = 0 .
\end{align*}

\end{thm}


\subsubsection{Discretized Parametrized Variational Formulation}
\label{chap_1_1:subsubsec_numdisc}

The method of weighted residuals is used to cast \eqref{chap_1_1:Lu} into a discrete variational formulation.
Given the linear PDE $L(u;\boldsymbol\mu) = f_L(\boldsymbol\mu)$, consider a discrete, i.e., finite dimensional approximation $u_h \in V_h \subset V$
to $u$ as

\begin{equation}
 u_h(\boldsymbol\mu) = \sum_{i=1}^{N_h} u_h^{(i)} \varphi^i .
  \label{chap_1_1:discr_exp}
\end{equation}

The dimension of $V_h$ is $N_h$ and the set of ansatz functions $\varphi^i(\mathbf{x}) : \Omega \rightarrow \mathbb{R}$ belong to $V$.
The $u_h^{(i)}$ are scalar coefficients such that the vector $\textbf{u}_h = (u_h^{(1)}, \ldots, u_h^{(N_h)})^T \in \mathbb{R}^{N_h}$ is the coordinate 
representation of $u_h$ in the basis $\{\varphi^i\}$ of $V_h$.
A conforming discretizations is considered, i.e., $V_h \subset V$ holds.

Plugging \eqref{chap_1_1:discr_exp} into \eqref{chap_1_1:Lu} yields the discrete residual $R(u_h(\boldsymbol\mu)) = L(u_h(\boldsymbol\mu);\boldsymbol\mu) - f_L(\boldsymbol\mu) \in V'$.
To compute the scalar coefficients $u_h^{(i)}$, Galerkin orthogonality is invoked, as 

\begin{equation}
 0 = (\varphi_j, R)_{(V,V')}, \quad j = 1 \ldots N_h ,
 \label{chap_1_1:mwr_galerkin_orth}
\end{equation}

\noindent where $(\cdot, \cdot)_{(V,V')}$ is the duality pairing between $V$ and $V'$.

In short, Galerkin orthogonality means that the test space is orthogonal to the residual.
In Ritz-Galerkin methods, the residual is tested against the same set of functions as the ansatz functions.
If test space and trial space are different, one speaks of a Petrov-Galerkin method.
Numerous discretization methods can be understood in terms of the method of weighted residuals.
They are distinguished by the particular choice of trial and test space.

The well-posedness of the discrete setting follows the presentation of the continuous setting, by casting the equations and properties over $V_h$ instead of $V$.

The weak form in the discrete setting is given as 

\begin{equation}
 a(u_h({\boldsymbol\mu}),v_h;{\boldsymbol\mu}) = f(v_h;{\boldsymbol\mu}) \quad \forall v_h \in V_h,
\label{chap_1_1:discr_weak}
\end{equation}

\noindent with bilinear form $a : V_h \times V_h \times \mathcal{P} \rightarrow \mathbb{R}$ and linear form $f : V_h \times \mathcal{P} \rightarrow \mathbb{R}$.
The discrete bilinear form is then derived from \eqref{chap_1_1:mwr_galerkin_orth} through the integration-by-parts formula and Green's theorem.

Correspondingly, the discrete coercivity constant $\alpha_h(\mu)$ and the discrete continuity constant $\gamma_h(\mu)$ are defined

\begin{eqnarray}
 \alpha_h(\boldsymbol\mu) &=& \min_{w_h \in V_h} \frac{a(w_h,w_h;\boldsymbol\mu)}{\|w_h\|^2_{V_h}} , \\
 \gamma_h(\boldsymbol\mu) &=& \max_{w_h \in V_h} \max_{v_h \in V_h} \frac{a(w_h,v_h;\boldsymbol\mu)}{\|w_h\|_{V_h} \|v_h\|_{V_h}} .
\end{eqnarray}

The well-posedness of \eqref{chap_1_1:cont_weak} is then analogously established by the Lax-Milgram theorem and  Banach-Ne\v cas-Babu\v ska theorem.
Cea's Lemma is a fundamental result about the approximation quality that can be achieved:

\begin{lemma}
 
\textbf{Cea's Lemma}
Let $a : V \times V \rightarrow \mathbb{R}$ be a continuous and coercive bilinear form over a Hilbert space $V$
and $f \in V'$ a continuous linear form. Given a conforming finite-dimensional subspace $V_h \subset V$, the continuity constant $\gamma$ and 
coercivity constant $\alpha$ of $a(\cdot,\cdot)$ it holds for the solution $u_h$ to 

\begin{equation}
 a(u_h,v_h) = f(v_h) \quad \forall v_h \in V_h ,
\end{equation}

\noindent that 

\begin{equation}
 \| u - u_h \|_V \leq \frac{\gamma}{\alpha} \inf_{v_h \in V_h} \|u - v_h \|_V . 
\end{equation}

\end{lemma}

The stiffness matrix $\mathbb{A}_h \in \mathbb{R}^{N_h \times N_h}$ assembles the bilinear form entrywise as $(\mathbb{A}_h)_{ij} = a(\varphi^j, \varphi^i)$.
The load vector $\textbf{f}_h \in \mathbb{R}^{N_h}$ is assembled entrywise as $(\textbf{f}_h)_i = f(\varphi^i)$ and the
solution vector is denoted $\textbf{u}_h$ with coefficients $u_h^{(j)}$.

Then solving \eqref{chap_1_1:discr_weak} amounts to solving the linear system

\begin{equation}
 \mathbb{A}_h \textbf{u}_h = \textbf{f}_h .
\label{chap_1_1:linalg_dscr_weak}
\end{equation}

The most common discretization method is the finite element method (FEM)~\cite{Boffi2013}, besides finite difference~\cite{Smith1985}, 
discontinuous Galerkin~\cite{ArnoldBrezzi2002}, finite volume~\cite{Eymard2000}
and spectral element method~\cite{CHQZ1}.

\subsection{Model Reduction Basic Concepts}
\label{chap_1_1:subsec:MOR_base_concepts}

A wide variety of reduced order modeling methods exist today, thanks to large research efforts in the last decades.
Reduced basis model order reduction is a projection-based MOR method and also shares many features with other MOR methods, so that the topics mentioned here will occur throughout the handbook.
Two common algorithms for the generation of a projection space, the proper orthogonal decomposition and the greedy algorithm, are presented first.

\subsubsection{Proper Orthogonal Decomposition}
\label{chap_1_1:subsubsec_POD}

Assume a sampled set of high-fidelity solutions $\{u_h({\boldsymbol\mu_i}), i = 1, \ldots, N_{max} \}$, i.e., solutions to \eqref{chap_1_1:discr_weak} or \eqref{chap_1_1:linalg_dscr_weak}, respectively.
The discrete solution vectors are stored column-wise in a snapshot matrix $\mathbb{S} \in \mathbb{R}^{N_h \times N_{max}}$.
The proper orthogonal decomposition (POD) compresses the data stored in $\mathbb{S}$ by computing an orthogonal matrix $\mathbb{V}$, which is a best approximation in the least square sense to $\mathbb{S}$.
In particular, the POD solution of size $N$ is the solution to

\begin{eqnarray}
\label{chap_1_1:minim_SVD1}
 \min_{\mathbb{V} \in \mathbb{R}^{N_h \times N}} \| \mathbb{S} - \mathbb{V} \mathbb{V}^T \mathbb{S} \|_F , \\ 
 \label{chap_1_1:minim_SVD2}
 \text{subject to } \mathbb{V}^T \mathbb{V} = \mathbb{I}_{N \times N} ,
\end{eqnarray}

\noindent with $\| \cdot \|_F$ the Frobenius norm and $\mathbb{I}_{N \times N}$ the identity matrix.

There exists a solution to \eqref{chap_1_1:minim_SVD1} - \eqref{chap_1_1:minim_SVD2} according to the Eckardt-Young-Mirsky theorem~\cite{Eckart1936}, which can be computed with the
singular value decomposition (SVD) as

\begin{equation}
 \mathbb{S} = \mathbb{U} \Sigma \mathbb{Z},
\end{equation}

\noindent with orthogonal matrix $\mathbb{U} \in \mathbb{R}^{N_h \times N_h}$, rectangular diagonal matrix $\Sigma \in \mathbb{R}^{N_h \times N_{max}}$ 
and orthogonal matrix $\mathbb{Z} \in \mathbb{R}^{N_{max} \times N_{max}}$. 
The solution $\mathbb{V}$ is composed of the first $N$ column vectors of $\mathbb{U}$. They are also called the \emph{POD modes}.
The diagonal entries $\{\sigma_i, i = 1, \dots, \min(N_h, N_{max})\}$ of $\Sigma$ are non-negative and called \emph{singular values}.
It holds that 

\begin{equation}
 \min_{\mathbb{V} \in \mathbb{R}^{N_h \times N}} \| \mathbb{S} - \mathbb{V} \mathbb{V}^T \mathbb{S} \|_F = \sum_{i = N + 1}^{\min(N_h, N_{max})} \sigma_i .
 \label{chap_1_1:trunc_SVD}
\end{equation}

Thus, the neglected singular values give an indication of the approximate truncation error. 
In practise, a high tolerance threshold like $99\%$ or  $99.99\%$ is chosen and $N$ is determined so that the sum of the first $N$ 
\emph{singular values} reaches this percentage of the sum of all \emph{singular values}. 
In many applications, an exponential singular value decay can be observed, which allows to reach the tolerance with a few POD modes.

\subsubsection{Greedy Algorithm}\label{subsec:greedy}

The greedy algorithm also computes an orthogonal matrix $\mathbb{V} \in \mathbb{R}^{N_h \times N}$ to serve as a projection operator, just as in the POD case.
The greedy algorithm is an iterative procedure, which enriches the snapshot space according to where an error indicator or error estimator $\Delta$ attains its maximum.
Starting from a field solution at a given initial parameter value, the parameter location is sought, whose field solution is worst approximated with the initial solution.
This solution is then computed and appended to the projection matrix to obtain a $2$-dimensional projection space. 
The greedy typically searches for new snapshot solutions within a discrete surrogate $P$ of the parameter space $\mathcal{P}$.
The process is repeated until a given tolerance on the error estimator is fulfilled.
The error estimator is residual-based and estimates the error between a reduced order solve for a projection space $\mathbb{V}$ and the high-fidelity solution, see subsection \ref{subsec:error_bounds}.
The greedy algorithm is stated in pseudocode in algorithm \ref{alg:greedy}.

\begin{algorithm}
\caption{The greedy algorithm}
\begin{algorithmic} 
\REQUIRE discrete surrogate $P$ of parameter space $\mathcal{P}$, approximation tolerance \texttt{tol}, initial parameter $\boldsymbol\mu_1$
\ENSURE projection matrix $\mathbb{V}$
\STATE $N = 1$
\STATE $\mathbb{V}_1 = \frac{\textbf{u}_h(\boldsymbol\mu_1)}{ \| \textbf{u}_h(\boldsymbol\mu_1) \| } $
\WHILE {$\max_{\boldsymbol\mu \in P} \Delta(\boldsymbol\mu) > $ \texttt{tol}}
\STATE $N = N + 1$
\STATE $\boldsymbol\mu_N = \argmax_{\boldsymbol\mu \in P} \Delta(\boldsymbol\mu)$
\STATE solve \eqref{chap_1_1:linalg_dscr_weak} at $\boldsymbol\mu_N$ for $\textbf{u}_h(\boldsymbol\mu_N)$
\STATE orthonormalize $\textbf{u}_h(\boldsymbol\mu_N)$ with respect to $\mathbb{V}_{N-1}$ to obtain $\zeta_N$
\STATE append $\zeta_N$ to $\mathbb{V}_{N-1}$ to obtain $\mathbb{V}_{N}$
\ENDWHILE
\STATE set $\mathbb{V} = \mathbb{V}_{N}$
\end{algorithmic}
\label{alg:greedy}
\end{algorithm}

\subsubsection{Reduced Order System}
\label{chap_1_1:subsubsec_projectROM}

Starting from the discrete high-fidelity formulation \eqref{chap_1_1:discr_weak}, another Galerkin projection is invoked to arrive at
the reduced order formulation. 
Assume a projection space $V_N$ is then determined through either a proper orthogonal decomposition (POD) or the greedy sampling, with $\mathbb{V} \in \mathbb{R}^{N_h \times N}$ denoting a discrete 
basis of $V_N$. 
Thus $V_N \subset V_h$ and $\dim V_N = N$. 

The reduced order variational formulation is to determine $u_N(\boldsymbol\mu) \in V_N$, such that

\begin{equation}
 a(u_N(\boldsymbol\mu),v_N;\boldsymbol\mu) = f(v_N;\boldsymbol\mu) \quad \forall v_N \in V_N.
\label{chap_1_1:ROM_weak}
\end{equation}

Eq. \eqref{chap_1_1:linalg_dscr_weak} is then projected onto the reduced order space as 

\begin{equation}
 \mathbb{V}^T \mathbb{A}_h \mathbb{V} \textbf{u}_N = \mathbb{V}^T \textbf{f}_h.
\label{chap_1_1:linalg_rom_weak}
\end{equation}

The reduced system matrix $\mathbb{A}_N = \mathbb{V}^T \mathbb{A}_h \mathbb{V}$ is then a dense matrix of small size $N \times N$ as depicted in \eqref{chap_1_1:ROM_FOM_project_pic}:

\begin{equation}
\begin{bmatrix} \mathbb{A}_N \end{bmatrix} =   
\label{chap_1_1:ROM_FOM_project_pic}
\begin{bmatrix} \\ \mathbb{V}  \\ \\ \end{bmatrix}^T
\begin{bmatrix} a(\varphi^1, \varphi^1) &  \ldots & \ldots  \\
 \ldots  & \ldots &  \ldots \\
\ldots  & \ldots &  a(\varphi^{N_h}, \varphi^{N_h})
\end{bmatrix}
\begin{bmatrix} \\ \mathbb{V}  \\ \\ \end{bmatrix}.
\end{equation}

The high-order solution is then approximated as 

\begin{equation}
 \textbf{u}_h \approx \mathbb{V} \textbf{u}_N .
\label{chap_1_1:ROMu_approx_FOMu}
\end{equation}

\subsubsection{Affine Parameter Dependency}\label{subsec:affine}

Many MOR algorithms rely on an affine parameter dependency, because
the affine parameter dependency provides the computational efficiency of the model reduction.
Thus, it is a significant advancement from the 2000's~\cite{Patera:2008}
over the first use of reduced order models~\cite{Almroth1978, Noor1982}.

An affine parameter dependency means that the bilinear form can be expanded as

\begin{equation}
 a(\cdot,\cdot;\boldsymbol\mu) = \sum_{i=1}^{Q_a} \Theta_a^i(\boldsymbol\mu) a_i(\cdot,\cdot) ,
\label{chap_1_1:affine_param_dep_bilifo}
\end{equation}

\noindent and affine expansions hold as

\begin{align}
 f(\cdot; \boldsymbol\mu) &= \sum_{i=1}^{Q_f} \Theta_f^i(\boldsymbol\mu) f_i(\cdot) , \label{chap_1_1:affine_param_dep_bilifo_f}\\
 l(\cdot; \boldsymbol\mu) &= \sum_{i=1}^{Q_l} \Theta_l^i(\boldsymbol\mu) l_i(\cdot) , \label{chap_1_1:affine_param_dep_bilifo_l}
\end{align}

\noindent with scalar-valued functions $\Theta_a^i : \mathcal{P} \rightarrow \mathbb{R}, \Theta_f^i : \mathcal{P} \rightarrow \mathbb{R}$ and $\Theta_l^i  : \mathcal{P} \rightarrow \mathbb{R}$.

Correspondingly the linear system \eqref{chap_1_1:linalg_dscr_weak} can be expanded as

\begin{equation}
 \left( \sum_{i=1}^{Q_a} \Theta_a^i(\boldsymbol\mu) \mathbb{A}_i \right) \textbf{u}_h = \sum_{i=1}^{Q_f} \Theta_f^i(\boldsymbol\mu) \textbf{f}_i ,
\label{chap_1_1:linalg_dscr_weak_affine_expanded}
\end{equation}

\noindent as well as the reduced order form \eqref{chap_1_1:linalg_rom_weak} 

\begin{align}
 \mathbb{V}^T \left( \sum_{i=1}^{Q_a} \Theta_a^i(\boldsymbol\mu) \mathbb{A}_i \right) \mathbb{V} \textbf{u}_N &= \mathbb{V}^T \sum_{i=1}^{Q_f} \Theta_f^i(\boldsymbol\mu) \textbf{f}_i , \\
  \left( \sum_{i=1}^{Q_a} \Theta_a^i(\boldsymbol\mu) \mathbb{V}^T \mathbb{A}_i \mathbb{V} \right)  \textbf{u}_N &=  \sum_{i=1}^{Q_f} \Theta_f^i(\boldsymbol\mu) \mathbb{V}^T \textbf{f}_i .
\label{chap_1_1:linalg_rom_weak_affine_expanded}
\end{align}

MOR relies on an affine parameter dependency, such that all computations depending on the high-order model size can be moved into
a parameter-independent offline phase, while having a fast input-output evaluation online.
If the problem is not affine, an affine representation can be approximated using a technique such as the \emph{empirical interpolation method}, see section \ref{sec:nonaffinity}.

\subsubsection{Affine shape parametrizations: an example}
\label{sec:affinepar}
Consider heat conduction in a square domain $\Omega(x,y) = [ 0, 1 ]^2$.
On the left side $x = 0$, inhomogeneous Neumann conditions, i.e., a non-zero heat flux is imposed and 
on the right side $x = 1$, homogeneous Dirichlet conditions, i.e., zero temperature is imposed.
The top and bottom side, homogeneous Neumann conditions, i.e., a zero heat flux is imposed. 
Consider two different media with different conductivities $\sigma_1$ and $\sigma_2$ 
occupying the subdomains $\Omega_1(\mu) = [0, \mu] \times [0, 1] $ and $\Omega_2(\mu) = [\mu, 1] \times [0, 1]$, for $\mu \in \mathcal{P} = (0, 1)$, 
as shown in Figure~\ref{chap_1_1:fig:example_affine_domains}. 
For the sake of clarity, in the rest of this section we identify the $1$-dimensional parameter vector $\boldsymbol\mu$ with its (only) component 
$\mu$, thus dropping the bold notation from the symbol.

\begin{figure}[h!]
\centering
\includegraphics[width=0.7\textwidth]{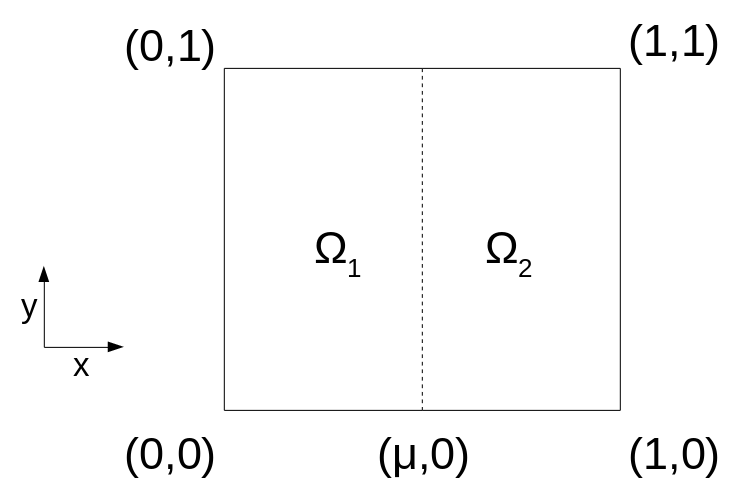}
\caption{The computational domain is subdivided into two domains $\Omega = \Omega_1 \cup \Omega_2$, depending on the parameter $\mu$. Shown here for $\mu = 0.5$. }
\label{chap_1_1:fig:example_affine_domains}
\end{figure}

Choosing $\overline{\mu} = 0.5$ as the reference configuration, there exist affine transformations from the reference domain to the actual domain.
It holds 

\begin{eqnarray}
 T_1 : \Omega_1(\overline{\mu}) \rightarrow \Omega_1(\mu) &:& (\overline{x},\overline{y}) \mapsto (2 \mu \overline{x}, \overline{y}) , \\
 T_2 : \Omega_2(\overline{\mu}) \rightarrow \Omega_2(\mu) &:& (\overline{x},\overline{y}) \mapsto ( (2 - 2 \mu) \overline{x}, \overline{y}) + (2\mu -1, 0) .
\end{eqnarray}

In general, an affine transformation of a subdomain can be expressed as 

\begin{equation}
 T_k : \Omega_k(\overline{\mu}) \rightarrow \Omega_k(\mu) : \mathbf{x} \mapsto G_k(\mu) \mathbf{x} + D_k(\mu),
\end{equation}

\noindent with $\mathbf{x} \in \mathbb{R}^d$, $G_k \in \mathbb{R}^{d \times d}$ and $D_k \in \mathbb{R}^d$ in $d = 2,3$ spatial dimensions.

Thus, the bilinear form 

\begin{equation}
 a(u,v;\mu) = \int_{\Omega_1(\mu)} \sigma_1 \nabla u \cdot \nabla v d\mathbf{x} + \int_{\Omega_2(\mu)} \sigma_2 \nabla u \cdot \nabla v \, d\mathbf{x}
\end{equation}

\noindent can be mapped to the reference domain with the inverse affine transformation

\begin{equation}
 T_k^{-1} : \Omega_k(\mu) \rightarrow \Omega_k(\overline{\mu}) : \mathbf{x} \mapsto G_k^{-1}(\mu) \mathbf{x} -  G_k^{-1}(\mu) D_k(\mu),
\end{equation}

\noindent and integration by substitution as

\begin{align}
 a(u,v;\mu) =& \int_{\Omega_1(\overline{\mu})} \sigma_1  (\nabla u G_1^{-1}(\mu)) \cdot (G_1^{-T}(\mu) \nabla v ) \det(G_1(\mu)) \, d\mathbf{x} \\
&+ \int_{\Omega_2(\overline{\mu})} \sigma_2 (\nabla u G_2^{-1}(\mu))  \cdot  (G_2^{-T}(\mu) \nabla v) \det(G_2(\mu)) \, d\mathbf{x} ,
\end{align}

\noindent which establishes the affine parameter dependency \eqref{chap_1_1:affine_param_dep_bilifo} by computing 
 $\Theta_a^i(\mu)$ from the coefficients of $G_1$ and $G_2$~\cite{Patera:2008, RozzaHuynh2009, ChinestaHuerta2017}.
It is:

\begin{align}
 &\int_{\Omega_1(\overline{\mu})} \sigma_1  (\nabla u G_1^{-1}(\mu)) \cdot (G_1^{-T}(\mu) \nabla v ) \det(G_1(\mu)) \, d\mathbf{x} \\
  &= \int_{\Omega_1(\overline{\mu})} \sigma_1  ((2\mu)^{-1} \partial_x u, \partial_y u) \cdot ((2\mu)^{-1} \partial_x v, \partial_y v) 2\mu \, d\mathbf{x} \\
  &= (2\mu)^{-1} \int_{\Omega_1(\overline{\mu})} \sigma_1  (\partial_x u) (\partial_x v) \, d\mathbf{x} + 2\mu \int_{\Omega_1(\overline{\mu})} \sigma_1  (\partial_y u) (\partial_y v) \, d\mathbf{x} ,
\end{align}

\noindent and

\begin{align}
 &\int_{\Omega_2(\overline{\mu})} \sigma_2  (\nabla u G_2^{-1}(\mu)) \cdot (G_2^{-T}(\mu) \nabla v ) \det(G_2(\mu)) \, d\mathbf{x} \\
  &= \int_{\Omega_2(\overline{\mu})} \sigma_2  ((2-2\mu)^{-1} \partial_x u, \partial_y u) \cdot ((2-2\mu)^{-1} \partial_x v, \partial_y v) (2-2\mu) \, d\mathbf{x} \\
  &= (2 - 2\mu)^{-1} \int_{\Omega_2(\overline{\mu})} \sigma_2  (\partial_x u) (\partial_x v) \, d\mathbf{x} + (2-2\mu) \int_{\Omega_2(\overline{\mu})} \sigma_2  (\partial_y u) (\partial_y v) \, d\mathbf{x} ,
\end{align}

\noindent which establishes the affine form \eqref{chap_1_1:affine_param_dep_bilifo} with $Q_a = 4$ and

\begin{align}
 \Theta_a^1(\mu) &= (2\mu)^{-1}, \\
 \Theta_a^2(\mu) &= 2\mu, \\
 \Theta_a^3(\mu) &= (2-2\mu)^{-1}, \\
 \Theta_a^4(\mu) &= 2-2\mu,
\end{align}

\noindent and

\begin{align}
  a_1(\cdot,\cdot) &= \int_{\Omega_1(\overline{\mu})} \sigma_1  (\partial_x u) (\partial_x v) \, d\mathbf{x}, \\
  a_2(\cdot,\cdot) &= \int_{\Omega_1(\overline{\mu})} \sigma_1  (\partial_y u) (\partial_y v) \, d\mathbf{x}, \\
  a_3(\cdot,\cdot) &= \int_{\Omega_2(\overline{\mu})} \sigma_2  (\partial_x u) (\partial_x v) \, d\mathbf{x}, \\
  a_4(\cdot,\cdot) &= \int_{\Omega_2(\overline{\mu})} \sigma_2  (\partial_y u) (\partial_y v) \, d\mathbf{x}.
\end{align}



The second and fourth term can be further simplified to a term depending on $2\mu$ and a $\mu$-independent term, but in this case it still leaves $Q_a = 4$ terms.
In some cases the number of affine terms can be automatically reduced further using symbolic computations.

\subsubsection{Offline-Online Decomposition}
\label{chap_1_1:subsubsec_offon_decomp}

The offline-online decomposition enables the computational speed-up of the ROM approach in many-query scenarios.
It is also known as the offline-online paradigm, which assumes that a compute-intensive offline phase can be performed on a 
supercomputer, which generates all quantities depending on the large discretization size $N_h$.
Once completed, a reduced order solve, i.e., an online solve for a new parameter of interest can be performed with computational 
cost independent of the large discretization size $N_h$. The online phase can thus be performed even on mobile and embedded devices, see Figure~\ref{fig:off_online}.
If a supercomputer is not available, this can be relaxed however. There exist heuristic algorithms to make also the offline phase feasible 
on a common workstation, such that 
a typical scenario would be that the offline phase runs overnight and a reduced model is available the next morning.

\begin{figure}[h!]
\centering
\includegraphics[width=0.75\textwidth]{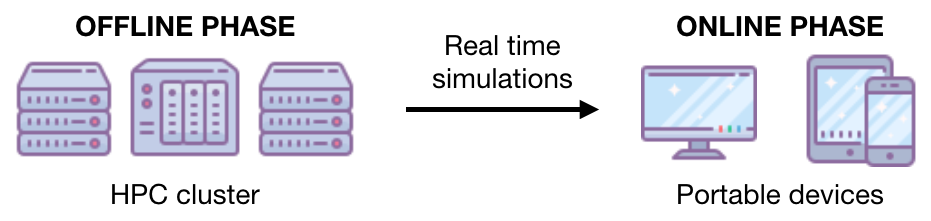}
\caption{Offline-online paradigm. The complex high fidelity simulations are
  carried out in high performance (HPC) clusters for given preselected
parameters. The solution snapshots can be stored and the ROM
trained. Then in the offline phase the ROM provides approximated
solutions at new untried parameters in real time on simple portable devices.}
\label{fig:off_online}
\end{figure}

Noticing that the terms $\mathbb{V}^T \mathbb{A}_i \mathbb{V}$ and $\mathbb{V}^T \textbf{f}_i$ in \eqref{chap_1_1:linalg_rom_weak_affine_expanded}
are parameter-independent, they can be precomputed, prior to any ROM
parameter sweep. This will store small-sized dense matrices of dimension $N \times N$. 
Once a reduced order solution $\textbf{u}_N$ is desired for a given parameter $\boldsymbol\mu$, the sum given in \eqref{chap_1_1:linalg_rom_weak_affine_expanded}
is formed and solved for $\textbf{u}_N$.
Since this is the same as solving \eqref{chap_1_1:linalg_rom_weak}, the reduced order approximation is then available as $\textbf{u}_h \approx \mathbb{V} \textbf{u}_N$, 
see \eqref{chap_1_1:ROMu_approx_FOMu}.

\subsection{Error bounds}
\label{subsec:error_bounds}
In this section we develop effective and reliable a posteriori error
estimators for the field variable or an output of interest. The use of
such error bounds drives the construction of the reduced basis during
the offline stage, thanks to the so-called greedy algorithm. Moreover,
during the online stage, such bounds provide a certified accuracy of
the proposed reduced order model. 

Following~\cite{Patera:2008}, we introduce residual-based a posteriori error estimation for the elliptic case. From \eqref{chap_1_1:discr_weak} and \eqref{chap_1_1:ROM_weak} it follows that the error $e({\boldsymbol\mu}) = u_h({\boldsymbol\mu}) - u_N({\boldsymbol\mu})$ satisfies
\begin{equation}
 a(e(\boldsymbol\mu),v_h;\boldsymbol\mu) = r(v_h;\boldsymbol\mu) \quad \forall v_h \in V_h.
\label{chap_1_1:ROM_error}
\end{equation}
where the residual $r(\cdot;\boldsymbol\mu) \in V_h'$ is defined as
\begin{equation}
r(v_h;\boldsymbol\mu) = f(v_h;\boldsymbol\mu) - a(u_N({\boldsymbol\mu}), v_h;\boldsymbol\mu)  \quad \forall v_h \in V_h.
\label{chap_1_1:ROM_residual}
\end{equation}

The following theorem further characterizes the relation between error and residual:
\begin{thm}
Under compliance assumptions, the following inequalities hold
\begin{align}
&\left\|e(\boldsymbol\mu)\right\|_{\boldsymbol\mu} = \left\|u_h({\boldsymbol\mu}) - u_N({\boldsymbol\mu})\right\|_{\boldsymbol\mu} \leq \Delta_{en}(\boldsymbol\mu) = \frac{\left\|r(\cdot; \boldsymbol\mu)\right\|_{V_h'}}{\sqrt{\alpha_h(\boldsymbol\mu)}},\label{aposteriori_en}\\
&0 \leq s_h({\boldsymbol\mu}) - s_N({\boldsymbol\mu}) \leq \Delta_{s}(\boldsymbol\mu) = \frac{\left\|r(\cdot; \boldsymbol\mu)\right\|_{V_h'}^2}{\alpha_h(\boldsymbol\mu)}\label{aposteriori_s},
\end{align}
where $\left\|v\right\|_{\boldsymbol\mu}^2 = a(v, v; \boldsymbol\mu)$ defines an equivalent norm to $\left\|v\right\|_{V_h}$.
\end{thm}
\begin{proof}
$\left\|\cdot\right\|_{\boldsymbol\mu}$ defines an equivalent norm thanks to symmetry, continuity and coercivity of $a(\cdot, \cdot; \boldsymbol\mu)$.

Since $e(\boldsymbol\mu) \in V_h$, from \eqref{chap_1_1:ROM_error} with $v_h = e(\boldsymbol\mu)$ it follows
\begin{equation*}
\left\|e(\boldsymbol\mu)\right\|_{\boldsymbol\mu}^2 = a(e(\boldsymbol\mu), e(\boldsymbol\mu); \boldsymbol\mu) = r(e(\boldsymbol\mu);\boldsymbol\mu) \leq \left\|r(\cdot; \boldsymbol\mu)\right\|_{V_h'} \left\|e(\boldsymbol\mu)\right\|_{V_h},
\end{equation*}
the last inequality being due to the definition of the norm in $V_h'$. Furthermore, due to coercivity, it holds
\begin{equation*}
\left\|e(\boldsymbol\mu)\right\|_{\boldsymbol\mu}^2 = a(e(\boldsymbol\mu), e(\boldsymbol\mu); \boldsymbol\mu) \geq \alpha(\boldsymbol\mu) \left\|e(\boldsymbol\mu)\right\|_{V_h}^2.
\end{equation*}
Combining these two results yields \eqref{aposteriori_en}.

Furthermore, since $l = f$ are linear forms,
\begin{equation}
s_h({\boldsymbol\mu}) - s_N({\boldsymbol\mu}) = l(e({\boldsymbol\mu}); \boldsymbol\mu) = f(e({\boldsymbol\mu}); \boldsymbol\mu) = a(u_h({\boldsymbol\mu}), e({\boldsymbol\mu}); \boldsymbol\mu)
\label{proof1}
\end{equation}
From \eqref{chap_1_1:ROM_error} with $v_h := v_N \in V_N$ and \eqref{chap_1_1:ROM_weak} it follows
\begin{equation*}
 a(e(\boldsymbol\mu),v_N; \boldsymbol\mu) = r(v_N({\boldsymbol\mu});\boldsymbol\mu) = 0.
\end{equation*}
This holds in particular for $v_N = u_N(\boldsymbol\mu)$. Moreover, due to symmetry,
\begin{equation*}
 a(u_N(\boldsymbol\mu), e(\boldsymbol\mu); \boldsymbol\mu) = 0
\end{equation*}
as well. Thus, $a(u_h({\boldsymbol\mu}), e({\boldsymbol\mu}); \boldsymbol\mu) = a(e({\boldsymbol\mu}), e({\boldsymbol\mu}); \boldsymbol\mu)$ in \eqref{proof1}, and we conclude that 
\begin{equation}
s_h({\boldsymbol\mu}) - s_N({\boldsymbol\mu}) = \left\|e(\boldsymbol\mu)\right\|_{\boldsymbol\mu}^2.
\label{proof2}
\end{equation}
The upper bound in \eqref{aposteriori_s} is then a consequence of \eqref{aposteriori_en}, while the lower bound trivially holds as the right-hand side of \eqref{proof2} is a non-negative quantity.
\end{proof}

Offline-online decomposition is usually solicited for the a posteriori error bounds introduced by the previous theorem, for the sake of a fast computation of the right-hand side of \eqref{aposteriori_en}-\eqref{aposteriori_s}. This requires the efficient evaluation of both the numerator (dual norm of the residual) and the denominator (parametrized coercivity constant). The Riesz representation theorem is employed to define the unique $\hat{r}(\boldsymbol\mu) \in V_h$ such that
\begin{equation}
(\hat{r}(\boldsymbol\mu), v_h)_{V_h} = r(v_h; \boldsymbol\mu), \quad \forall v_h \in V_h.
\label{eq:riesz}
\end{equation}
Under affine separability assumptions \eqref{chap_1_1:affine_param_dep_bilifo}-\eqref{chap_1_1:affine_param_dep_bilifo_l}, it holds
\begin{equation*}
r(v_h; \boldsymbol\mu) = \sum_{i=1}^{Q_f} \Theta_f^i(\boldsymbol\mu) f_i(v_h) - \sum_{n=1}^N {\mathbf{u}_N}_n \sum_{i=1}^{Q_a} \Theta_a^i(\boldsymbol\mu) a_i(\zeta^n, v_h), \quad \forall v_h \in V_h;
\end{equation*}
so that an affine expansion with $Q_f + N Q_a$ terms is obtained for $r(\cdot; \boldsymbol\mu)$. Riesz representation is then invoked for
\begin{align*}
&r_1(v_h; \boldsymbol\mu) = f_1(v_h),&\hdots,&\quad r_{Q_f}(v_h; \boldsymbol\mu) = f_{Q_f}(v_h),\\
&r_{Q_f + 1}(v_h; \boldsymbol\mu) = a_1(\zeta^1, v_h),&\hdots,&\quad r_{Q_f + Q_a}(v_h; \boldsymbol\mu) = a_{Q_a}(\zeta^1, v_h),\\
&\hdots\\
&r_{Q_f + (N-1) Q_a + 1}(v_h; \boldsymbol\mu) = a_1(\zeta^N, v_h),&\hdots,&\quad r_{Q_f + N Q_a}(v_h; \boldsymbol\mu) = a_{Q_a}(\zeta^{N}, v_h)
\end{align*}
during the offline stage, storing the corresponding solutions to \eqref{eq:riesz}.

For what concerns the evaluation of the denominator of \eqref{aposteriori_en}-\eqref{aposteriori_s}, exact evaluation of $\alpha(\boldsymbol\mu)$ is seldom employed. Instead, an offline-online decomposable lower bound is sought. Early proposals on the topic are available in~\cite{Veroy2002,PrudHomme2002,Veroy2005,Patera:2008,Canuto2009}.
In 2007, the \emph{successive constraint method} (SCM) was devised in~\cite{SCM_original} based on successive linear programming approximations, and subsequently extended in~\cite{Chen2,Chen1,Vallaghe2011,Zhang2011}. Alternative methodologies based on interpolation techniques have also appeared in recent years in~\cite{Hess:2015,Manzoni2015,Iapichino2017}.

A posteriori error estimation can be derived for more general problems as well (including non-coercive linear, nonlinear or time-dependent problems), through application of the Brezzi-Rappaz-Raviart theory. We refer to~\cite{Veroy2005,Deparis2009,Yano2014,manzoni_2014,Enrique2017} for a few representative cases.
To this end, extensions of SCM are discussed in~\cite{Chen1,Hesthaven2012,Huynh2010,Chen2016}.

\section{Geometrical parametrization for shapes and domains}
\label{sec:geom}

In this section we discuss problems characterized by a geometrical
parametrization. In particular, a reference domain approach is
discussed, relying on a map that deforms the reference domain into the
parametrized one. Indeed, while affine shape parametrization (see section \ref{sec:affinepar} for
an example, and \cite{Patera:2008} for more details) naturally abides by the offline-online separability assumption,
it often results in very limited deformation of the reference domain, or strong assumptions on the underlying shape.

Let $\Omega \subset \mathbb{R}^d$, $d=2,3$, be the reference domain.
Let $\mathcal{M}$ be a parametric shape morphing function, that is
\begin{equation}
  \label{eq:general_morphing}
  \mathcal{M}(\boldsymbol{x}; \mupar): \mathbb{R}^d \to \mathbb{R}^d
\end{equation}
which maps the reference domain
$\Omega$ into the deformed domain $\Omega(\mupar)$ as
$\Omega(\boldsymbol{\mu}) = \mathcal{M}(\Omega; \mupar)$, 
where $\mupar \in \mathcal{P}$ represents the vector of the geometrical
parameters. This map will change accordingly to the chosen shape
morphing technique. The case of Section \ref{sec:affinepar} is representative of an affine map $\mathcal{M}(\cdot; \mupar)$.
Instead, in the following we address more general (not necessarily affine)
techniques such as the free form
deformation (FFD), the radial basis functions (RBF) interpolation, and
the inverse distance weighting (IDW) interpolation. 

From a practical point of view, we recommend the Python package called
PyGeM - Python Geometrical Morphing (see~\cite{pygem}), which allows an easy
integration with the majority of industrial CAD files and the most
common mesh files. 

\subsection{Free form deformation}
\label{sec:ffd}

The free form deformation is a widely used parametrization and
morphing technique both in academia and in industry. 

For the original formulation see~\cite{sederbergparry1986}. More
recent works use FFD coupled with reduced basis methods for shape
optimization and design of systems modeled by elliptic PDEs
(see~\cite{LassilaRozza2010},~\cite{rozza2013free},
and~\cite{sieger2015shape}), in naval engineering for the optimization
of the bulbous bow shape of cruise ships (in~\cite{demo2018shape}), in
the context of sailing boats in~\cite{lombardi2012numerical},
and in automotive engineering in~\cite{salmoiraghi2018}.

FFD can be used both for global and local deformations and it is
completely independent to the geometry to morph. It acts through the
displacement of a lattice of points, called FFD control points,
constructed around the domain of interest. In particular it consists
in three different steps as depicted in
Figure~\ref{fig:ffd_scheme}. First the physical domain $\Omega$ is
mapped to $\hat{\Omega}$, the
reference one, through the affine map $\boldsymbol{\psi}$. Then the lattice of control
points is constructed and the displacements of these points by the map
$\hat{T}$, is what we call geometrical parameters $\mupar$. The
deformation is propagated to the entire embedded body usually by using
Bernstein polynomials. Finally through the inverse map $\boldsymbol{\psi}^{-1}$ we
return back to the parametric physical space $\Omega (\mupar)$. 

\begin{figure}[h!]
\centering
\includegraphics[width=0.8\textwidth]{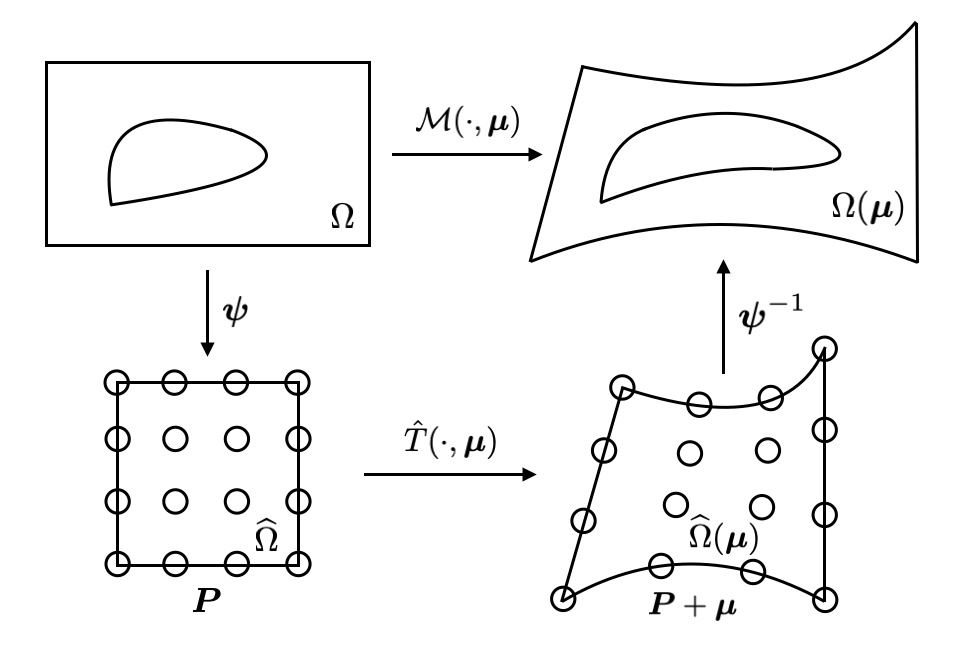}
\caption{Scheme of the three maps composing the FFD map
  $\mathcal{M}$. In particular $\psi$ maps the physical space to the reference one,
  then $\hat{T}$ deforms the entire geometry according to the
  displacements of the lattice control points, and finally $\psi^{-1}$
maps back the reference domain to the physical one.}
\label{fig:ffd_scheme}
\end{figure}

So, recalling Eq.~\eqref{eq:general_morphing}, we have the explicit
map $\mathcal{M}$ for the FFD, that is the composition of the three
maps presented, i.e.
\begin{align}
\mathcal{M} (\x, \mupar) &= (\boldsymbol{\psi}^{-1} \circ \widehat{T}
\circ \boldsymbol{\psi}) (\x, \mupar) =\\
&= \boldsymbol{\psi}^{-1} \left( \sum_{l=0} ^L \sum_{m=0}
  ^M \sum_{n=0} ^N b_{lmn}(\boldsymbol{\psi}(\x))
  \boldsymbol{P}_{lmn}^0 \left(\mupar_{lmn}\right) \right)
\quad \forall \x \in \Omega, 
\end{align}
where $b_{lmn}$ are Bernstein polynomials of degree $l$, $m$,
$n$ in each direction, respectively, $\boldsymbol{P}_{lmn}^0
\left(\mupar_{lmn}\right) =
\boldsymbol{P}_{lmn} + \mupar_{lmn}$, with
$\boldsymbol{P}_{lmn}$ representing the coordinates of the control
point identified by the three indices $l$, $m$, $n$ in the lattice of
FFD control points. In an offline-online fashion, for a given $\x$,
terms $\{b_{lmn}(\boldsymbol{\psi}(\x))\}_{l,m,n}$ can be precomputed during
the offline stage, resulting in an inexpensive linear combination of
$\x$-dependent precomputed quantities and $\mupar$-dependent control points
locations $\{\boldsymbol{P}_{lmn}^0 \left(\mupar_{lmn}\right)\}_{l,m,n}$.
The application of $\boldsymbol{\psi}^{-1}$ does not hinder such offline-online
approach as $\boldsymbol{\psi}$ is affine.

We can notice that the deformation does not depend on the topology of
the object to be morphed, so this technique is very versatile and
non-intrusive, especially for complex geometries or in industrial
contexts (see e.g. \cite{salmoiraghi2016advances,rozza2018advances}).

In the case where the deformation has to satisfy some constraints,
like for example continuity constraints, it is possible to increase
the number of control points. Often it is the case where at the
interface between the undeformed portion of the geometry and the
morphed area the continuity has to be prescribed for physical
reasons. 

As an example, in Figure~\ref{fig:bow_ffd} we present a free form
deformation of a bulbous bow, where an STL file of a complete hull is
morphed continuously by the displacement of only some control
points.

\begin{figure}[h!]
\centering
\includegraphics[trim=200 0 200 0, width=0.7\textwidth]{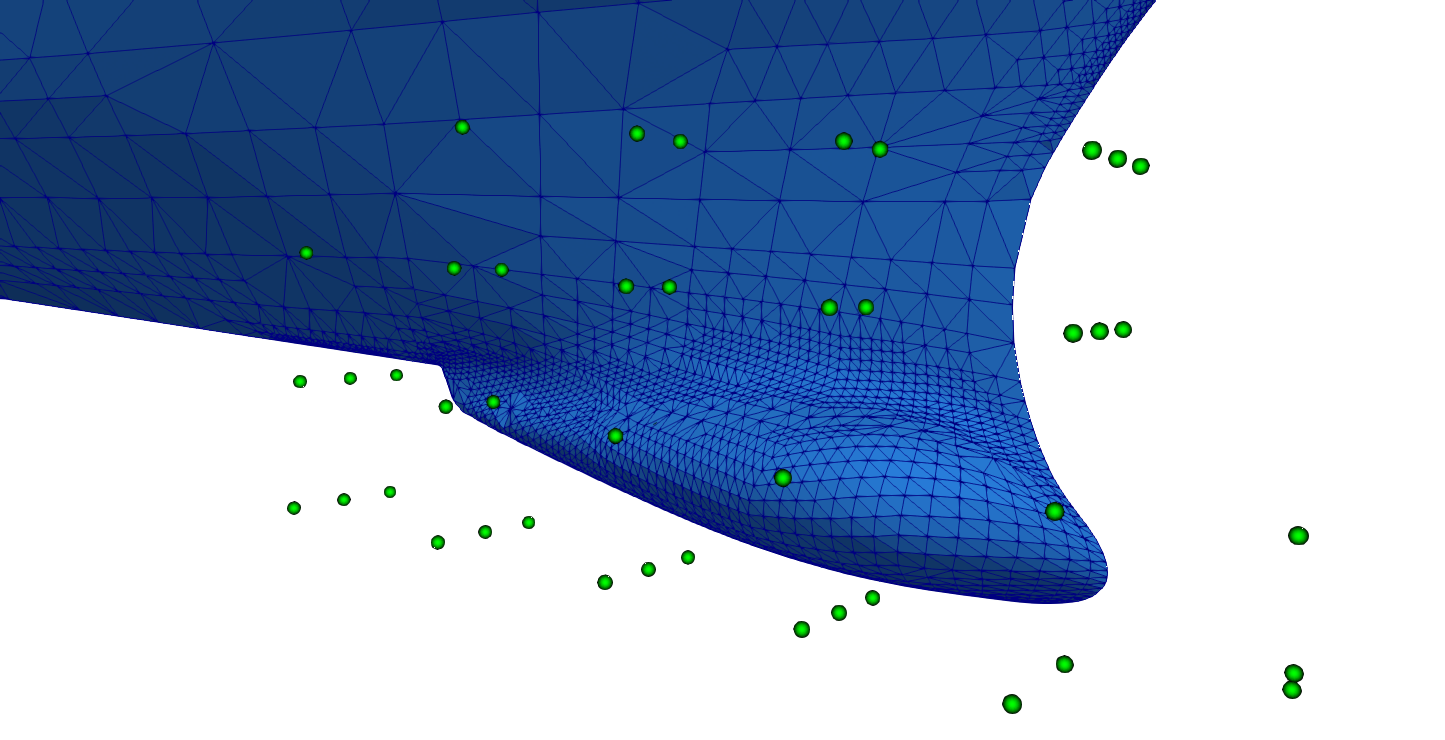}
\caption{Bulbous bow deformation using FFD. In green the FFD control
  points defining the morphing.}
\label{fig:bow_ffd}
\end{figure}

\subsection{Radial basis functions interpolation}
\label{sec:rbf}

Radial Basis Functions (RBF) represent a powerful tool for
nonlinear multivariate approximation, interpolation between
nonconforming meshes~(\cite{deparis2014rescaled}), and for shape
parametrization due to its approximation properties
(see~\cite{buhmann2003radial}).

A radial basis function is any smooth real-valued function
$\widetilde{\varphi}: \mathbb{R}^d \to \mathbb{R}$ 
such that it exists $\varphi: \mathbb{R}^+ \to \mathbb{R}$ and
$\widetilde{\varphi} (\x) = \varphi (\| \x \|)$, where $\| \cdot \|$
indicates the Euclidean norm in $\mathbb{R}^d$. The 
most widespread radial basis functions are the following:\\
\begin{itemize}
\item Gaussian splines~(\cite{buhmann2003radial}) defined as
  \[
     \varphi(\| \x \|) = e^{-\| \x \|^2/R};
  \]
\item thin plate splines~(\cite{duchon1977splines}) defined as
  \[
    \varphi(\| \x \|) = \left( \frac{\| \x \|}{R} \right)^2
    \ln \left( \frac{\| \x \|}{R} \right);
  \]
\item Beckert and Wendland $C^2$
  basis~(\cite{beckert2001multivariate}) defined as
  \[
    \varphi(\| \x \|) = \left ( 1 - \frac{\| \x \|}{R} \right)^4_+ \left ( 4
      \frac{\| \x \|}{R} + 1 \right);
  \]
\item multi-quadratic biharmonic
  splines~(\cite{sandwell1987biharmonic}) defined as 
  \[
    \varphi(\| \x \|) = \sqrt{\| \x \|^2 + R^2};
  \]
\item inverted multi-quadratic biharmonic
  splines~(\cite{buhmann2003radial}) defined as 
  \[
    \varphi(\| \x \|) = \frac{1}{\sqrt{\| \x \|^2 + R^2}};
  \]
\end{itemize}
where $R > 0$ is a given radius, and the subscript $_+$ indicates the
positive part.

Following \cite{morris2008cfd,manzoni2012model}, given $\mathcal{N}_C$ control points situated on the surface of the
body to morph, we can generate a deformation by moving some of these
points and imposing the new surface which interpolates them. The
displacements of the control points represent the geometrical
parameters $\mupar$. 

\begin{figure}[h!]
\centering
\includegraphics[trim=400 0 450 0, width=0.45\textwidth]{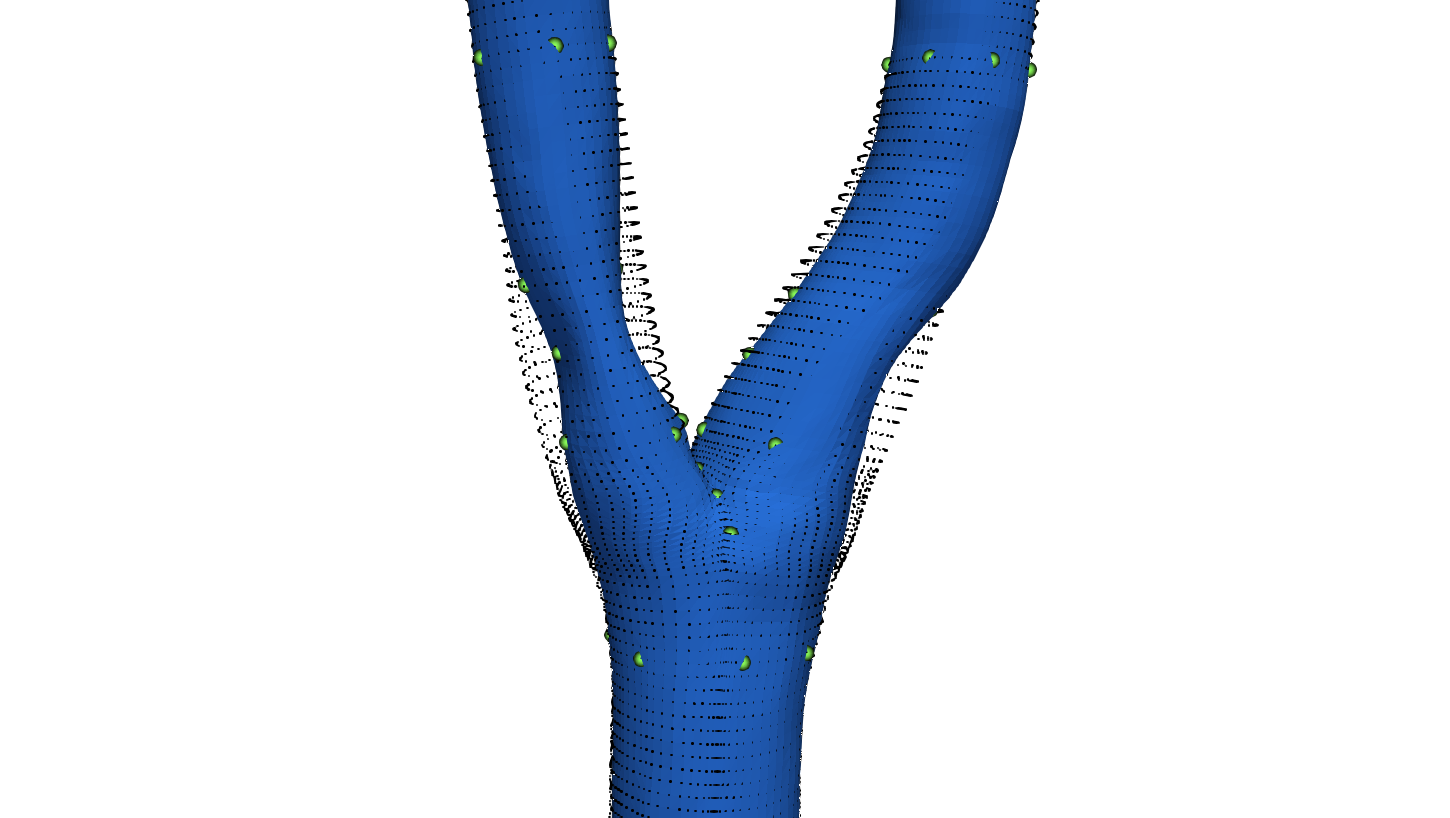}\hfill
\includegraphics[trim=450 0 400 0, width=0.45\textwidth]{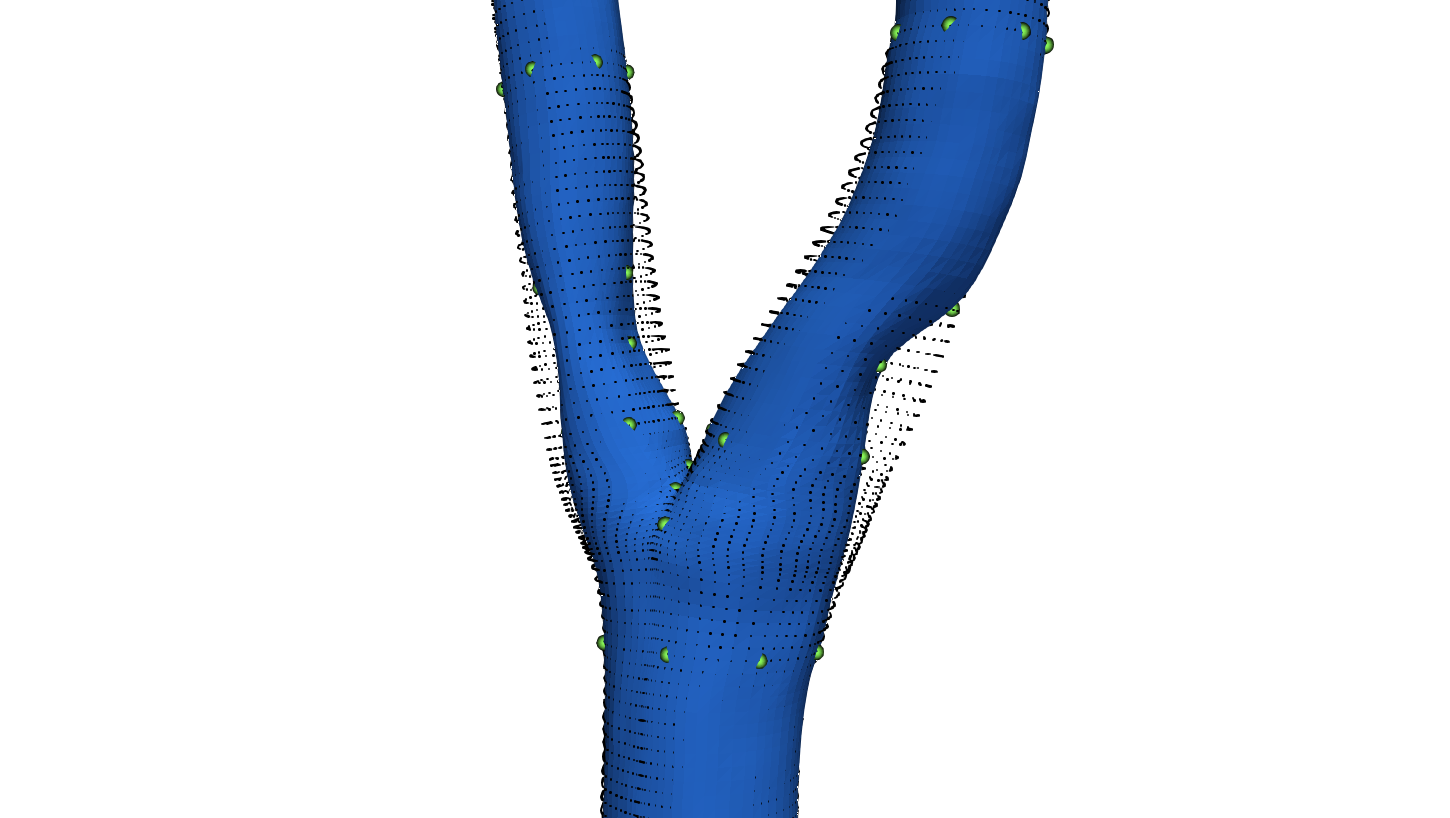}
\caption{Two different views of the same deformed carotid artery model
  using the radial basis functions (RBF) interpolation technique. The green
  dots indicate the RBF control points that define the morphing. The
  black small points highlight the original undeformed geometry. The
  occlusion of the two branches is achieved through a displacement
  along the normal direction with respect to the carotid surface of
  the control points after the bifurcation.}
\label{fig:carotid_def}
\end{figure}

We can now define the map $\mathcal{M}$ in
Eq.~\eqref{eq:general_morphing} for the RBF interpolation technique,
that is 
\begin{equation}
\label{eq:rbf_map}
	\mathcal{M}(\x; \mupar) = q(\x; \mupar) + 
	\sum_{i=1}^{\mathcal{N}_C} \gamma_i(\mupar)\;
	\varphi(\| \x - \x_{C_i} \|),
\end{equation}
where $q(\x; \mupar)$ is a polynomial term, generally of degree 1,
$\gamma_i(\mupar)$ is the weight associated to the basis function
$\varphi_i$, and $\{\x_{C_i}\}_{i=1}^{\mathcal{N}_C}$ are control
points selected by the user (denoted by spherical green markers in
Figure \ref{fig:carotid_def}), and $\x \in \Omega$. We underline
that in the three dimensional case~\eqref{eq:rbf_map} has $d \times
\mathcal{N}_C + d + d^2$ unknowns, which are $d \times \mathcal{N}_C$ for
the $\gamma_i$ and $d + d^2$ for the polynomial term $q(\x; \mupar) =
c(\mupar) + \mathbf{Q} (\mupar) \x$. To this end we impose the
interpolatory constraint
\begin{equation}
\mathcal{M}(\x_{C_i}; \mupar) = \y_{C_i}(\mupar) \qquad \forall i \in
\{1, \dots, \mathcal{N}_C\},
\end{equation}
where $\y_{C_i}$ are the deformed control points obtained applying the
displacement $\mupar$ to $\x_{C_i}$, in particular
\begin{align}
\x_{C} &= [\x_{C_1}, \dots,
\x_{C_{\mathcal{N}_C}}] \in \mathbb{R}^{\mathcal{N}_C \times d},\\
\y_{C}(\mupar) &= [\y_{C_1}(\mupar), \dots,
\y_{C_{\mathcal{N}_C}}(\mupar)] \in \mathbb{R}^{\mathcal{N}_C \times
  d}.
\end{align}
For the remaining $d + d^2$ unknowns, due to the presence of the polynomial
term, we complete the system with additional constraints that
represent the conservation of the total force and
momentum (see~\cite{buhmann2003radial,morris2008cfd}), as follows 
\begin{equation}
\sum_{i=1}^{\mathcal{N}_C} \gamma_i(\mupar) = 0,
\end{equation}
\begin{equation}
\sum_{i=1}^{\mathcal{N}_C} \gamma_i(\mupar) [\x_{C_i}]_1 = 0, \hdots
\sum_{i=1}^{\mathcal{N}_C} \gamma_i(\mupar) [\x_{C_i}]_d = 0,
\end{equation}
where the notation $[\x]_d$ denotes the $d$-th component of the vector $\x$.

Following an offline-online strategy, for a given $\x$, evaluation of $\varphi(\| \x - \x_{C_i} \|)$, $i = 1, \hdots, \mathcal{N}_C$
can be precomputed in the offline stage. Further online effort is only required for (i) given $\mupar$, solve a $d \times
\mathcal{N}_C + d + d^2$ linear system, and, (ii) given $\mupar$ and $\x$, perform linear combinations and the matrix vector product
in \eqref{eq:rbf_map} employing either precomputed quantities or coefficients from (i).

\subsection{Inverse distance weighting interpolation}
\label{sec:idw}

The inverse distance weighting (IDW) method has been proposed in \cite{shepard1968} to deal with 
interpolation of scattered data. We follow \cite{witteveenbijl2009,forti2014efficient,BallarinDAmarioPerottoRozza2017}
for its presentation and the application of IDW to shape parametrization.

As in the previous section, let $\{\x_{C_k}\}_{k=1}^{\mathcal{N}_c} \subset \mathbb{R}^d$ be a set of control points.
The IDW interpolant $\Pi_{\rm IDW}(f)$ of a scalar function $f: \mathbb{R}^d \to \mathbb{R}$ is defined as
\begin{equation}\label{IDW_formula}
\Pi_{\rm IDW}(f)(\x) =
\sum _{k=1}^{\mathcal{N}_c}  w_k(\x) \, f(\x_{C_k})\quad \x\in \Omega,
\end{equation}
where the weight functions $w_k: \Omega \to \mathbb{R}$, for $k=1, \ldots, {\mathcal{N}_c}$ are given by
\begin{equation}\label{weighting_function}
w_k(\x) = \begin{cases}
\frac{\displaystyle\| \x - \x_{C_k} \|^{-s}}{\displaystyle\sum_{j=1}^{\mathcal{N}_c}{\| \x - \x_{C_j} \|^{-s}} } & \text{ if } \x \neq \x_{C_k},\\[2mm]
1 & \text{ if } \x = \x_{C_k},\\
0 & \text{ otherwise}.
\end{cases}
\end{equation}
$s$ is a positive integer, modelling the assumption
that the influence of the $k$-th control point 
$\x_{C_k}$ on $\x$ diminishes with rate $-s$
as the distance between $\x$ and $\x_{C_k}$ increases.
IDW interpolation trivially extends to vector functions $\boldsymbol{f}: \mathbb{R}^d \to \mathbb{R}^d$ by application to each component $f_1, \hdots, f_d$, where the weight functions $w_k: \Omega \to \mathbb{R}$ do not depend on the specific component.

In the case of IDW shape parametrization, for any given $\mupar$, the deformed position of the control points $\{\x_{C_k}\}_{k=1}^{\mathcal{N}_c}$ is supposed to be known, and equal to $\y_{C_k}(\boldsymbol\mu) := \boldsymbol{f}(\x_{C_k})$ for $k=1, \hdots, \mathcal{N}_c$. We remark that the analytic expression of $\boldsymbol{f}$ is not known, but only its action through $\{\x_{C_k}\}_{k=1}^{\mathcal{N}_c}$. This is indeed the minimum requirement to properly define \eqref{IDW_formula}. The deformation map is therefore
\begin{equation*}
	\mathcal{M}(\x; \mupar) = \sum _{k=1}^{\mathcal{N}_c}  w_k(\x) \, \y_{C_k}(\mupar) \quad\forall \x\in \Omega,
\end{equation*}
In an offline-online separation effort, efficient deformation can be obtained by noting that the $\mupar$ dependent part is decoupled from the
$\x$-dependent weight function $w_k(\x)$. Thus, for any $\x$, weight terms can be precomputed once and for all and stored. The online cost
of the evaluation of $\mathcal{M}(\x; \mupar)$ thus requires an inexpensive linear combination of $\x$-dependent precomputed quantities
and $\mupar$-dependent control points locations. We remark that, in contrast, the RBF approach (even though still based on interpolation)
required a further solution of linear system of size $d \times \mathcal{N}_C + d + d^2$.

Application in the context of fluid-structure interaction (FSI) problems between a wing (structure) and surrounding air (fluid) is shown in Figure~\ref{fig:idw}. The IDW deformation of the fluid mesh resulting from a vertical displacement of the tip of the wing is depicted; the structural mesh is omitted from the picture. We refer to \cite{BallarinDAmarioPerottoRozza2017} for more details.

\begin{figure}[h!]
\centering
\includegraphics[width=0.7\textwidth]{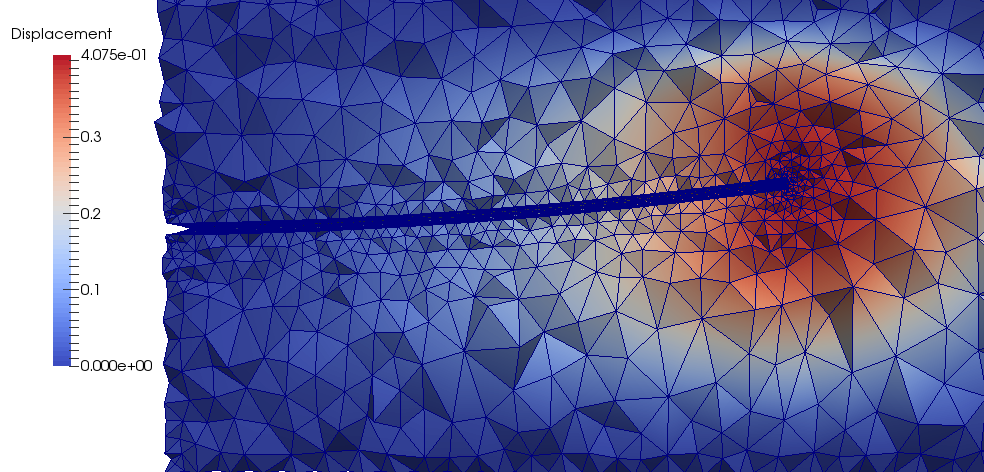}
\caption{Deformation of the fluid mesh of a fluid-structure interaction problem by IDW.}
\label{fig:idw}
\end{figure}

\section{Beyond affinity assumptions: parametric interpolation}
\label{sec:nonaffinity}
\def\plist@algorithm{\algorithmname\space}

We describe here several options to deal with cases when an exact affine decomposition of the discretized differential operators, right hand sides or outputs of interest is not existing. The section begins with a brief overview concerning the description of general non-affine problems \ref{subsec:nonaffine} and later we describe the so-called empirical interpolation method (EIM) family of algorithms. This methodology becomes particularly useful to obtain an efficient offline-online splitting also in cases with nonlinearities and non-affine parametrization. We provide a full description of the different alternatives, starting from its standard continuous version (EIM), and presenting also its discrete (DEIM) and matrix (M-DEIM) variants. The methodologies are tested for both non-affine and non-linear problems. In section \ref{subsec:EIM} we explain in detail the basics of the empirical interpolation method. In section \ref{subsec:DEIM} we introduce the discrete variant of the empirical interpolation method at both matrix and vector level and we mention further options to obtain an approximate affine expansion. In section \ref{subsec:example} we present two examples using the EIM (section \ref{subsubsec:ex1}) and the M-DEIM algorithm to deal with both non-affinity and non-linearity (section \ref{subsubsec:ex2}).

\subsection{Non-affine problems}\label{subsec:nonaffine}
As already discussed in \ref{subsec:affine}, the existence of an affine decomposition of the linear and bilinear forms of the considered problem is crucial in order to obtain a computationally efficient framework, see \eqref{chap_1_1:affine_param_dep_bilifo}-\eqref{chap_1_1:affine_param_dep_bilifo_l}.

This assumption fails to be true in several situations. Such situations occur for example in case of problems with non-affine parametric-dependency, in cases with non-linear differential operators and in cases dealing with the non-affine geometrical parametrizations introduced in section \ref{sec:geom}. 

In fact, in these situations, the differential operators or the right-hand sides or the outputs of interest cannot be directly written using an exact affine decomposition and we have therefore to rely on an approximate affine decomposition. The Empirical Interpolation Method (EIM) is one of the key instruments to recover an approximate affine decomposition.

The EIM is a general tool for the approximation of parameterized or non-linear functions by a sum of affine terms. In the expression below we report an example for a generic parameterized function $f$:
\begin{equation}
f(\bm x;\bm \mu) \approx \sum_{q=1}^Q c_q(\bm \mu)h_q(\bm x).
\end{equation}
The EIM has been firstly proposed in~\cite{Barrault2004} to deal with non-affine problems in the context of RB methods and later applied to reduced order modelling in~\cite{Grepl2007}. In \cite{Maday2008} it has been extended to a general context, a slightly different variant of EIM namely Discrete Empirical Interpolation Method has been firstly proposed in \cite{Chaturantabut_2009,Chaturantabut2010}. For more details on the a posteriori error analysis the interested reader may see \cite{Grepl2007,Eftang2010,Chen_2014} while for an extension to $hp$-adaptive EIM we refer to \cite{Eftang_2012}. A generalization of the EIM family of algorithms has been proposed in \cite{Maday2013,Chen_2014,Maday2016} while a nonintrusive EIM technique is presented in \cite{Casenave_2014} and an extension with special focus on high dimensional parameter spaces is given in \cite{Hesthaven_2014}. 

\subsection{The empirical interpolation method (EIM)}\label{subsec:EIM}
The EIM is a general method to approximate a parametrized function $f (\bm x;\bm \mu) :\Omega \times \peim \to \mathbb{R}$ by a linear combination of $Q$ precomputed basis functions in the case where each function $f_\mu:=(\cdot;\mupar)$ belongs to some Banach space $\mathcal{X}_\Omega$. In what follows $\mupar\in\peim$ is the parameter vector and $\peim$ is the parameter space. The EIM approximation is based on an interpolation operator $\mathrm{I}_Q$ that interpolates the given function $f_\mu$ in a set of interpolation points $\{\bm{x}_i\}_{i=1}^Q \in \Omega$.  The interpolant function is constructed as a linear combination of hierarchically chosen basis functions $\{h_i\}_{i=1}^Q \in \mathbb{V}_{EIM}$, where $\mathbb{V}_{EIM}$ is an approximation of the function space $\mathcal{U}$ that contains $f$, i.e. $\mathbb{V}_{EIM}\subseteq \mathcal{U}$. On the contrary to other interpolation methods, that usually work with generic and multi-purpose basis functions such as polynomial functions, the EIM works with problem-specific basis functions with global support and selected hierarchically. The interpolant function can be then expressed by:
\begin{equation}
\mathrm{I}_Q \left[ f_{\mu} \right] (\bm x) = \sum_{q=1}^Q c_q(\bm \mu)h_q(\bm x), \quad \bm x \in \Omega, \quad \mupar\in\mathcal{P}_{EIM},
\end{equation}
where $c_q$ are parameter dependent coefficients. Once the basis functions $h_q(\bm x)$ are set, the problem of finding the coefficients $c_q(\bm\mu)$ is solved imposing the interpolation condition, i.e.:
\begin{equation}
\mathrm{I}_Q \left[ f_\mu \right] (\bm x_q) = \sum_{q=1}^Q c_q(\bm \mu)h_q(\bm x_q) = f_\mu(\bm x_q), \quad q=1, \dots ,Q.
\end{equation}

The above problem can be recast in matrix form as $\bm T \bm c_\mu = \bm f_\mu$ with:
\begin{equation}\label{eq:inter_matrix}
(\bm T)_{ij} = h_j(\bm x_i), \quad (\bm c_\mu)_j = c_j(\bm \mu), \quad (\bm f (\bm \mu))_j = f(\bm x_i; \bm \mu), \quad i,j = 1,\dots,Q.
\end{equation}

This problem can be easily solved given the fact that the basis functions $h_q(\bm x)$ and the interpolation points $\bm x_q$ are known and that the matrix $\bm T$ is invertible.

The selection of the basis functions $\{h_q\}_{q=1}^Q$ and of the interpolation points $\{\bm x_q\}_{q=1}^Q$, that are defined by a linear combination of selected function realizations $\{f_{\mu_i}\}_{i=1}^Q$, is done following a greedy approach similar to the one presented in Subsection~\ref{subsec:greedy} (see~Algorithm \ref{alg:eim}). The procedure provides also a set of sample points $\{\mupar_q\}_{q=1}^Q$ that are required for the construction of the basis functions.

Since the basis functions are defined as linear combinations of the function realizations inside the parameter space, in order to approximate the function $f$ with a relatively small number of basis functions $h_q$, the manifold:
\begin{equation}
\mathcal M_{\mathrm{EIM}} = \{ f(\bm x; \mupar) | \mupar \in \peim\},
\end{equation}
must have a small Kolmogorov N-width \cite{Kolmogoroff1936}.

Once a proper norm on $\Omega$ has been defined, and here we consider $L^p(\Omega)$-norms for $1\leq p\leq \infty$, the procedure starts with the selection of the first parameter sample which is computed as:
$$
\bm \mu_1 = \argsup_{\bm \mu \in \peim} \norm{f_\mu(\bm x)}_{L^p(\Omega)},
$$
while the first interpolation point is computed as:
$$
\bm x_1 = \argsup_{\bm x \in \Omega}|f_{\mu_1}(\bm x)|.
$$
The first basis function and the interpolation operator at this stage are then defined as:
$$
h_1(x) = \frac{f_{\mu_1}(\bm x)}{f_{\mu_1}(\bm x_1)}, \quad \mathrm{I}_1[f_\mu](\bm x) = f(\bm x_1;\mupar)h_1(\bm x). 
$$
At the subsequent steps, the next basis function is selected as the one that is the worse approximated by the current interpolation operator and using a similar concept the interpolation point, often referred as \emph{magic point}, is the one where the interpolation error is maximized. In mathematical terms, at the step $k$, the sample point is selected as the one that maximizes the error between the function $f$ and the interpolation operator computed at the previous step $\mathrm{I}_{k-1}[f]$:
$$
\mupar_{k} = \argsup_{\bm \mu \in \peim} \norm{f_\mu(\bm x)-\mathrm{I}_{k-1} \left[ f_{\mu} \right](\bm x)}_{L^p(\Omega)}.
$$
Once the sample point has been determined, the interpolation point is selected, in a similar fashion, as the point inside the domain that maximizes the error between the function $f$ and the interpolation operator:
$$
\bm x_k = \argsup_{\bm x \in \Omega}|f_{\mu_k}(\bm x) - \mathrm{I}_{k-1} \left[ f_{\mu_k} \right](\bm x)|.
$$
The next basis function is defined similarly to the first one with:
$$
h_k(\bm x) = \frac{f_{\mu_k}(\bm x) - \mathrm{I}_{k-1}[ f_{\mu_k} ](\bm x)}{f_{\mu_k}(\bm x_k) - \mathrm{I}_{k-1}[ f_{\mu_k} ](\bm x_k;)}
$$
The procedure is repeated until a certain tolerance $\mathrm{tol}$ is reached or a maximum number of terms $N_{max}$ are computed (see Algorithm~\ref{alg:eim}).
We remark that by construction the basis functions $\{ h_1, \dots, h_Q \}$ and the functions $\{ f_{\mu_1}, \dots,f_{\mu_Q} \}$ span the same space $\mathbb{V}_{EIM}$:
$$
\mathbb{V}_{EIM} = \spn{h_1, \dots, h_Q} = \spn{f_{\mu_1}, \dots,f_{\mu_Q}}.
$$
However, the former are preferred for the following reasons (for more details and for the mathematical proofs we refer to \cite{Barrault2004}):
\begin{itemize}
\item They are linearly independent,
\item $h_i(\bm x_i)=1$ for $1\leq i \leq Q$  and $h_i(\bm x_j) = 0$ for $1\leq i \leq j \leq Q$,
\item They make the interpolation matrix $\bm T$ of Equation \ref{eq:inter_matrix} to be lower triangular and with diagonal elements equal to unity and therefore the matrix is invertible.
\end{itemize} 
The third point implies that the interpolation problem is well-posed.
\begin{algorithm}
\caption{The EIM algorithm - Continuous Version}
\begin{algorithmic} 
\REQUIRE set of parameterized functions $f_\mu:\Omega \to \mathbb{R}$, tolerance $\mathrm{tol}$ and maximum number of basis functions $N_{max}$, $p$ order of the chosen $p$-norm.
\ENSURE basis functions $\{h_1, ..., h_Q \}$, interpolation points $\{ \bm x_1, \dots,\bm x_Q\}$;
\STATE $k = 1$ ; $\varepsilon = \mathrm{tol}+1$;
\WHILE {$k < N_{max} \mbox{ and } \varepsilon > \mathrm{tol}$}
\STATE Pick the Sample point:\\
$\mupar_{k} = \argsup_{\bm \mu \in \peim} \norm{f_\mu(\bm x)-\mathrm{I}_{k-1} \left[ f_{\mu} \right](\bm x)}_{L^p(\Omega)}$;
\STATE Compute the corresponding interpolation point:\\
$\bm x_k = \argsup_{\bm x \in \Omega}|f_{\mu_k}(\bm x) - \mathrm{I}_{k-1} \left[ f_{\mu_k} \right](\bm x)|$;
\STATE Define the next basis function:\\
$h_k(\bm x) = \frac{f_{\mu_k}(\bm x) - \mathrm{I}_{k-1}[ f_{\mu_k} ](\bm x)}{f_{\mu_k}(\bm x_k) - \mathrm{I}_{k-1}[ f_{\mu_k} ](\bm x_k;)}$;
\STATE Compute the error level: \\
$\varepsilon = \norm{\varepsilon_p}_{L^\infty}$ with $\varepsilon_p (\mupar) =  \norm{f_{\mu}(\bm x)-\mathrm{I}_{k-1}[f_{\mu}](\bm x)}_{L^p(\Omega)}$;
\STATE $k = k+1$;
\ENDWHILE
\end{algorithmic}
\label{alg:eim}
\end{algorithm}
\subsubsection{Error Analysis}
Dealing with interpolation procedures, the error analysis usually involves a Lebesgue constant. In particular, in case one is using the $L^\infty(\Omega)$-norm the error analysis involves the computation of the Lebesgue constant $\Lambda_{q} = \sup_{\bm x \in \Omega}\sum_{i=1}^q |L_i(\bm x)|$ being $L_i \in \mathbb{V}_{EIM}$ Lagrange functions that satisfies $L_i(x_j) = \delta_{ij}$. It can be proved that the interpolation error is bounded by the following expression \cite{Barrault2004}:
\begin{equation}
\norm{f_\mu - \mathrm{I_q}[f_\mu]}_{L^\infty(\Omega)}\leq(1+\Lambda_q) \inf_{v_q \in \mathbb{V}_{EIM}}\norm{f_\mu -v_q}_{L^\infty(\Omega)}.
\end{equation}
An upper bound for the Lebesgue constant, that in practice has been demonstrated to be very conservative \cite{Barrault2004} can be computed as:
$$
\Lambda_q \leq 2^q-1.
$$
For more details concerning the estimates of the interpolation error we refer to \cite{Barrault2004, Maday2008}.
\subsubsection{Practical implementation of the algorithm}
Practically, finding the maximum of Algorithm~\ref{alg:eim} is usually not feasible and therefore the continuous version must be transformed into a computable one. 

This is done selecting a finite dimensional set of training points in the parameter space $\{ \mupar_i\}_{i=1}^N \in \mathcal{P}_{\mathrm{EIM}}^{\mathrm{train}} \subset \mathcal{P_{\mathrm{EIM}}}$ and in the physical domain $\{\bm x_i\}_{i=1}^M \in \Omega_h \subset \Omega$. For this reason we introduce the vector $\bm f:\Omega_h\times\mathcal{P}_{\mathrm{EIM}}^{\mathrm{train}}\to\mathbb{R}^{M}$ which consists into a discrete representation of the function $f$:
\begin{equation}
(\bm{f}_\mu)_i = f_\mu(\bm x_i), \quad i=1,\dots,M.
\end{equation}
We also define the matrix $\bm H_Q \in \mathbb{R}^{M\times Q}$ which is defined by the discrete basis functions $H_Q = [\bm h_1,\dots,\bm h_Q]$ and the interpolation indices vector $\bm i_Q = (i_1, \dots, i_Q)$. The discrete interpolation operator of order $Q$ for the vector function $\bm f$ is then defined by:
\begin{equation}
\mathrm{I}_Q[\bm f_{\mupar}] = \bm{H}_Q \bm a_{\bm f_{\mupar}},
\end{equation}
where the coefficients $\bm a_{\bm f_{\mupar}}$ are defined such that $\bm T \bm a_{\bm f_{\mupar}} = \bm f_{\mupar}$, being:
\begin{equation}
\bm T_{kq} = (\bm H_Q)_{i_k q}, \quad k,q = 1,\dots,Q.
\end{equation}
The implementation of the algorithm is similar to the continuous version and is reported in Algorithm~\ref{alg:eim_pract}. In the algorithm we use the notation $\bm F_{:,j}$ to denote the $j$-th column of the matrix $\bm F$. Where $\bm F\in\mathbb{R}^{M\times N}$ is a matrix containing vector representations of the function $\bm f$:
\begin{equation}
(\bm F)_{ij} = f(\bm x_i; \mupar_j).
\end{equation}
Once the basis and the interpolation indices are defined, during the online stage it is required to make a point-wise evaluation of the $f$ function in the points defined by the interpolation indices. 
\begin{algorithm}[H]
\caption{The EIM algorithm - Practical Implementation}
\begin{algorithmic} 
\REQUIRE set of parameter samples $\{\mupar_i\}_{i=1}^M \in \mathcal{P}_{\mathrm{EIM}}^{\mathrm{train}} \subset \mathcal{P_{\mathrm{EIM}}}$, set of discrete points $\{\bm x_i\}_{i=1}^N \in \Omega^{\mathrm{train}}$, tolerance $\mathrm{tol}$, maximum number of basis functions $N_{max}$, $p$ order of the chosen $p$-norm.
\ENSURE basis functions matrix $\bm H_Q = \{\bm h_1, ..., \bm h_Q \}$, interpolation indices vector $\bm i_Q=\{i_1, \dots,i_Q\}$;
\STATE Assemble the matrix:\\
$(\bm F)_{ij} = f(\bm x_i; \mupar_j), \quad i=1,\dots,M, \quad j=1,\dots,N$;
\STATE $k=1, \quad \varepsilon = \mathrm{tol}+1$;
\WHILE {$k<N_{max}$ and $\varepsilon > \mathrm{tol}$}
\STATE Pick the sample index:\\
\STATE 	$j_k = \argmax_{j=1,\dots,M}\norm{\bm F_{:,j}-\mathrm{I}_{k-1}[\bm F_{:,j}]}_{L^p}$;
\STATE and compute the interpolation point index: \\
$i_k = \argmax_{i=1,\dots,N}|\bm F_{i,j_k}-(\mathrm{I}_{k-1}[\bm F_{:,j_k}])_i|$;
\STATE define the next approximation column: \\
$\bm h_k = \frac{\bm F_{:,j_k} - \mathrm{I}_{k-1}[\bm F_{:,j_k}]}{\bm F_{i_k,j_k} - (\mathrm{I}_{k-1}[\bm F_{:,j_k}])_{i_k}}$
\STATE define the error level: \\
$\varepsilon=\max\limits_{j=1,\dots,M}\norm{\bm F_{:,j} - \mathrm{I}_{k-1}[\bm F_{:,j}]}_{L^p}$
\STATE $k = k+1$
\ENDWHILE
\end{algorithmic}
\label{alg:eim_pract}
\end{algorithm}
\subsection{The Discrete Empirical Interpolation Method (DEIM)}\label{subsec:DEIM}
The computable version of EIM is similar to the so-called Discrete Empirical Interpolation Method (DEIM) introduced in \cite{Chaturantabut2010}. The main difference between the EIM and the DEIM is given by the way the basis functions are computed. In the DEIM the basis functions are computed relying on a POD procedure which is performed on a set of discrete snapshots of the parametrized function $\{\bm f_{i}\}_{i=1}^M$. Each snapshot $\bm f_i$ is already considered in discrete form in a prescribed set of points $\{\bm x_i\}_{i=1}^{N_h}$.   
The procedure, which is described in detail in Algorithm~\ref{alg:deim} can be summarized into the following steps:
\begin{enumerate}
\item Construct the DEIM basis functions using a POD procedure on a set of previously computed snapshots: 
\begin{equation}
\bm H_M = [\bm h_1,\dots,\bm h_M] = \mathrm{POD}(\bm f(\mupar_1,\dots,\mupar_M)).
\end{equation}
\item Given a prescribed tolerance $\mathrm{tol}$ determine the indices $
\bm{i}_Q$ and truncate the dimension of the POD space using an iterative greedy approach (see Algorithm~\ref{alg:deim}).
\end{enumerate}
In Algorithm~\ref{alg:deim}, with the term $\bm e_{i_k}$, we identify a vector of dimension $N_h$ where the only non-null element is equal to $1$ and is located at the index $i_k$:
$$
(\bm e_{i_k})_j = 1 \mbox{ for } j=i_k, \quad (\bm e_{i_k})_j = 0 \mbox{ for } j\neq i_k.
$$
During the online stage, when a new value of the parameter $\mupar$ needs to be tested, it is
required to compute the function $\bm f(\mupar)$ only in the location identified by the
indices $\bm i_Q$. Therefore, the nonlinear function needs to be evaluated only in a relatively small number of points which is usually much smaller with respect to the total number of degrees of freedom used to discretize the domain.
\begin{algorithm}
\caption{The DEIM procedure}
\begin{algorithmic}
\REQUIRE {snapshots matrix $\bm S = [\bm f(\mupar_1), \dots, \bm f (\mupar_M)]$ , tolerance $\mathrm{tol}$.}
\ENSURE DEIM basis Functions $\bm H_Q = [\bm h_1, \dots, \bm h_Q]$, Interpolation indices $\bm i_Q = [i_1, \dots, i_Q].$
\STATE compute the DEIM modes $\bm H_M = [\bm h_1, \dots, \bm h_M]=\mathrm{POD}(\bm S)$
\STATE $\varepsilon = \mathrm{tol} +1, \quad k=1$
\STATE $i_1 = \argmax_{j=1,N_h} | (\bm h_1)_j| $
\STATE $\bm H_Q = [\bm h_1], \quad \bm i_Q = [i_1], \quad \bm P = [\bm e_{i_1}]$
\WHILE {$\varepsilon > \mathrm{tol}$}
\STATE $k=k+1$
\STATE Solve $(\bm P^T \bm H_Q)\bm c = \bm P^T \bm h_{k}$ 
\STATE $\bm r = \bm h_{k} - \bm H_Q \bm c$
\STATE $i_k = \argmax_{j=1,N_h} | (\bm r)_j|$
\STATE $\bm H_Q = [\bm H_Q, \bm h_k], \quad \bm P = [\bm P, \bm e_{i_k}], \quad  \bm i_Q = [\bm i_Q, i_k]$ 
\ENDWHILE
\end{algorithmic}
\label{alg:deim}
\end{algorithm}
\subsection{Further options}
\label{subsec:further_options}

Apart from the EIM and the DEIM algorithm further options are available. We mention here the
matrix version of the DEIM algorithm (M-DEIM) \cite{Bonomi2017} that extends the DEIM also to the case of parametrized or non-linear matrices, the generalized empirical interpolation method (GEIM) \cite{Maday2013} and the gappy-POD \cite{BuiThanh2003,Carlberg2010}. 

The M-DEIM is used to perform model order reduction on discretized differential operators characterized by non-linearity or non-affinity with respect to the parameter vector $\mupar$. The algorithm is similar to the one in Algorithm \ref{alg:deim} with the only difference that a vectorized version of the matrices is used to describe snapshots and POD modes. In section \ref{subsec:example} we will provide an example dealing with both issues. 

The gappy-POD generalizes the interpolation condition to the case where the number of basis functions is smaller then the number of interpolation indices, i.e. $\mathrm{card}(\bm H_Q) < \mathrm{card}(\bm i_Q) $. In this case the interpolation condition is substituted by a least-squares regression.

The GEIM replaces the EIM requirement of a point-wise interpolation condition by the statement:
\begin{equation}
\sigma_j(\mathrm I_Q(f(\mupar))) = \sigma_j (f(\mupar)), \quad j=1,\dots,Q,
\end{equation}
where $\sigma_j$ are a set of ``well-chosen'' linear functionals. For more details and for convergence analysis of the present method we refer to \cite{Maday2016}.

\subsection{Some Examples}\label{subsec:example}
In the previous sections we have presented the empirical interpolation method family of algorithm and we have illustrated how it is possible to recover an approximate affine expansion of the discretized differential operators. In this section we show in more detail two examples on the practical application of the EIM and the M-DEIM algorithm. 
\subsubsection{An heat transfer problem with a parametrized non-affine dependency forcing term}\label{subsubsec:ex1}
In this example we illustrate the application of the computable version of the EIM on a steady state heat conduction problem in a two-dimensional square domain $\Omega = [-1,1]^2$ with a parametrized forcing term $g(\mupar)$ and homogeneous Dirichlet boundary conditions on the boundary $\partial \Omega$. The problem is described by the following equation:
\begin{equation}
\begin{cases}
- \alpha_t \Delta \theta = g(\mupar), \quad \mbox{ in } \Omega, \\
\theta = 0, \quad \mbox{ on } \partial \Omega,
\end{cases}
\end{equation}
where $\theta$ is the temperature field, $\alpha_t$ is the thermal conductivity coefficient and $g(\mupar)$ is the parametrized forcing term which is described by the following expression:
\begin{equation}\label{eq:par_function}
g(\bm x; \mupar) = e^{-2(x_1-\mu_1)^2-2(x_2-\mu_2)^2},
\end{equation}
where $\mu_1$ and $\mu_2$ are the first and second components of the parameter vector and $x_1$ and $x_2$ are the horizontal and vertical coordinates respectively. 
Let $V$ be a Hilbert space, the weak formulation of the problem can be written as, find 
$\theta \in V$ such that:
\begin{equation}
a(\theta(\mupar), v; \mupar) = f(v;\mupar), \quad \forall v \in V,
\end{equation}
where the parametrized bilinear and linear forms are expressed by:
\begin{equation}
a(\theta,v;\mupar) = \int_{\Omega} \nabla \theta \cdot \nabla v d \bm x, \quad f(v;\mupar) = \int_{\Omega} g(\bm x;\mupar) v d \bm x.
\end{equation}
In the above expressions, the bilinear form $a(\cdot,\cdot;\mupar): V\times V \to \mathbb{R}$ is trivially affine while for the linear form $f(\cdot;\mupar):V \to \mathbb{R}$ we have to rely on an approximate affine expansion using the empirical interpolation method. The problem is discretized using triangular linear finite elements according to the mesh reported on left side of Figure \ref{fig:par_function_decay}.

In the present case it is not possible to write an exact affine decomposition of the linear form $f$, we rely therefore on the computable version of the empirical interpolation method of Algorithm~\ref{alg:eim_pract} in order to recover an approximate affine expansion.

\begin{figure}[h]
\centering
\begin{minipage}{0.9\textwidth}
\begin{minipage}{0.48\textwidth}
\centering
\includegraphics[width=0.6\textwidth]{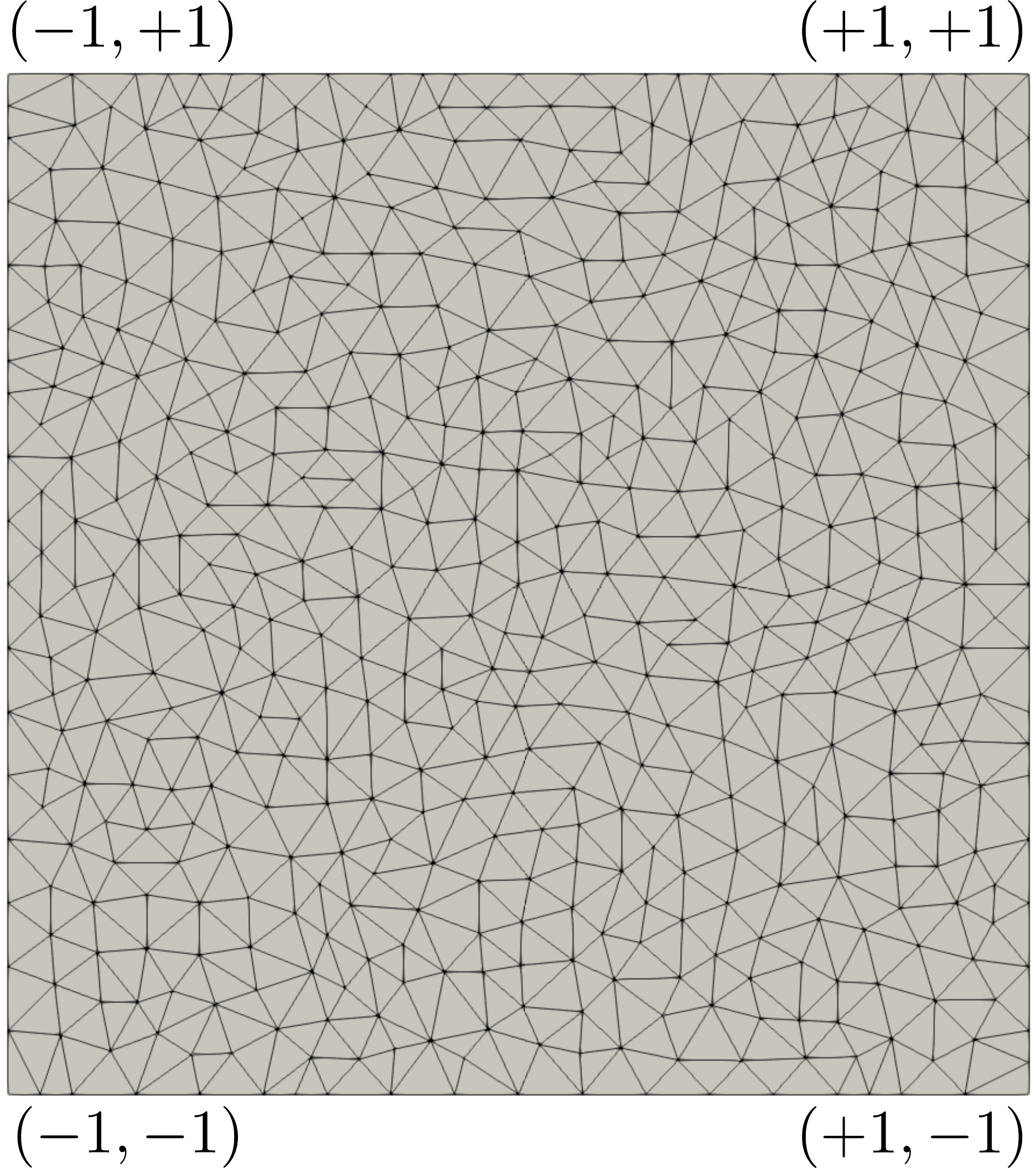}
\end{minipage}
\begin{minipage}{0.48\textwidth}
\centering
\vspace{0.7cm}
\includegraphics[width=0.6\textwidth]{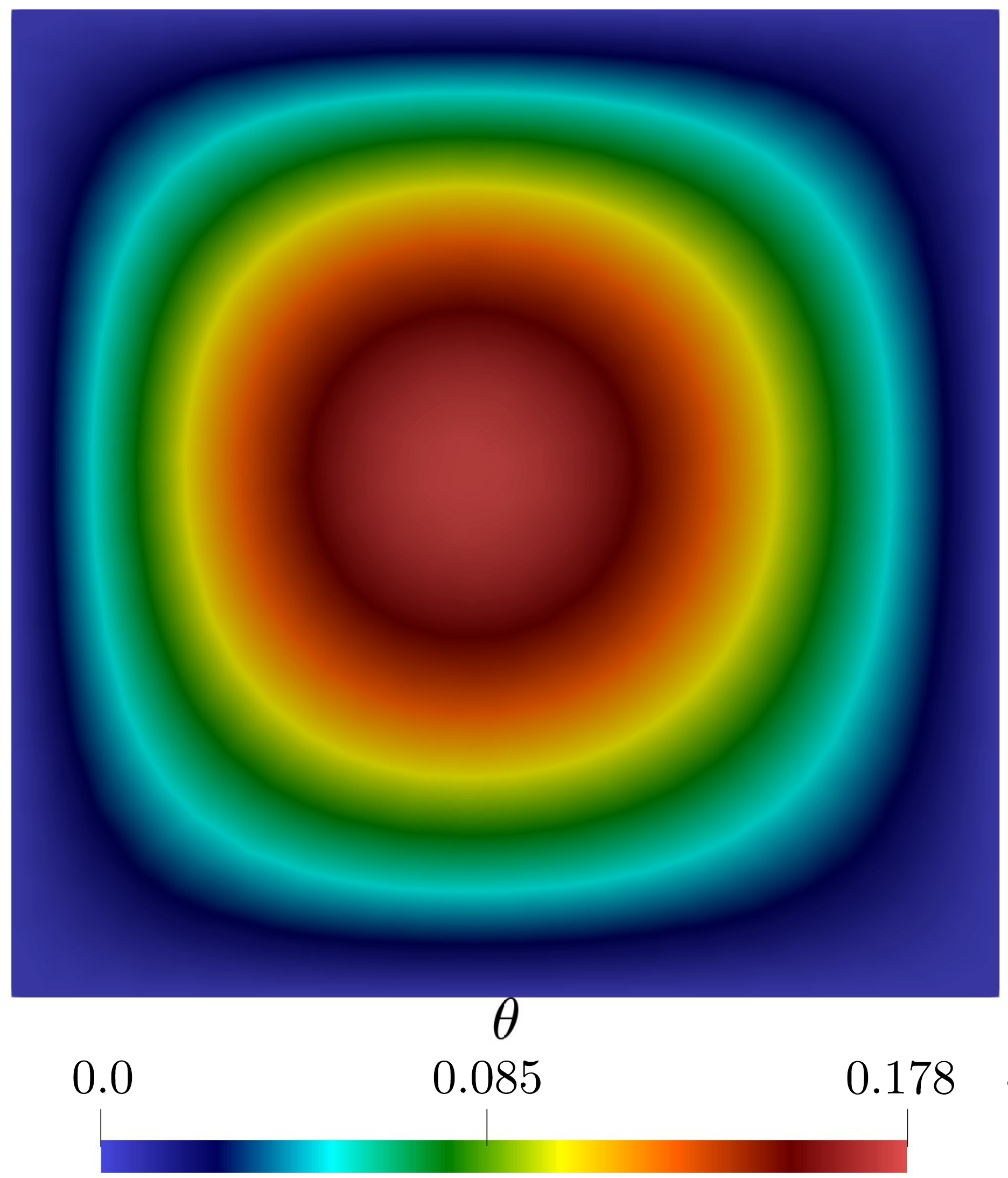}
\end{minipage}
\end{minipage}
\caption{Discretized domain into which the parameterized problem is solved (left image), together with an example of the value assumed by temperature field for one particular sample point inside the parameter space (right image).}\label{fig:par_function}
\end{figure}

The function $g(\bm x;\mupar)$ is parametrized with the parameter vector $\mupar = (\mu_1,\mu_2) \in \mathcal \peim = [-1,1]^2$ that describes the position of the center of the Gaussian function. The conductivity coefficient $\alpha_t$ is fixed constant and equal to $1$. The testing set for the implementation of the algorithm $\{\mupar_i\}_{i=N_{train}}\in \mathcal{P}_{train}$ is defined using $N_{train}=100$ and a uniform probability distribution. The set of points $\{\bm x_i\}_{i=1}^{N_h}\in \Omega$, that is used for the idenfitication of the magic points, is chosen to be coincident with the nodes of the finite element grid reported in Figure~\ref{fig:par_function}. 
\begin{figure}[h]
\centering
\begin{minipage}{0.8\textwidth}
\begin{minipage}{0.32\textwidth}
\centering
\includegraphics[width=\textwidth]{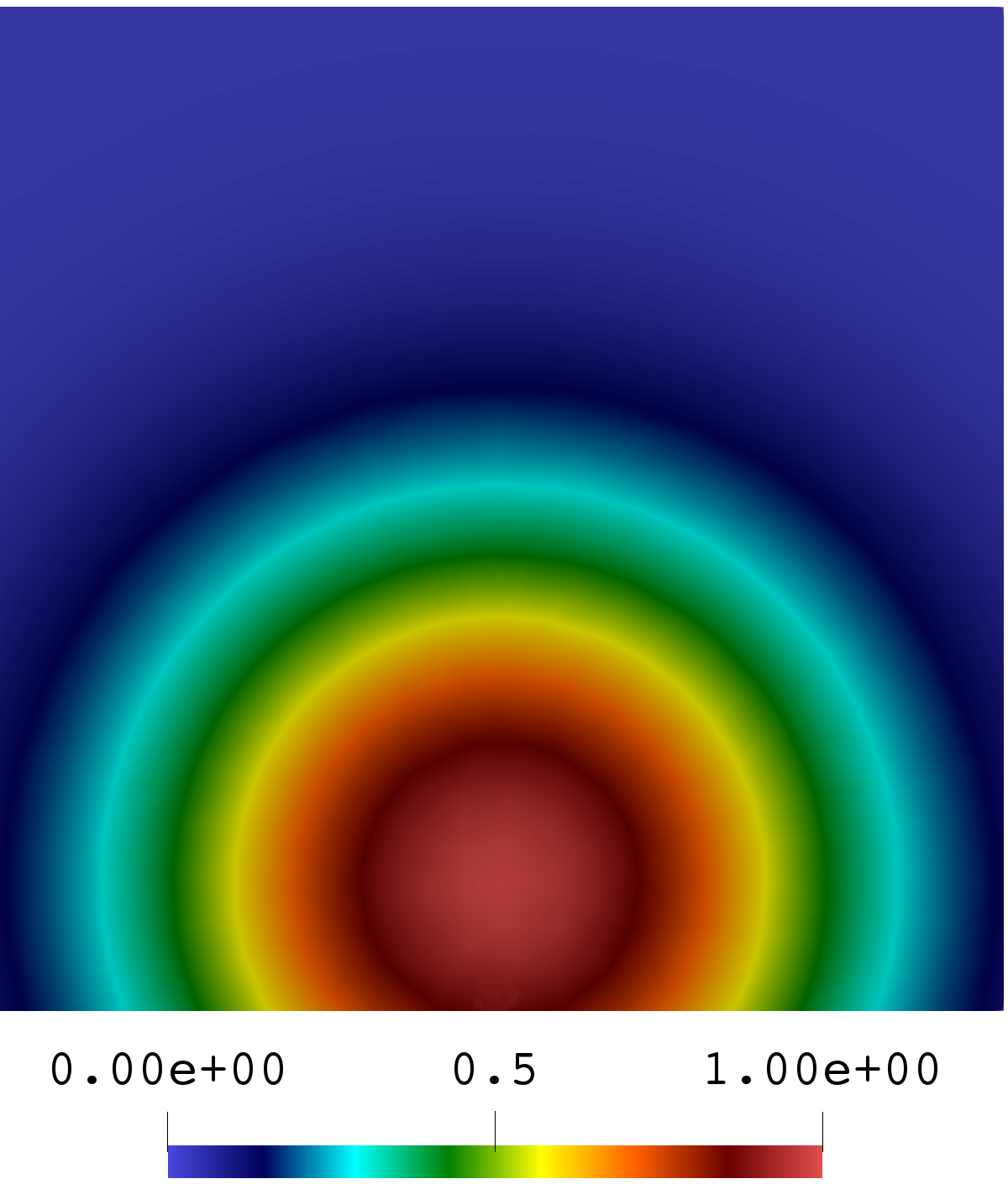}
\end{minipage}
\begin{minipage}{0.32\textwidth}
\centering
\includegraphics[width=\textwidth]{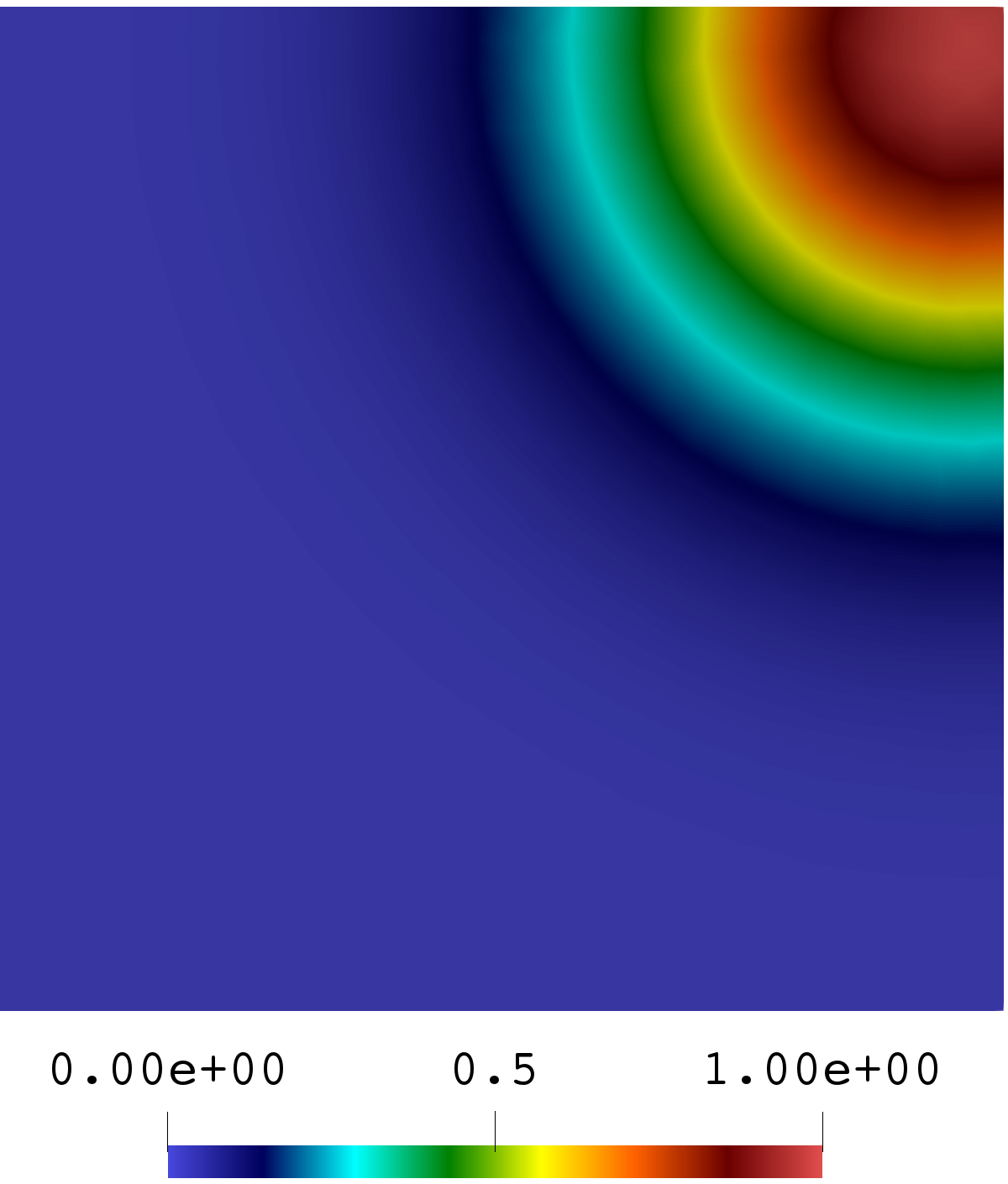}
\end{minipage}
\begin{minipage}{0.32\textwidth}
\centering
\includegraphics[width=\textwidth]{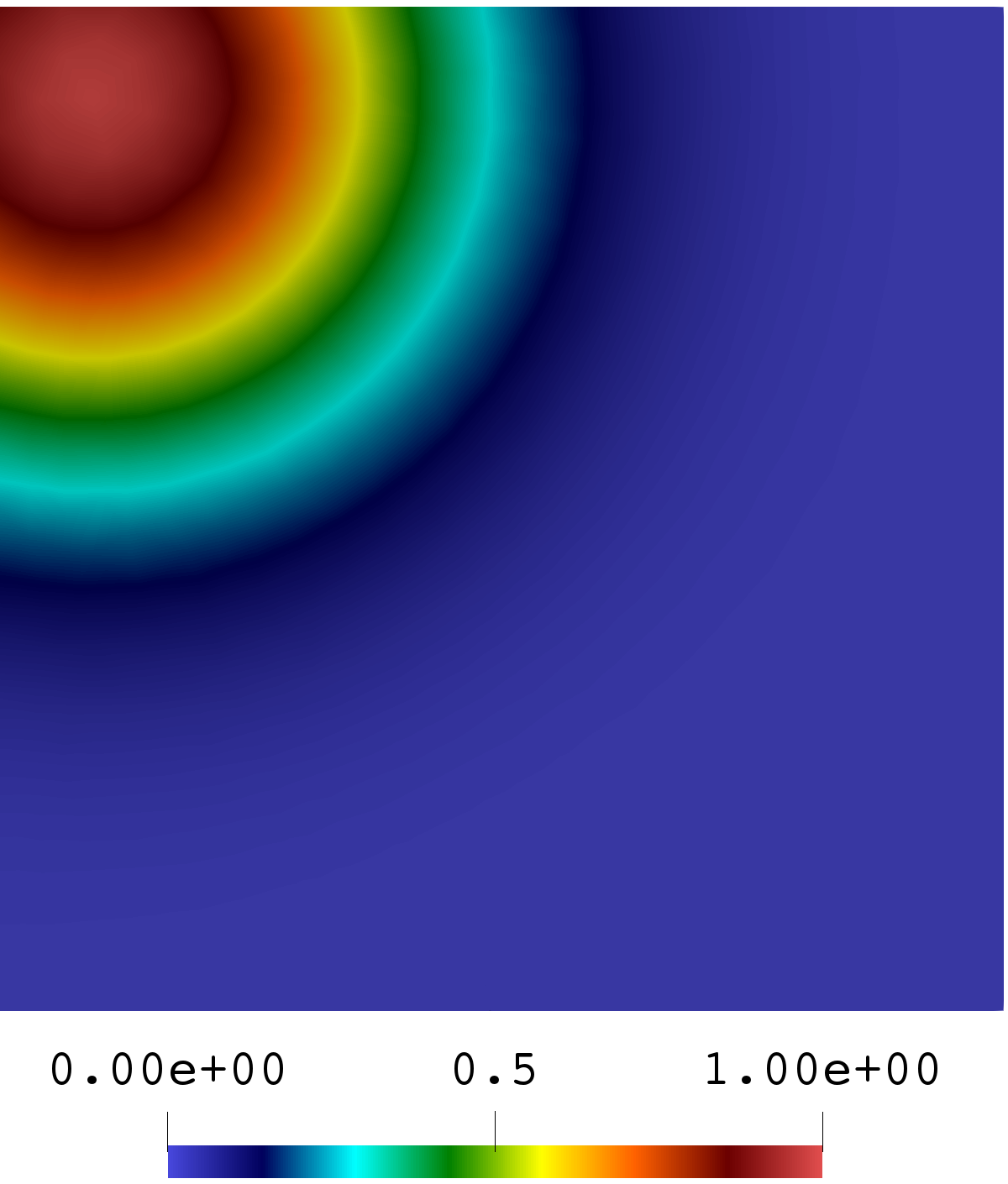}
\end{minipage}
\begin{minipage}{0.32\textwidth}
\centering
\includegraphics[width=\textwidth]{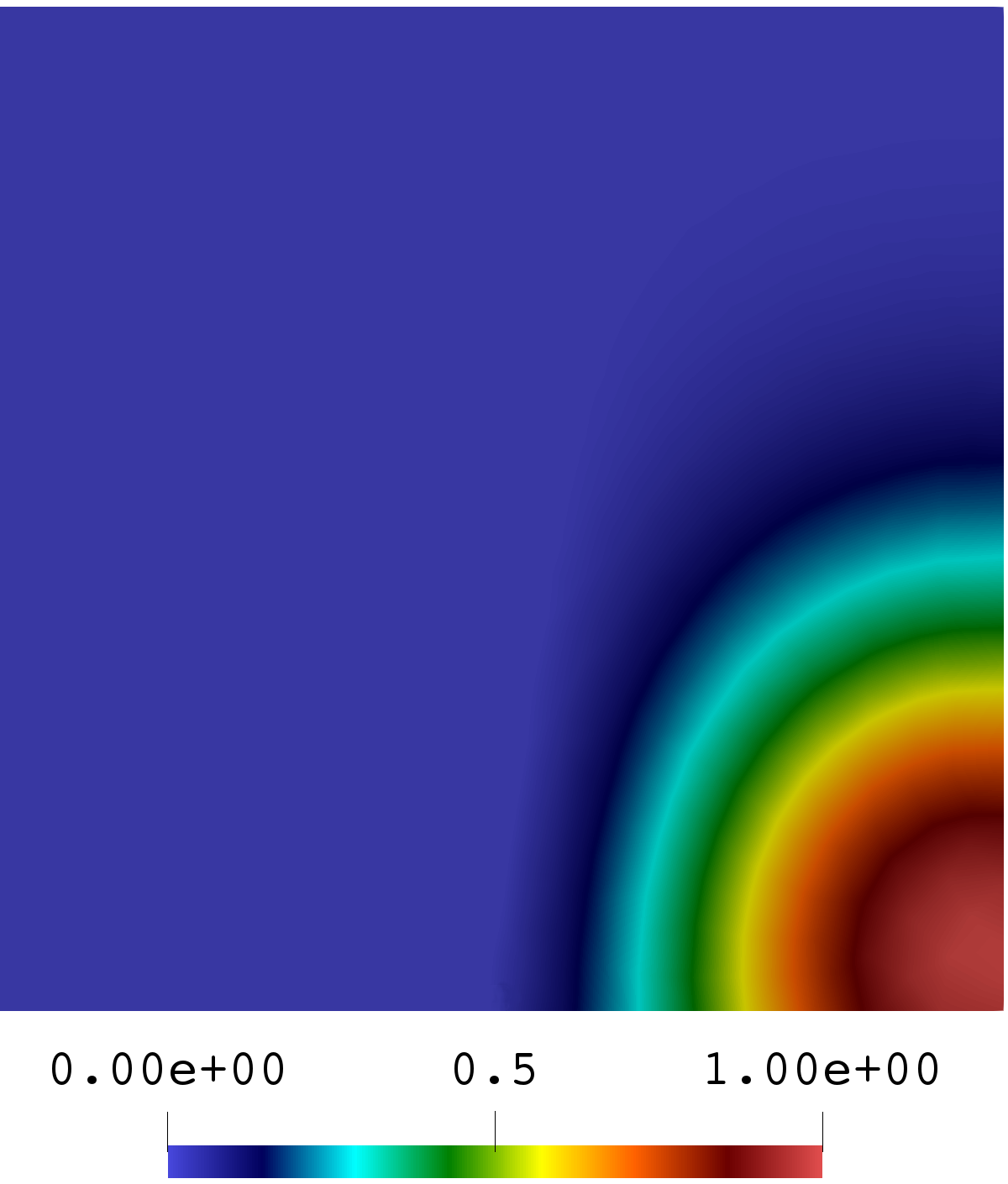}
\end{minipage}
\begin{minipage}{0.32\textwidth}
\centering
\vspace{-0.65cm}
\includegraphics[width=\textwidth]{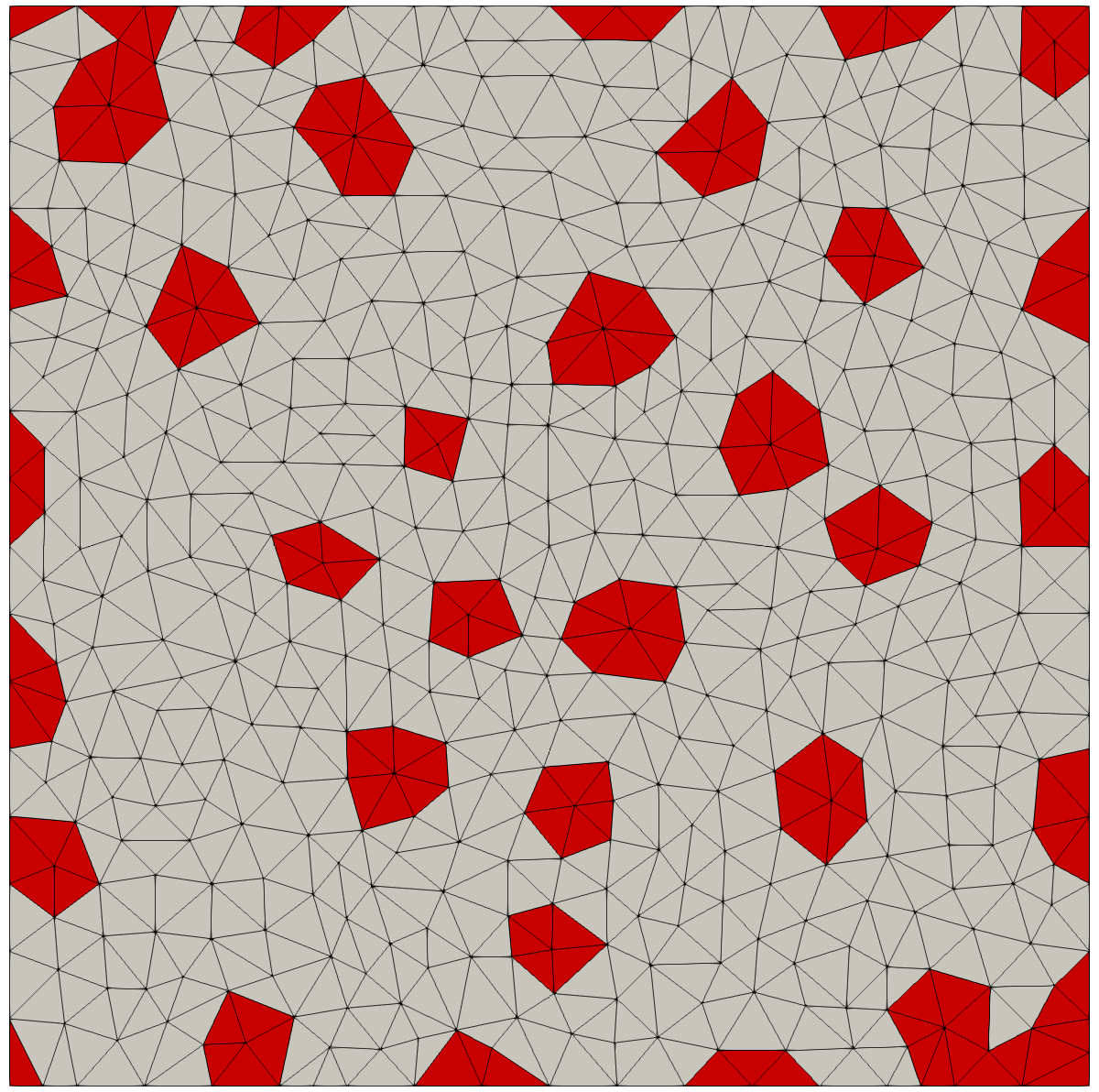}
\end{minipage}
\end{minipage}
\caption{Plot of the first four modes identified by the EIM algorithm (first row and left image in the second row) and the location of the first $35$ indices $i_Q$. The magic points are identified by the red elements in the right picture on the second row.}\label{fig:eimmodes}
\end{figure}
In Figure \ref{fig:eimmodes} we report the first four EIM basis functions for the non-linear function $g$ and the location of the magic points identified by the EIM algorithm. In Figure \ref{fig:par_function_decay} we report the convergence analysis of the EIM algorithm for the nonlinear function $g$ changing the number of EIM basis functions (left plot) and the convergence analysis of the reduced order model changing the number of reduced basis functions (right plot). 

\begin{figure}[H]
\begin{minipage}{0.5\textwidth}
\centering
\includegraphics[width=\textwidth]{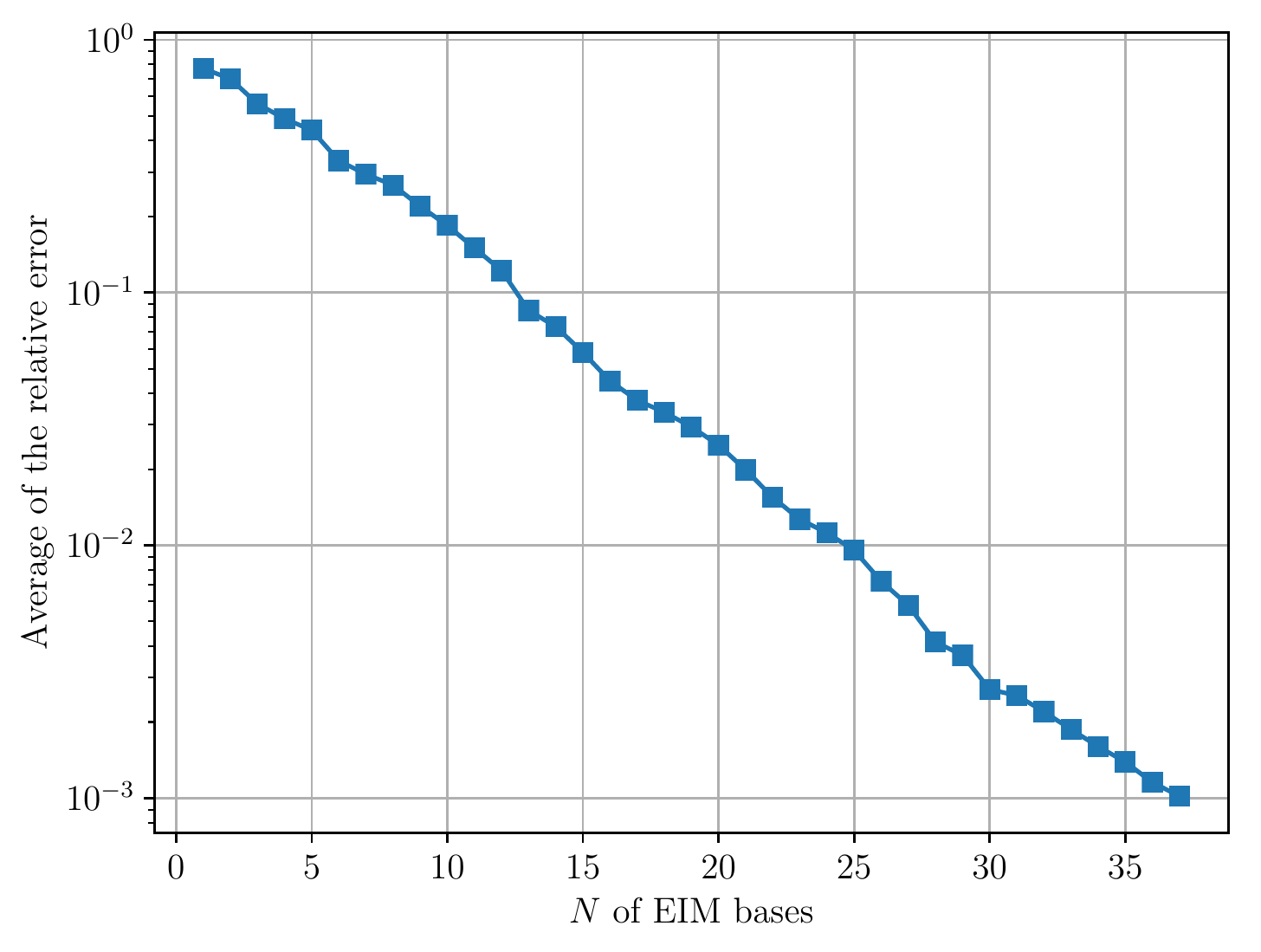}
\end{minipage}
\begin{minipage}{0.5\textwidth}
\centering
\includegraphics[width=\textwidth]{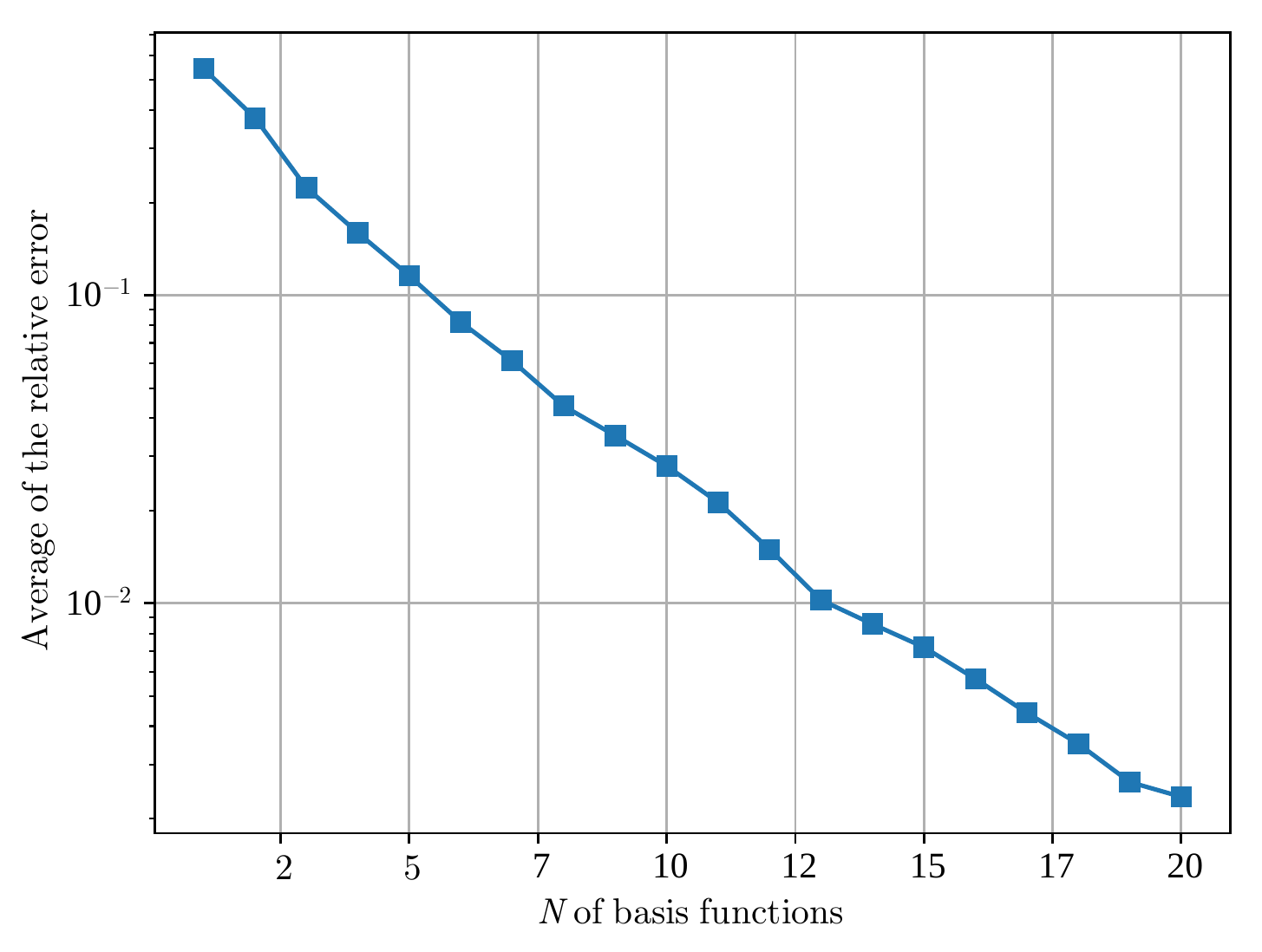}
\end{minipage}
\caption{Convergence analsys of the numerical example. In the left plot we can see the average value of the L2 relative error between the exact function $g$ and its EIM approximation. On the right plot we report the average value of the L2 relative error between the FOM the temparature field and the ROM temperature field. The plot is for different numbers of basis functions used to approximate the temperature field and keeping constant the number of basis functions used to approximate the forcing term ($N=11$)}\label{fig:par_function_decay}
\end{figure}

\subsubsection{An example in the context of reduced order models with non-linearity and non-affine
parametric dependency}\label{subsubsec:ex2}
In this second illustrative example we show the application of the DEIM algorithm to the stationary parametrized Navier-Stokes equations. In the present case we have both non-linearity and non-affinity with respect to the input parameters. Both non-linearity and non-affinity have been tackled using the matrix version of the Discrete Empirical Interpolation Method. The computational domain is given by the unit square $\Omega = [0,1]^2$ and the physical problem is described by the well known Navier-Stokes equations:
\begin{equation}\label{eq:navstokes}
\begin{cases}
 \dive(\bm{u} \otimes \bm{u})- \dive( 2 \nu(\mupar) \bm{\nabla^s} \bm{u} ) =-\bm{\nabla}p, &\mbox{ in } \Omega,\\
\dive \bm{u}=\bm{0}, &\mbox{ in } \Omega,\\
\bm{u} (x) = (1,0), &\mbox{ on } \Gamma_{TOP},\\
\bm{u} (x) = \bm{0}, &\mbox{ on } \Gamma_{0}. 
\end{cases}
\end{equation}
The physical problem is the classical benchmark of the lid-driven cavity problem with a parametrized diffusivity constant $\nu(\mupar)$. In this case the impossibility of recovering an affine decomposition of the differential operators is given by the convective term, which is by nature a nonlinear term, and by the parametrized diffusion term. The diffusivity constant $\nu(\mupar)$ has in fact been parametrized by the following non-linear function:
\begin{equation}
\nu(\bm x;\mupar) = \frac{e^{2({-2(x_1 - \mu_1 - 0.5)^2 - 2(x_2 - \mu_2 - 0.5)^2}) } } {100} + 0.01,
\end{equation}
the above function is a Gaussian function and the position of its center has been parametrized using the parameter vector $\mupar = (\mu_1,\mu_2)$. 
For the particular case, the discretized algebraic version of the continuous formulation can be rewritten as:
\begin{equation}
\left(
\begin{matrix}
\bm C(\bm u) + \bm A(\mupar) & \bm B^T \\
\bm B & 0
\end{matrix}
\right)
\left(
\begin{matrix}
\bm u\\
p
\end{matrix}
\right)
=
\left(
\begin{matrix}
\bm f\\
0
\end{matrix}
\right).
\end{equation}
The matrix $\bm A(\mupar)$ represents the discretized diffusion operator, the matrix $\bm C(\bm u)$ represents the discretized non-linear convective operator while the term $\bm B$ represents the divergence operator. The term $\bm A(\mupar)$ is characterized by a non-affine parametric dependency while the term $\bm C(\bm u)$ is characterized by non-linearity with respect to the solution. The velocity and pressure fields are approximated as:
\begin{equation}
\bm u (\mupar) \approx \sum_{q=1}^{N_u} c_q^u(\mupar) \bm h_q^u, \quad p (\mupar) \approx \sum_{q=1}^{N_p} c_q^p(\mupar) \bm h_q^p,
\end{equation}
and, in order to achieve an efficient offline-online splitting, the discretized operators are approximated by the matrix version of the DEIM algorithm and expressed as:
\begin{equation}
\bm A (\mupar) \approx \sum_{q=1}^{N_A} c_q^A(\mupar) \bm h_q^A, \quad \bm C (\bm u) \approx \sum_{q=1}^{N_C} c_q^C(\bm c_u) \bm h_q^C.
\end{equation}
The problem is discretized using the finite volume method and a staggered cartesian grid made of $20\times 20$ cell-centered finite volume elements. The DEIM algorithm has been implemented using $100$ samples chosen randomly inside the training space $\mathcal P_{\mathrm{EIM}}^{\mathrm{train}} \in [-0.5,0.5]^2$. The magic points necessary for the implementation of the DEIM algorithm are chosen to be coincident to the cell centers of the discretized problem. The basis functions $\bm h_q^A$ and $\bm h_C^A$ are obtained using the DEIM algorithm applied on the vectorized version of the discretized differential operator snapshots computed during the training stage $\bm{S}_A = [\mathrm{vec}(\bm A_1), \dots, \mathrm{vec}(\bm A_{M})]$ and $\bm{S}_C = [\mathrm{vec}(\bm C_1), \dots, \mathrm{vec}(\bm C_{M})]$. The snapshot matrices $\bm S_A$ and $\bm S_C$ contain in fact the discretized differential operators in vector form obtained for the different samples of the training set.
\begin{figure}
\begin{minipage}{\textwidth}
\centering
\begin{minipage}{0.24\textwidth}
\centering
\includegraphics[width=\textwidth]{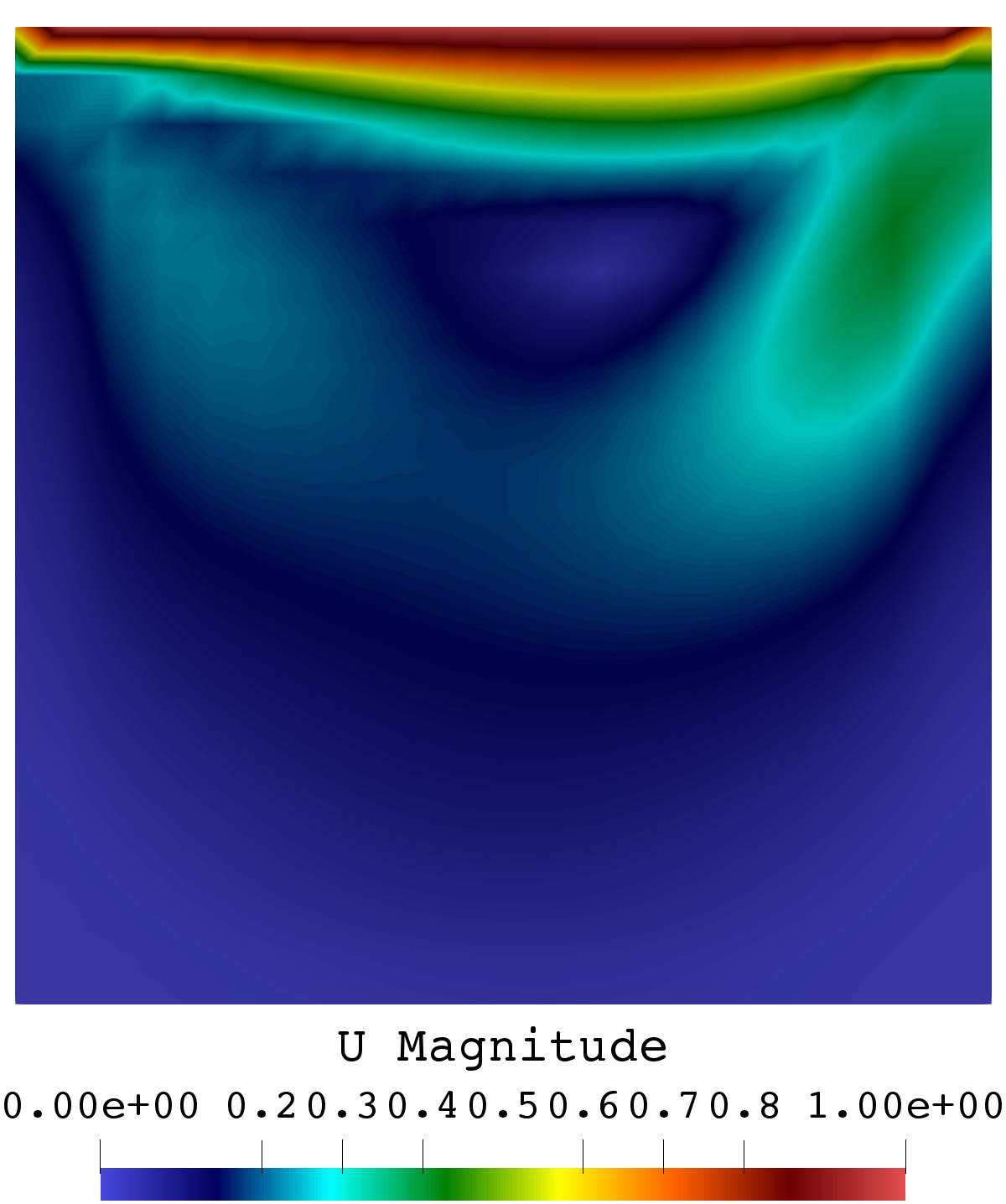}\\
\footnotesize(a)
\end{minipage}
\begin{minipage}{0.24\textwidth}
\centering
\includegraphics[width=\textwidth]{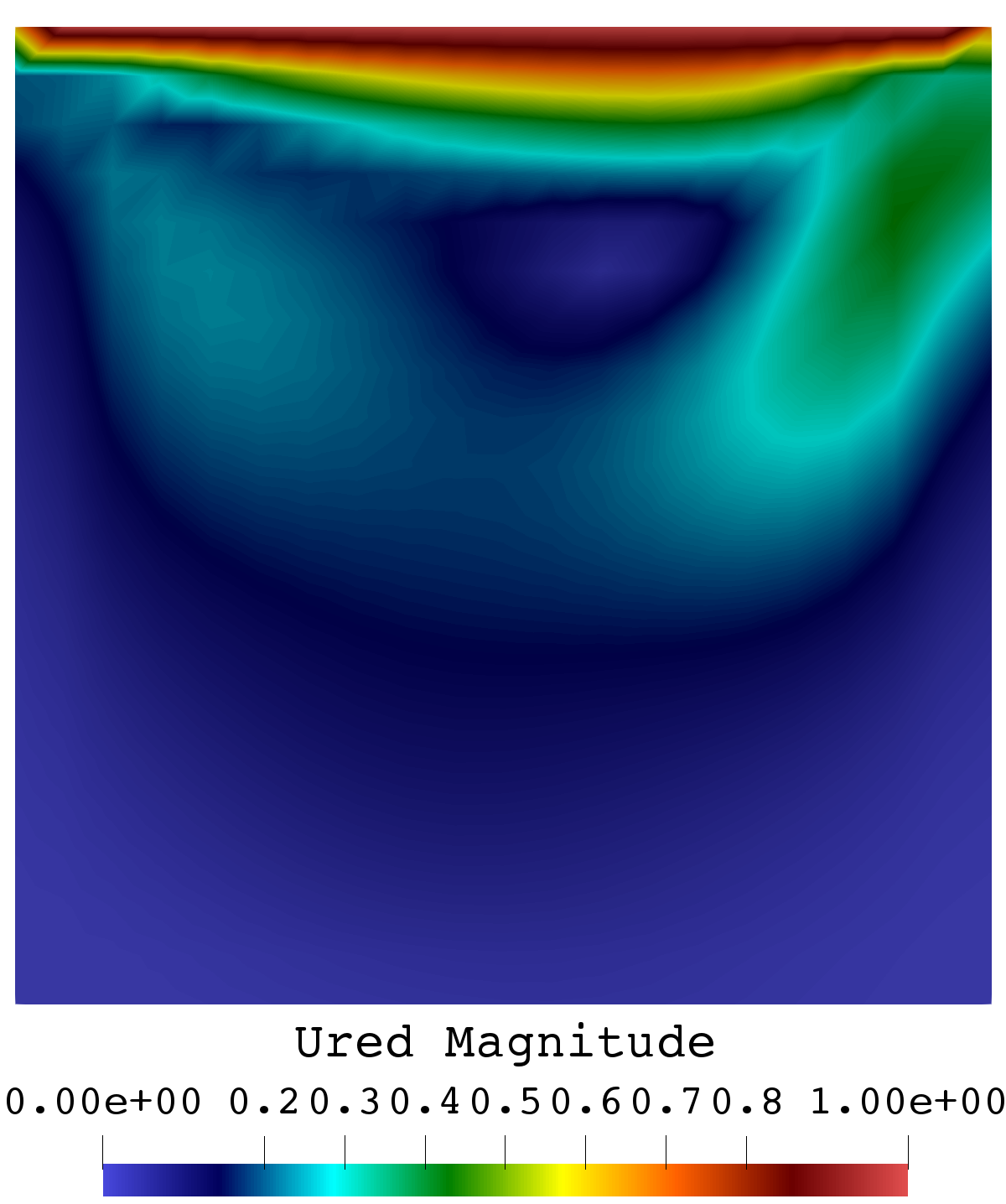}\\
\footnotesize(b)
\end{minipage}
\begin{minipage}{0.24\textwidth}
\centering
\includegraphics[width=\textwidth]{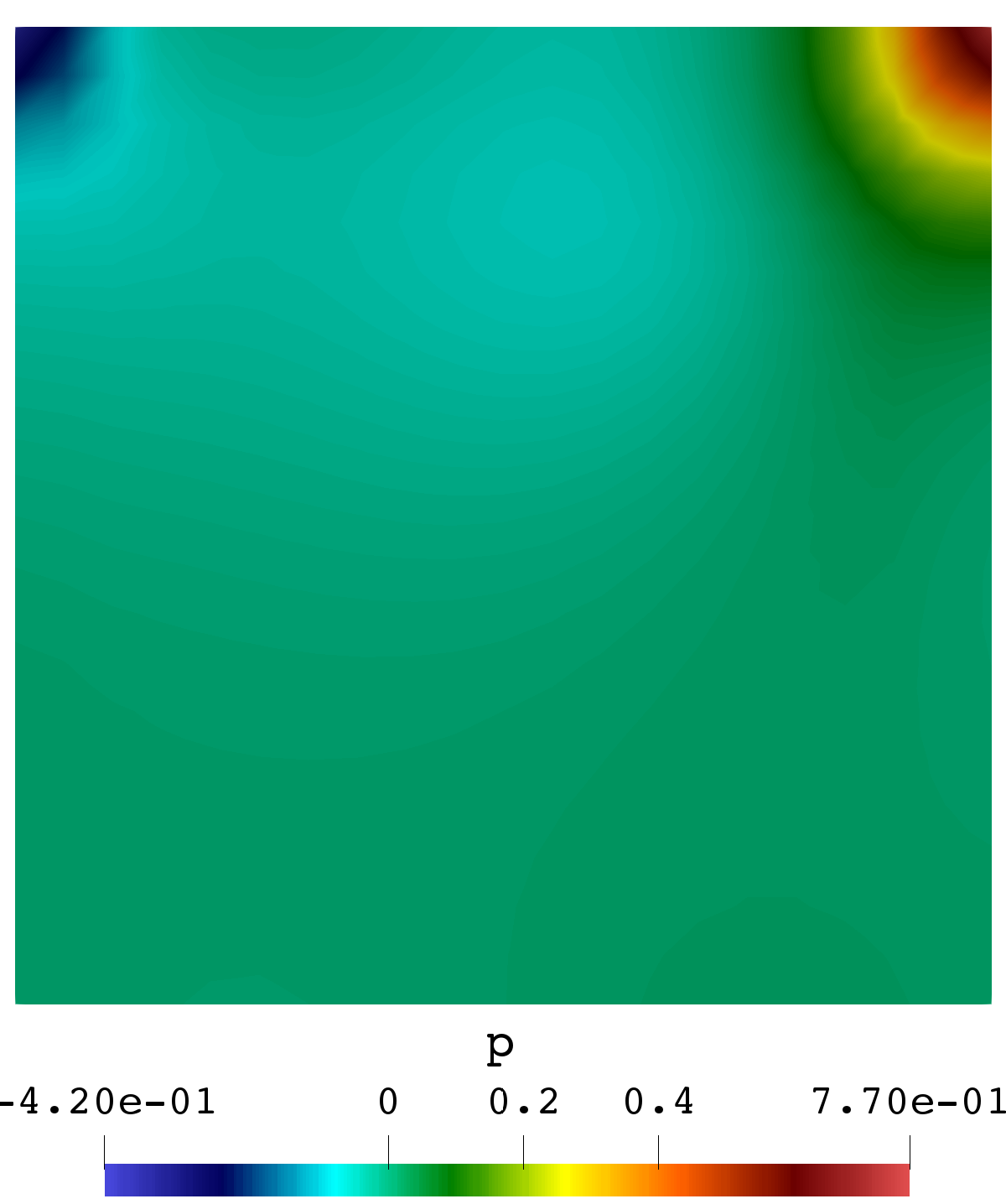}\\
\footnotesize(c)
\end{minipage}
\begin{minipage}{0.24\textwidth}
\centering
\includegraphics[width=\textwidth]{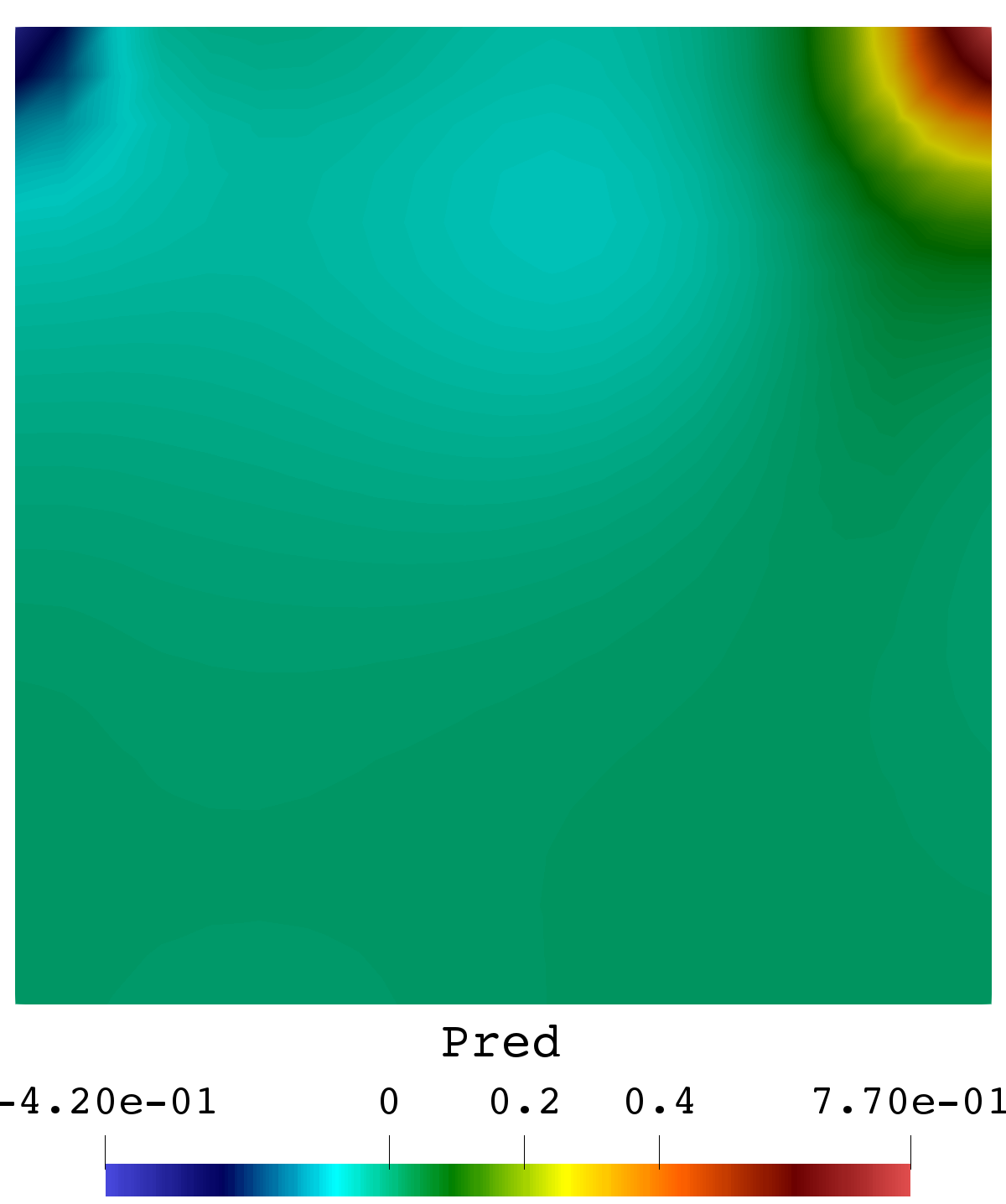}\\
\footnotesize(d)
\end{minipage}
\end{minipage}
\caption{Comparison between the FOM velocity (a) and pressure (c) fields and the ROM velocity (b) and pressure (d) fields. The plots are reported for one selected sample value inside the testing set. The reduced order model solutions have been computed using $14$ basis functions for the velocity space, $10$ for the pressure space, $10$ DEIM basis functions for the convective matrix  $\bm C$ and $10$ DEIM basis functions for the diffusion matrix $\bm B$.
}\label{fig:DEIM_comp}
\end{figure}

In Figure \ref{fig:DEIM_comp} we report the comparison of the Full Order Model fields and the reduced order model ones; the comparison is depicted for a parameter sample not used to train the ROM. On the right side of Figure \ref{fig:mdeim} we report the convergence analysis for the numerical example. The plots are performed testing the reduced order model on $100$ additional sample values selected randomly inside the parameter space $\mathcal P_{\mathrm{EIM}}^{\mathrm{test}} \in [-0.5,0.5]^2$. In the plots is reported the average value over the testing space of the $L^2$ relative error. 

\begin{figure}[h]
\begin{minipage}{0.5\textwidth}
\includegraphics[width=\textwidth]{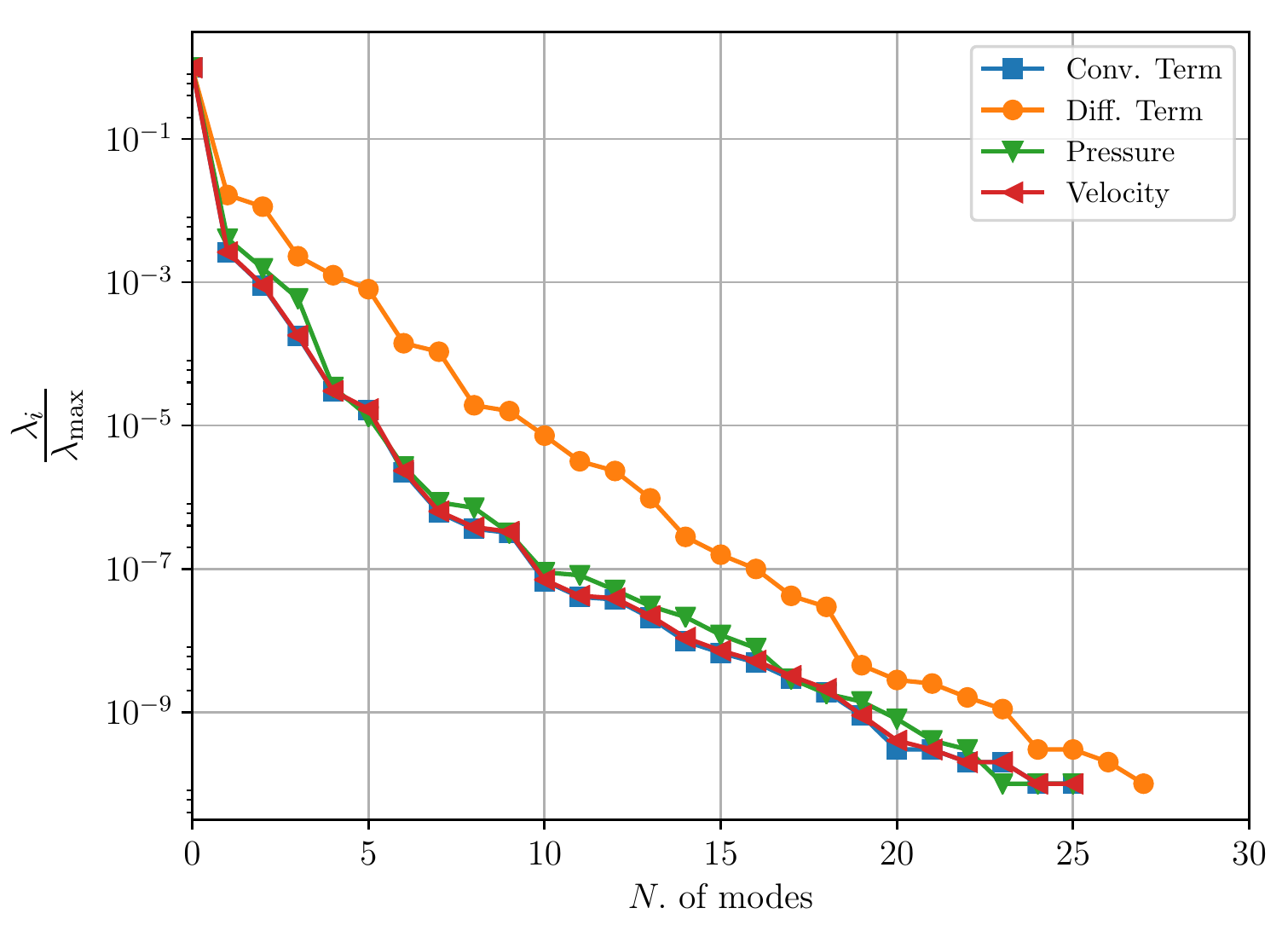}
\end{minipage}
\begin{minipage}{0.5\textwidth}
\includegraphics[width=\textwidth]{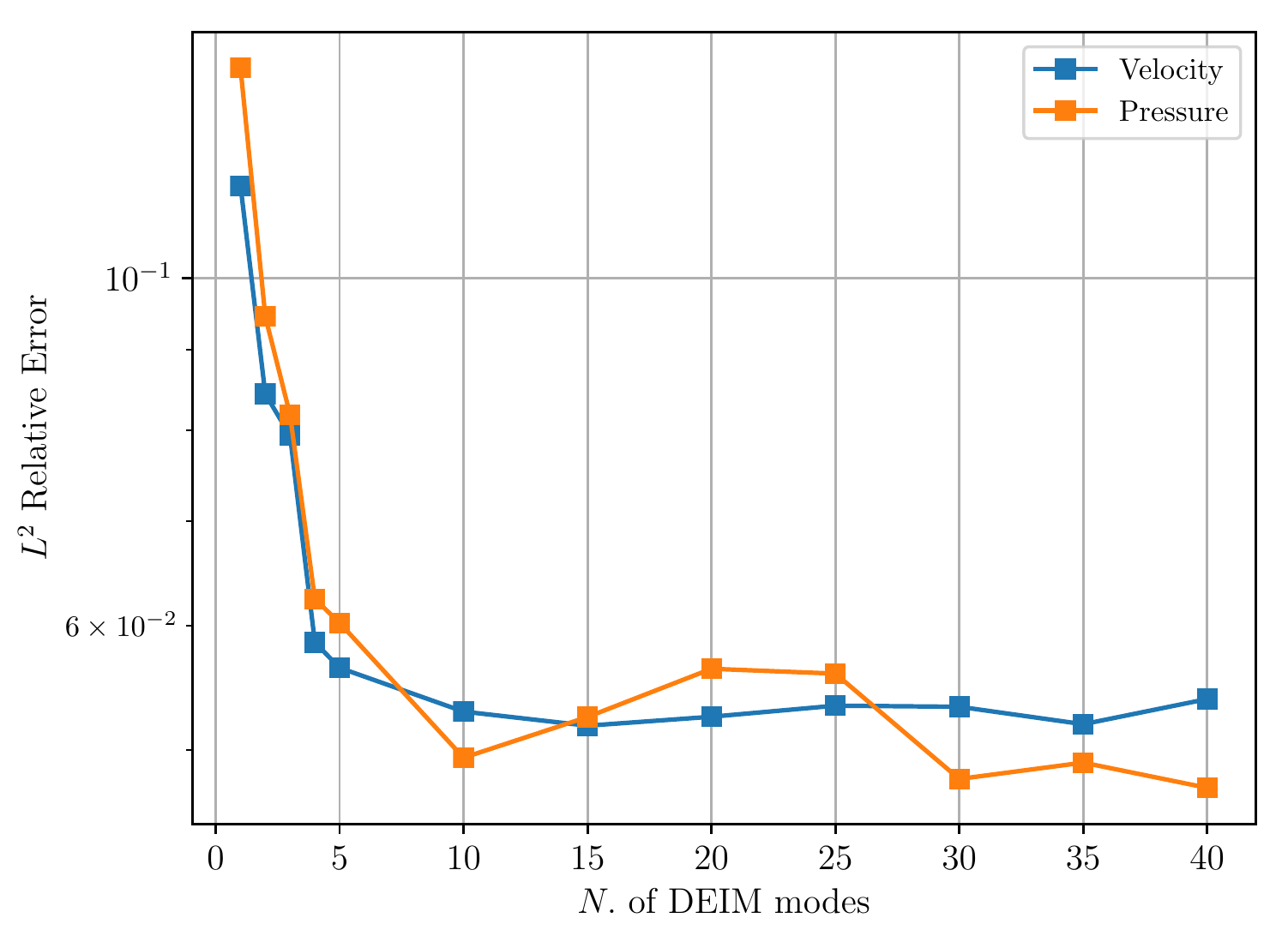}
\end{minipage}
\caption{Eigenvalue Decay of the POD procedure during the DEIM algorithm (left plot). The convergence analysis with respect to the number of DEIM basis function (right plot), which is computed using the average value over the testing set of the $L^2$ relative error, has been performed keeping constant the number of basis functions used to approximate the velocity and pressure fields ($N_u=14, N_p=10$) and changing the number of DEIM basis functions used to approximate the convective and diffusion terms ($N_C=N_A$).}\label{fig:mdeim}
\end{figure}

\section{Advanced tools: reduction in parameter spaces}
\label{sec:active}

Often the use of the aforementioned geometrical morphing techniques in
Section~\ref{sec:geom} does not tell us how many control points,
i.e. geometrical parameters, are enough to conduct a proper
analysis. This leads to self imposing too few parameters in order to
avoid the curse of dimensionality and dealing with intractable
problems. To overcome this issue there exist techniques for
parameter space dimensionality reduction, both linear and
nonlinear. In particular we present here the active subspaces
property for linear dimensionality reduction, while in the last
section we show an overview of possible nonlinear methods. 

These methods have to be intended as general tools, not restricted to
parametrised PDEs. Moreover the nature of the parameter space can be
very diverse, including both geometrical and physical parameters. They
are data-driven tools working with couples of input/output data, and
they can be used to enhance other model order reduction techniques.

\subsection{Active subspaces property and its applications}
In this and the following sections we present the active subspaces
(AS) property proposed by Trent Russi~\cite{russi2010uncertainty} and
developed by Paul Constantine~\cite{constantine2015active}. In brief,
active subspaces are defined as the leading eigenspaces of the second
moment matrix of the function's gradient and constitute a global
sensitivity index.

We present how to exploit AS to reduce the
parameter space dimensionality, and use it as a powerful
pre-processing tool. Moreover we show how to combine it with a model
reduction methodology and present its application to a cardiovascular
problem. In particular, after identifying a lower dimensional
parameter subspace, we sample it to apply further model order
reduction methods. This results in improved computational efficiency.

The main characteristic of AS is the fact that it uses information of
both the output function of interest, and the input parameter space in
order to reduce its dimensionality. 
The active subspaces have been successfully employed in many
engineering fields. We cite, among others, applications in
magnetohydrodynamics power generation model
in~\cite{glaws2017dimension}, in naval engineering for the computation
of the total drag resistance with both geometrical and physical
parameters in~\cite{tezzele2018dimension,demo2018isope}, and for
constrained shape optimization~\cite{lukaczyk2014active} using the concept of shared active
subspaces in~\cite{tezzele2018model}. There are also applications to
turbomachinery in~\cite{bahamonde2017active}, to uncertainty
quantification in the numerical simulation of a scramjet
in~\cite{constantine2015exploiting}, and to accelerate Markov chain
Monte Carlo in~\cite{constantine2016accelerating}.
Extension of active subspace discovery for time-dependent processes
and application to a lithium ion battery model can be found
in~\cite{constantine2017time}. A multifidelity approach to reduce
the cost of performing dimension reduction through the computation of
the active subspace matrix is presented in~\cite{lam2018multifidelity}.
In~\cite{eriksson2018scaling} they exploit AS for Bayesian
optimization, while the coupling with reduced order methods can be
found in~\cite{demo2019cras}, for a non-intrusive data-driven approach,
and the coupling with POD-Galerkin methods for biomedical engineering
will be presented in Section~\ref{sec:as-pod} following~\cite{tezzele2018combined}.

\subsection{Active subspaces definition}
Given a parametric scalar function $f(\mupar) : \mathbb{R}^p
\rightarrow \mathbb{R}$, where $p$ is the number of parameters,
representing the output of interest, and given a probability density
function $\rho: \mathbb{R}^p \rightarrow \mathbb{R}^+$ that represents
uncertainty in the model inputs, active
subspaces are low dimensional subspaces of the input space where $f$
varies the most on average. It is a property of the of the pair $(f,
\rho)$ (see~\cite{constantine2015active}). In order to uncover AS we exploit
the gradients of the function with respect to the input parameters, so
it can be viewed as a derivative-based sensitivity analysis that
unveils low dimensional parametrization of $f$ using some linear
combinations of the original parameters. Roughly speaking, after a
rescaling of the input parameter space to the hypercube $[-1, 1]^p$,
we rotate it until the lower rank approximation of the output of
interest is discovered, that means a preferred direction in the input
space is identified. Then we can project all the data onto the orthogonal
space of this preferred direction and we can construct a surrogate model on
this low dimensional space.

Let us add some hypotheses to $f$ in order to proper construct the
matrix we will use to find the active subspaces: let $f$ be continuous and
differentiable with square-integrable partial derivatives in the
support of $\rho$. We define the so-called uncentered covariance
matrix $\mathbf{C}$ of the gradients of $f$, as the matrix whose
elements are the average products of partial derivatives of the map
$f$, that~is 
\begin{equation}
\label{eq:uncentered}
\mathbf{C} = \mathbb{E}\, [\nabla_{\mupar} f \, \nabla_{\mupar} f
^T] =\int (\nabla_{\mupar} f) ( \nabla_{\mupar} f )^T
\rho \, d \mupar,
\end{equation}
where $\mathbb{E}$ is the expected value, and $\nabla_{\mupar} f =
\nabla f(\mupar) = \left [ \frac{\partial f}{\partial \mu_1}, \dots,
  \frac{\partial f}{\partial \mu_p} \right ]^T$ is the column vector
of partial derivatives of $f$. This matrix is symmetric so it has a
real eigenvalue decomposition:
\begin{equation}
\mathbf{C} = \mathbf{W} \mathbf{\Lambda} \mathbf{W}^T,
\end{equation}
where $\mathbf{W} \in \mathbb{R}^{p \times p}$ is the
orthogonal matrix of eigenvectors, and $\mathbf{\Lambda}$ is the
diagonal matrix of non-negative eigenvalues arranged in descending
order. The eigenpairs of the uncentered covariance matrix define the
active subspaces of the pair $(f, \rho)$. Moreover Lemma 2.1
in~\cite{constantine2014active} states that the eigenpairs are
functionals of $f(\mupar)$ and it holds
\begin{equation}
\lambda_i = \mathbf{w}_i^T \mathbf{C} \mathbf{w}_i = \int
(\nabla_{\mupar} f^T \mathbf{w}_i)^2 \rho \, d \mupar, 
\end{equation}
that means that the $i$-th eigenvalue is the average squared directional
derivative of $f$ along the eigenvector $\mathbf{w}_i$. Alternatively
we can say that the eigenvalues represent the magnitude of the
variance of $\nabla_{\mupar} f$ along their eigenvectors orientation. So small
values of the eigenvalues correspond to small perturbation of $f$
along the corresponding eigenvectors. It also follows that large gaps
between eigenvalues indicate directions where $f$ changes the most on
average. Since we seek for a lower dimensional space of dimension $M < p$ where
the target function has exactly this property, we define the active
subspace of dimension $M$ as the span of the first $M$ eigenvectors
(they correspond to the most energetic eigenvalues before a gap). Let us partition
$\mathbf{\Lambda}$ and $\mathbf{W}$ as
\begin{equation}
\mathbf{\Lambda} =   \begin{bmatrix} \mathbf{\Lambda}_1 & \\
                                     &
                                     \mathbf{\Lambda}_2\end{bmatrix},
\qquad
\mathbf{W} = \left [ \mathbf{W}_1 \quad \mathbf{W}_2 \right ],
\end{equation}
where $\mathbf{\Lambda}_1 = \text{diag}(\lambda_1, \dots, \lambda_M)$, and
$\mathbf{W}_1$ contains the first $M$ eigenvectors. We can use
$\mathbf{W}_1$ to project the original parameters to the active
subspace obtaining the reduced parameters, that is the input space is
geometrically transformed and aligned with $\mathbf{W}_1$, in order to
retain only the directions where the function variability is high. We
call active variable $\mupar_M$ the range of $\mathbf{W}_1^T$, and
inactive variable $\etapar$ the range of $\mathbf{W}_2^T$:
\begin{equation}
\label{eq:active_var}
\mupar_M = \mathbf{W}_1^T\mupar \in \mathbb{R}^M, \qquad
\etapar = \mathbf{W}_2^T \mupar \in \mathbb{R}^{p - M} .
\end{equation}
We can thus express any point in the parameter space $\mupar \in
\mathbb{R}^p$ in terms of $\mupar_M$ and $\etapar$ as
\begin{equation}
\mupar = \mathbf{W}\mathbf{W}^T\mupar =
\mathbf{W}_1\mathbf{W}_1^T\mupar +
\mathbf{W}_2\mathbf{W}_2^T\mupar = \mathbf{W}_1 \mupar_M +
\mathbf{W}_2 \etapar.
\end{equation}

The lower-dimension approximation, or surrogate quantity of interest,
$g : \mathbb{R}^M \rightarrow \mathbb{R}$ of the target function $f$,
is a function of only the active variable $\mupar_M$ as
\begin{equation}
  f (\mupar) \approx g (\mathbf{W}_1^T \mupar) = g(\mupar_M).
\end{equation}
Such $g$ is called ridge function (see~\cite{pinkus2015ridge}) and, as we can
infer from this section, it is constant along the span of $\mathbf{W}_2$. 

From a practical point of view Eq.~\eqref{eq:uncentered} is estimated
through Monte Carlo method. We draw $N_{\text{train}}$ independent
samples $\mupar^{(i)}$ according to the measure $\rho$ and we
approximate 
\begin{equation}
\label{eq:covariance_approx}
\mathbf{C} \approx \hat{\mathbf{C}} = \frac{1}{N_{\text{train}}} \sum_{i=1}^{N_{\text{train}}} \nabla_{\mupar} f_i \,
\nabla_{\mupar} f^T_i  = \hat{\mathbf{W}} \hat{\mathbf{\Lambda}} \hat{\mathbf{W}}^T,
\end{equation}
where $\nabla_{\mupar} f_i = \nabla_{\mupar}
f(\mupar^{(i)})$. In~\cite{constantine2015active} they provide an
heuristic formula for the number of samples $N_{\text{train}}$ needed to properly
estimate the first $k$ eigenvalues, that is
\begin{equation}
N_{\text{train}} = \alpha \, k \ln(p),
\end{equation}
where $\alpha$ usually is between 2 and 10.
Moreover they prove that for sufficiently large $N_{\text{train}}$ the
error $\varepsilon$ committed in the approximation of the active
subspace of dimension $n$ is bounded from above by
\begin{equation}
  \varepsilon = \text{dist}(\text{rank}(\mathbf{W}_1),
  \text{rank}(\hat{\mathbf{W}}_1)) \leq \frac{4 \lambda_1
    \delta}{\lambda_n - \lambda_{n+1}},
\end{equation}
where $\delta$ is a positive scalar bounded from above by
$\frac{\lambda_n - \lambda_{n+1}}{5 \lambda_1}$. Here we can clearly
see how the gap between two eigenvalues is important in order to
properly approximate $f$ exploiting AS.\\

\subsection{Some examples}
\label{chap_1_4:some_examples}

In this section we present two simple examples with the computation of
the active subspaces using analytical gradients. To highlight the
possibility that the presence of an active subspace is not alway
guaranteed we also show an example in this direction. We choose
for both the cases a tridimensional input parameter space without loss
of generality. In order to identify the low-dimensional structure of
the function of interest we use the sufficient summary plots,
developed in~\cite{cook2009regression}. In our cases, they are 
scatter plots of $f(\mupar)$ against the active variable
$\mupar_M$. 

The presence of an active subspace is not always guaranteed. For
example every target function that has a radial symmetry has not a
lower dimensional representation in terms of active variable. This is
due to the fact that there is no rotation of the input parameter space
that aligns it along a preferred direction since all of them are
equally important.

Let us consider for example the function $f (\mupar) = \frac{1}{2} \mupar^T
\mupar$ representing an $n$-dimensional elliptic paraboloid, where the
parameter $\mupar$ is a column vector in $[-1, 1]^3$. In this case we
have the exact derivatives, in fact $\nabla_{\mupar} f = \mupar$ and
we do not have to approximate them. If we draw 1000 samples and we
apply the procedure to find an active subspace and we plot the
sufficient summary plot in one dimension, as in
Figure~\ref{fig:active_not_working}, we clearly see how it is unable
to find the active variable along which $f$ varies the most on
average. In fact there is not a significant gap between the
eigenvalues, since we have that $\mathbf{C} = \frac{1}{3}
\mathbf{Id}$. Moreover the projection of the data onto the inactive
subspace suggest us the presence of an $n$-dimensional elliptic
paraboloid.

\begin{figure}[h!]
\centering
\includegraphics[width=0.457\textwidth]{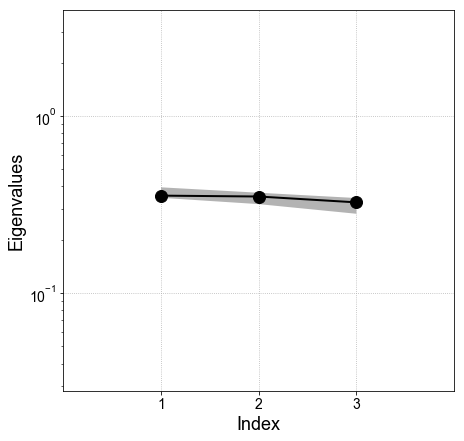}\hfill
\includegraphics[width=0.45\textwidth]{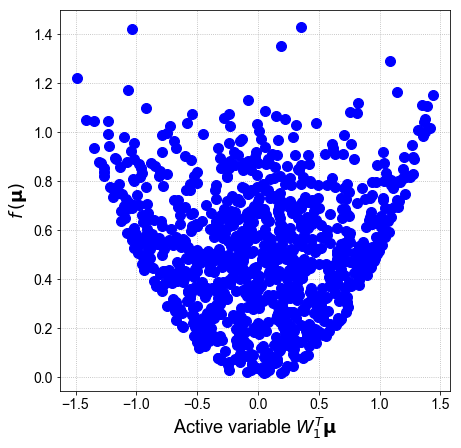}
\caption{Example of an output function with a radial symmetry. On the
  left the exact eigenvalues of the uncentered covariance matrix. On
  the right the   sufficient summary plot in one dimension ($f(\mupar)$ against
  $\mupar_M = \mathbf{W}_1^T\mupar$) shows how the projection of
  the data along the inactive directions does not unveil a lower
  dimensional structure for $f$.}
\label{fig:active_not_working}
\end{figure}

Let us consider now another quadratic function in 3 variables. We define the
output of interest $f$ as
\begin{equation}
f (\mupar) = \frac{1}{2} \mupar^T \mathbf{A} \mupar, 
\end{equation}
where $\mupar \in [-1, 1]^3$, and $\mathbf{A}$ is symmetric positive
definite with a major gap between the first and the second
eigenvalue. With this form we can compute the exact gradients as
$\nabla_{\mupar} f (\mupar) = \mathbf{A} \mupar$ and, taking $\rho$ as
a uniform density function, compute $\mathbf{C}$ as
\begin{equation}
\mathbf{C} = \mathbf{A} \left ( \int \mupar \mupar^T \rho \, d \mupar
\right ) \mathbf{A}^T = \frac{1}{3} \mathbf{A}^2.
\end{equation}
So the squared eigenvalues of $\mathbf{A}$ are the eigenvalues of
$\mathbf{C}$. Since, by definition, $\mathbf{A}$ has a
significant gap between the first and the second eigenvalues we can
easily find an active subspace of dimension one.  

In Figure~\ref{fig:active_quadratic} we show the sufficient summary
plot of $f$ with respect to its active variable. A clear univariate
behavior is present, as expected, so we can easily construct $g$, for
instance taking a quadratic unidimensional function. We can also see the associated
eigenvalues of the uncentered covariance matrix.

\begin{figure}[h!]
\centering
\includegraphics[width=0.457\textwidth]{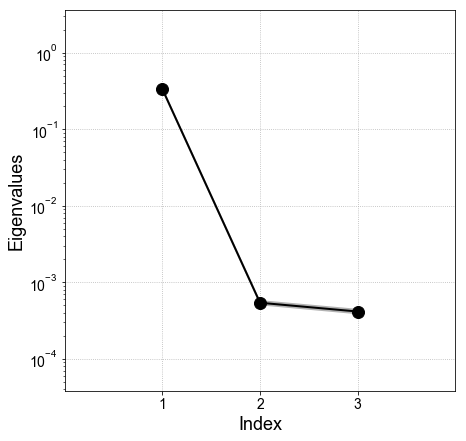}\hfill
\includegraphics[width=0.45\textwidth]{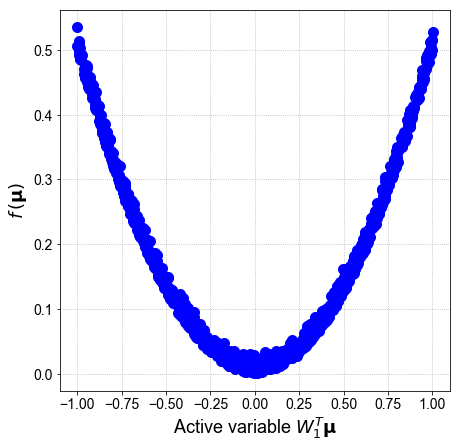}
\caption{Example of a quadratic function with an active subspace of
  dimension one. On the left the exact eigenvalues of the uncentered covariance matrix. On
  the right the sufficient summary plot in one dimension ($f(\mupar)$ against
  $\mupar_M = \mathbf{W}_1^T\mupar$) shows how the projection of
  the data along the inactive directions unveils a univariate structure for $f$.}
\label{fig:active_quadratic}
\end{figure}

\subsection{Active subspaces as pre-processing tool to enhance model reduction}
\label{sec:as-pod}
The presence of an active subspace for an output of interest, derived
from the solution of a parametric PDE, can be exploited for further
model order reduction. Thus, in this context, AS can be seen as a
powerful pre-processing technique to both reduce the parameter space
dimensionality and boost the performance of other model order
reduction methods. 

In~\cite{tezzele2018combined} the active
subspace for a relative pressure drop in a stenosed carotid artery is
used as a reduced sampling space to improve the reconstruction of
the output manifold.
We used as parameters the displacement of a selection radial basis functions (RBF)
control points to simulate the occlusion of the carotid artery after
the bifurcation. For a review of RBF interpolation technique
see Section~\ref{sec:rbf}. In Figure~\ref{fig:carotid_def} two
different views of the same carotid and the control points highlighted with green
dots. The target function was a relative pressure drop between the two
branches computed solving a stationary Navier-Stokes problem. 

After the identification of the active subspace we exploit it by sampling
the original full parameter space along the active subspace. These
sampled parameters were used, in the offline phase, to construct the
snapshots matrix for the training of a ROM. This leads to better
approximation properties for a given number of snapshots with respect
to usual sampling techniques. The natural construction of the
uncentered covariance matrix, which uses information from both the
inputs and the outputs is the reason of such improvements.


The same idea has been coupled also with non-intrusive model order reduction
techniques, such as proper orthogonal decomposition with interpolation
(PODI), in~\cite{tezzele2019shape}, while for the reconstruction of modal
coefficient using PODI with AS for low computational budget we
suggest~\cite{demo2019cras}.

\subsection{About nonlinear dimensionality reduction}
\label{sec:nonlinear}

There are plenty of other techniques that reduce the dimensionality of
a given dataset. They do not exploit simultaneously the structure of
the output function and the input parameter space like AS, they just
express the datavectors we want to reduce in a reduced space embedded
in the original one. For a comprehensive overview see~\cite{lee2007nonlinear}
and~\cite{van2009dimensionality}. The main assumption is that the
dataset at hand has an intrinsic dimensionality, which is lower than the full space
where they belong. This means that the data are lying on or near a
manifold with dimensionality $d$ embedded in a greater space of
dimension $D$. If we approximate this manifold with a linear subspace
we use a linear dimensionality reduction technique, otherwise assuming
the data lie on a curved manifold we can achieve better results using
a nonlinear method. Unfortunately in general neither the
characteristics of the manifold, nor the intrinsic dimensionality are
known, so the dimensionality reduction problem is ill-posed. There are
several algorithms to detect the intrinsic dimensionality of a
dataset, we suggest the review in~\cite{camastra2003data}.
Among all we cite two of the most popular techniques, which are locally
linear embedding (LLE) presented in~\cite{roweis2000nonlinear}, and
Isomap~\cite{tenenbaum2000global}. Extensions for the two
methods can be found in~\cite{bengio2004out}.

LLE seeks to preserve local
properties of the high dimensional data in the embedded space, and it
is able to detect non-convex manifolds. In particular it preserves
local reconstruction weights of the neighborhood graph, that is LLE
fits a hyperplane through each data point and its nearest
neighbors. Some applications can be found
in~\cite{gonzalez2016computational} for biomedical engineering, or
in~\cite{ibanez2018manifold} for computational mechanics.

Isomap instead seeks to preserve geodesic (or curvilinear) distances
between the high dimensional data points and the lower dimensional
embedded ones. Its topological stability has been investigated
in~\cite{balasubramanian2002isomap}, while it has been used in for
micromotility reconstruction in~\cite{arroyo2012reverse}. 

Other approaches include for example a manifold walking algorithm that
has been proposed in~\cite{meng2015identification} and
in~\cite{meng2018nonlinear}.

\section{Conclusion and Outlook}
\label{chap1:sec:conclude}

This introductory chapter provided the means to understand projection based MOR methods in section~\ref{sec:basicNotions}. 
Various techniques allowing the parametrization of complicated geometries are provided in section~\ref{sec:geom}.
Since many geometries of interest introduce non-linearities or non-affine parameter dependency an intermediate step such as the 
\emph{empirical interpolation method} is often applied. The basics were presented in section~\ref{sec:nonaffinity} and will be used 
further in the following chapters. 
The reduction in parameter space becomes necessary if high-dimensional parameter space are considered. 
Active subspaces (see section~\ref{sec:active}) provide a mean to tackle the curse of dimensionality.

Each chapter of the handbook gives an in-depth technical details upon a particular topic of interest.
This includes common MOR methods, several application areas of interest and a survey of current software frameworks for model reduction.
Whenever a method does not rely only on the PDE-based functional analysis setting introduced in this chapter, corresponding requirements will be mentoined within each technical chapter.

\begin{acknowledgement}
  We are grateful to EU-COST European Union Cooperation in Science
  and Technology, section EU-MORNET Model Reduction Network, TD 1307
  for pushing us into this initiative. This work is
  supported by European Union Funding for Research and Innovation ---
  Horizon 2020 Program --- in the framework of European Research
  Council Executive Agency: H2020 ERC CoG 2015 AROMA-CFD project
  681447 ``Advanced Reduced Order Methods with Applications in
  Computational Fluid Dynamics'' P.I. Gianluigi Rozza.
\end{acknowledgement}

\bibliographystyle{dgruyter}

\bibliography{biblio.bib}

\begin{thebibliography}{111}
\newcommand{\enquote}[1]{``#1''}
\providecommand{\natexlab}[1]{#1}
\providecommand{\url}[1]{\texttt{#1}}
\providecommand{\urlprefix}{URL }

\bibitem[{Adams(1975)}]{Adams1075}
Adams, R.~A. (1975): \emph{Sobolev Spaces}, Academic Press, New York.

\bibitem[{Ainsworth and Oden(1997)}]{AINSWORTH19971}
Ainsworth, M. and J.~Oden (1997): \enquote{A posteriori error estimation in
  finite element analysis,} \emph{Computer Methods in Applied Mechanics and
  Engineering}, 142, 1 -- 88.

\bibitem[{Almroth et~al.(1978)Almroth, Stern, and Brogan}]{Almroth1978}
Almroth, B., P.~Stern, and F.~Brogan (1978): \enquote{Automatic choice of
  global shape functions in structural analysis,} \emph{AIAA J.}, 16, 525--528.

\bibitem[{Arnold et~al.(2002)Arnold, Brezzi, Cockburn, and
  Marini}]{ArnoldBrezzi2002}
Arnold, D., F.~Brezzi, B.~Cockburn, and L.~Marini (2002): \enquote{Unified
  analysis of {D}iscontinuous {G}alerkin methods for elliptic problems,}
  \emph{SIAM Journal on Numerical Analysis}, 39, 1749--1779.

\bibitem[{Arroyo et~al.(2012)Arroyo, Heltai, Mill{\'a}n, and
  DeSimone}]{arroyo2012reverse}
Arroyo, M., L.~Heltai, D.~Mill{\'a}n, and A.~DeSimone (2012): \enquote{Reverse
  engineering the euglenoid movement,} \emph{Proceedings of the National
  Academy of Sciences}, 109, 17874--17879.

\bibitem[{Babuska(1970/71)}]{BABUSKA1970/71}
Babuska, I. (1970/71): \enquote{Error-bounds for finite element method,}
  \emph{Numerische Mathematik}, 16, 322--333.

\bibitem[{Bahamonde et~al.(2017)Bahamonde, Pini, De~Servi, and
  Colonna}]{bahamonde2017active}
Bahamonde, S., M.~Pini, C.~De~Servi, and P.~Colonna (2017): \enquote{Active
  subspaces for the optimal meanline design of unconventional turbomachinery,}
  \emph{Applied Thermal Engineering}, 127, 1108--1118.

\bibitem[{Balasubramanian and Schwartz(2002)}]{balasubramanian2002isomap}
Balasubramanian, M. and E.~L. Schwartz (2002): \enquote{The isomap algorithm
  and topological stability,} \emph{Science}, 295, 7--7.

\bibitem[{Ballarin et~al.(in press, 2018)Ballarin, D'Amario, Perotto, and
  Rozza}]{BallarinDAmarioPerottoRozza2017}
Ballarin, F., A.~D'Amario, S.~Perotto, and G.~Rozza (in press, 2018):
  \enquote{{A POD-Selective Inverse Distance Weighting method for fast
  parametrized shape morphing},} \emph{Int. J. Num. Meth. Eng.}

\bibitem[{Barrault et~al.(2004)Barrault, Maday, Nguyen, and
  Patera}]{Barrault2004}
Barrault, M., Y.~Maday, N.~C. Nguyen, and A.~T. Patera (2004): \enquote{An
  `empirical interpolation' method: application to efficient reduced-basis
  discretization of partial differential equations,} \emph{Comptes Rendus
  Mathematique}, 339, 667--672.

\bibitem[{Beckert and Wendland(2001)}]{beckert2001multivariate}
Beckert, A. and H.~Wendland (2001): \enquote{Multivariate interpolation for
  fluid-structure-interaction problems using radial basis functions,}
  \emph{Aerospace Science and Technology}, 5, 125--134.

\bibitem[{Bengio et~al.(2004)Bengio, Paiement, Vincent, Delalleau, Roux, and
  Ouimet}]{bengio2004out}
Bengio, Y., J.-f. Paiement, P.~Vincent, O.~Delalleau, N.~L. Roux, and M.~Ouimet
  (2004): \enquote{Out-of-sample extensions for lle, isomap, mds, eigenmaps,
  and spectral clustering,} in \emph{Advances in neural information processing
  systems}, 177--184.

\bibitem[{Boffi et~al.(2013)Boffi, Brezzi, and Fortin}]{Boffi2013}
Boffi, D., F.~Brezzi, and M.~Fortin (2013): \emph{Mixed Finite Element Methods
  and Applications}, Springer-Verlag Berlin Heidelberg.

\bibitem[{Bonomi et~al.(2017)Bonomi, Manzoni, and Quarteroni}]{Bonomi2017}
Bonomi, D., A.~Manzoni, and A.~Quarteroni (2017): \enquote{A matrix {DEIM}
  technique for model reduction of nonlinear parametrized problems in cardiac
  mechanics,} \emph{Computer Methods in Applied Mechanics and Engineering},
  324, 300--326.

\bibitem[{Buhmann(2003)}]{buhmann2003radial}
Buhmann, M.~D. (2003): \emph{Radial basis functions: theory and
  implementations}, volume~12, Cambridge university press.

\bibitem[{Bui-Thanh et~al.(2003)Bui-Thanh, Damodaran, and
  Willcox}]{BuiThanh2003}
Bui-Thanh, T., M.~Damodaran, and K.~Willcox (2003): \enquote{{P}roper
  {O}rthogonal {D}ecomposition extensions for parametric applications in
  compressible aerodynamics,} in \emph{21st {AIAA} Applied Aerodynamics
  Conference}, American Institute of Aeronautics and Astronautics.

\bibitem[{Camastra(2003)}]{camastra2003data}
Camastra, F. (2003): \enquote{Data dimensionality estimation methods: a
  survey,} \emph{Pattern recognition}, 36, 2945--2954.

\bibitem[{Canuto et~al.(2006)Canuto, Hussaini, Quarteroni, and Zhang}]{CHQZ1}
Canuto, C., M.~Y. Hussaini, A.~Quarteroni, and T.~Zhang (2006): \emph{Spectral
  Methods: Fundamentals in Single Domains}, Springer-Verlag Berlin Heidelberg.

\bibitem[{Canuto et~al.(2009)Canuto, Tonn, and Urban}]{Canuto2009}
Canuto, C., T.~Tonn, and K.~Urban (2009): \enquote{A posteriori error analysis
  of the reduced basis method for nonaffine parametrized nonlinear pdes,}
  \emph{SIAM Journal on Numerical Analysis}, 47, 2001--2022.

\bibitem[{Carlberg et~al.(2010)Carlberg, Bou-Mosleh, and Farhat}]{Carlberg2010}
Carlberg, K., C.~Bou-Mosleh, and C.~Farhat (2010): \enquote{Efficient
  non-linear model reduction via a least-squares {P}etrov-{G}alerkin projection
  and compressive tensor approximations,} \emph{International Journal for
  Numerical Methods in Engineering}, 86, 155--181.

\bibitem[{Casenave et~al.(2014)Casenave, Ern, and
  Leli{\`{e}}vre}]{Casenave_2014}
Casenave, F., A.~Ern, and T.~Leli{\`{e}}vre (2014): \enquote{A nonintrusive
  reduced basis method applied to aeroacoustic simulations,} \emph{Advances in
  Computational Mathematics}, 41, 961--986.

\bibitem[{Chaturantabut and Sorensen(2009)}]{Chaturantabut_2009}
Chaturantabut, S. and D.~C. Sorensen (2009): \enquote{Discrete empirical
  interpolation for nonlinear model reduction,} in \emph{Proceedings of the 48h
  {IEEE} Conference on Decision and Control ({CDC}) held jointly with 2009 28th
  Chinese Control Conference}, {IEEE}.

\bibitem[{Chaturantabut and Sorensen(2010)}]{Chaturantabut2010}
Chaturantabut, S. and D.~C. Sorensen (2010): \enquote{Nonlinear model reduction
  via discrete empirical interpolation,} \emph{{SIAM} Journal on Scientific
  Computing}, 32, 2737--2764.

\bibitem[{Chen et~al.(2014)Chen, Quarteroni, and Rozza}]{Chen_2014}
Chen, P., A.~Quarteroni, and G.~Rozza (2014): \enquote{A weighted empirical
  interpolation method: a priori convergence analysis and applications,}
  \emph{{ESAIM}: Mathematical Modelling and Numerical Analysis}, 48, 943--953.

\bibitem[{Chen(2016)}]{Chen2016}
Chen, Y. (2016): \enquote{A certified natural-norm successive constraint method
  for parametric inf--sup lower bounds,} \emph{Applied Numerical Mathematics},
  99, 98--108.

\bibitem[{Chen et~al.(2008)Chen, Hesthaven, Maday, and Rodr{\'\i}guez}]{Chen2}
Chen, Y., J.~S. Hesthaven, Y.~Maday, and J.~Rodr{\'\i}guez (2008): \enquote{A
  monotonic evaluation of lower bounds for inf-sup stability constants in the
  frame of reduced basis approximations,} \emph{Comptes Rendus Mathematique},
  346, 1295--1300.

\bibitem[{Chen et~al.(2009)Chen, Hesthaven, Maday, and Rodr{\'\i}guez}]{Chen1}
Chen, Y., J.~S. Hesthaven, Y.~Maday, and J.~Rodr{\'\i}guez (2009):
  \enquote{Improved successive constraint method based a posteriori error
  estimate for reduced basis approximation of 2d {M}axwell's problem,}
  \emph{ESAIM: Mathematical Modelling and Numerical Analysis}, 43, 1099--1116.

\bibitem[{Chinesta et~al.(2017)Chinesta, Huerta, Rozza, and
  Willcox}]{ChinestaHuerta2017}
Chinesta, F., A.~Huerta, G.~Rozza, and K.~Willcox (2017): \enquote{Model order
  reduction,} Encyclopedia of Computational Mechanics Second Edition, Elsevier,
  1--36.

\bibitem[{Ciarlet(2002)}]{ciarlet2002}
Ciarlet, P. (2002): \emph{The Finite Element Method for Elliptic Problems},
  volume~40, Classics in Applied Mathematics, Society for Industrial and
  Applied Mathematics, Philadelphia.

\bibitem[{Ciarlet(2014)}]{Ciarlet2014}
Ciarlet, P. (2014): \emph{Linear and nonlinear functional analysis with
  applications}, Society for Industrial and Applied Mathematics, Philadelphia.

\bibitem[{Constantine(2015)}]{constantine2015active}
Constantine, P.~G. (2015): \emph{{Active subspaces: Emerging ideas for
  dimension reduction in parameter studies}}, volume~2, SIAM.

\bibitem[{Constantine and Doostan(2017)}]{constantine2017time}
Constantine, P.~G. and A.~Doostan (2017): \enquote{Time-dependent global
  sensitivity analysis with active subspaces for a lithium ion battery model,}
  \emph{Statistical Analysis and Data Mining: The ASA Data Science Journal},
  10, 243--262.

\bibitem[{Constantine et~al.(2014)Constantine, Dow, and
  Wang}]{constantine2014active}
Constantine, P.~G., E.~Dow, and Q.~Wang (2014): \enquote{Active subspace
  methods in theory and practice: applications to kriging surfaces,} \emph{SIAM
  Journal on Scientific Computing}, 36, A1500--A1524.

\bibitem[{Constantine et~al.(2015)Constantine, Emory, Larsson, and
  Iaccarino}]{constantine2015exploiting}
Constantine, P.~G., M.~Emory, J.~Larsson, and G.~Iaccarino (2015):
  \enquote{Exploiting active subspaces to quantify uncertainty in the numerical
  simulation of the {HyShot} {II} scramjet,} \emph{Journal of Computational
  Physics}, 302, 1--20.

\bibitem[{Constantine et~al.(2016)Constantine, Kent, and
  Bui-Thanh}]{constantine2016accelerating}
Constantine, P.~G., C.~Kent, and T.~Bui-Thanh (2016): \enquote{Accelerating
  {M}arkov chain {M}onte {C}arlo with active subspaces,} \emph{SIAM Journal on
  Scientific Computing}, 38, A2779--A2805.

\bibitem[{Cook(2009)}]{cook2009regression}
Cook, R.~D. (2009): \emph{Regression graphics: ideas for studying regressions
  through graphics}, volume 482, John Wiley \& Sons.

\bibitem[{Demo et~al.(2018{\natexlab{a}})Demo, Tezzele, Gustin, Lavini, and
  Rozza}]{demo2018shape}
Demo, N., M.~Tezzele, G.~Gustin, G.~Lavini, and G.~Rozza (2018{\natexlab{a}}):
  \enquote{Shape optimization by means of proper orthogonal decomposition and
  dynamic mode decomposition,} in \emph{Technology and Science for the Ships of
  the Future: Proceedings of NAV 2018: 19th International Conference on Ship \&
  Maritime Research}, IOS Press, 212--219.

\bibitem[{Demo et~al.(2018{\natexlab{b}})Demo, Tezzele, Mola, and
  Rozza}]{demo2018isope}
Demo, N., M.~Tezzele, A.~Mola, and G.~Rozza (2018{\natexlab{b}}): \enquote{An
  efficient shape parametrisation by free-form deformation enhanced by active
  subspace for hull hydrodynamic ship design problems in open source
  environment,} in \emph{The 28th International Ocean and Polar Engineering
  Conference}.

\bibitem[{Demo et~al.(2019)Demo, Tezzele, and Rozza}]{demo2019cras}
Demo, N., M.~Tezzele, and G.~Rozza (2019): \enquote{A non-intrusive approach
  for proper orthogonal decomposition modal coefficients reconstruction through
  active subspaces,} \emph{Comptes Rendus de l'Acad\'emie des Sciences DataBEST
  2019 Special Issue}.

\bibitem[{Deparis et~al.(2014)Deparis, Forti, and
  Quarteroni}]{deparis2014rescaled}
Deparis, S., D.~Forti, and A.~Quarteroni (2014): \enquote{A rescaled localized
  radial basis function interpolation on non-{C}artesian and nonconforming
  grids,} \emph{SIAM Journal on Scientific Computing}, 36, A2745--A2762.

\bibitem[{Deparis and Rozza(2009)}]{Deparis2009}
Deparis, S. and G.~Rozza (2009): \enquote{Reduced basis method for
  multi-parameter-dependent steady navier--stokes equations: Applications to
  natural convection in a cavity,} \emph{Journal of Computational Physics},
  228, 4359--4378.

\bibitem[{Duchon(1977)}]{duchon1977splines}
Duchon, J. (1977): \enquote{Splines minimizing rotation-invariant semi-norms in
  {Sobolev} spaces,} \emph{Constructive theory of functions of several
  variables}, 85--100.

\bibitem[{Eckart and Young(1936)}]{Eckart1936}
Eckart, C. and G.~Young (1936): \enquote{The approximation of one matrix by
  another of lower rank. psychometrika,} \emph{Psychometrika}, 1, 211--218.

\bibitem[{Eftang et~al.(2010)Eftang, Grepl, and Patera}]{Eftang2010}
Eftang, J.~L., M.~A. Grepl, and A.~T. Patera (2010): \enquote{A posteriori
  error bounds for the empirical interpolation method,} \emph{Comptes Rendus
  Mathematique}, 348, 575--579.

\bibitem[{Eftang and Stamm(2012)}]{Eftang_2012}
Eftang, J.~L. and B.~Stamm (2012): \enquote{Parameter multi-domain `hp'
  empirical interpolation,} \emph{International Journal for Numerical Methods
  in Engineering}, 90, 412--428.

\bibitem[{Eriksson et~al.(2018)Eriksson, Dong, Lee, Bindel, and
  Wilson}]{eriksson2018scaling}
Eriksson, D., K.~Dong, E.~H. Lee, D.~Bindel, and A.~G. Wilson (2018):
  \enquote{Scaling {G}aussian process regression with derivatives,} \emph{arXiv
  preprint arXiv:1810.12283}.

\bibitem[{Evans(1998)}]{Evans1998}
Evans, L. (1998): \emph{Partial Differential Equations}, American Mathematical
  Society.

\bibitem[{Eymard et~al.(2000)Eymard, Gallou\"et, and Herbin}]{Eymard2000}
Eymard, R., T.~R. Gallou\"et, and R.~Herbin (2000): \enquote{The finite volume
  method,} in \emph{Handbook of Numerical Analysis}, volume~7, Editors: P.G.
  Ciarlet and J.L. Lions, 713--1020.

\bibitem[{Forti and Rozza(2014)}]{forti2014efficient}
Forti, D. and G.~Rozza (2014): \enquote{Efficient geometrical parametrisation
  techniques of interfaces for reduced-order modelling: application to
  fluid--structure interaction coupling problems,} \emph{International Journal
  of Computational Fluid Dynamics}, 28, 158--169.

\bibitem[{Glaws et~al.(2017)Glaws, Constantine, Shadid, and
  Wildey}]{glaws2017dimension}
Glaws, A., P.~G. Constantine, J.~N. Shadid, and T.~M. Wildey (2017):
  \enquote{Dimension reduction in magnetohydrodynamics power generation models:
  Dimensional analysis and active subspaces,} \emph{Statistical Analysis and
  Data Mining: The ASA Data Science Journal}, 10, 312--325.

\bibitem[{Gonz{\'a}lez et~al.(2016)Gonz{\'a}lez, Cueto, and
  Chinesta}]{gonzalez2016computational}
Gonz{\'a}lez, D., E.~Cueto, and F.~Chinesta (2016): \enquote{Computational
  patient avatars for surgery planning,} \emph{Annals of biomedical
  engineering}, 44, 35--45.

\bibitem[{Graetsch and Bathe(2005)}]{GRATSCH2005235}
Graetsch, T. and K.-J. Bathe (2005): \enquote{A posteriori error estimation
  techniques in practical finite element analysis,} \emph{Computers \&
  Structures}, 83, 235 -- 265.

\bibitem[{Grepl et~al.(2007)Grepl, Maday, Nguyen, and Patera}]{Grepl2007}
Grepl, M.~A., Y.~Maday, N.~C. Nguyen, and A.~T. Patera (2007):
  \enquote{Efficient reduced-basis treatment of nonaffine and nonlinear partial
  differential equations,} \emph{{ESAIM}: Mathematical Modelling and Numerical
  Analysis}, 41, 575--605.

\bibitem[{Hess et~al.(2015)Hess, Grundel, and Benner}]{Hess:2015}
Hess, M., S.~Grundel, and P.~Benner (2015): \enquote{Estimating the inf-sup
  constant in reduced basis methods for time-harmonic {M}axwell's equations,}
  \emph{IEEE Transactions on Microwave Theory and Techniques}, 63, 3549--3557.

\bibitem[{Hesthaven et~al.(2012)Hesthaven, Stamm, and Zhang}]{Hesthaven2012}
Hesthaven, J., B.~Stamm, and S.~Zhang (2012): \enquote{Certified reduced basis
  method for the electric field integral equation,} \emph{SIAM Journal on
  Scientific Computing}, 34, A1777--A1799.

\bibitem[{Hesthaven et~al.(2014)Hesthaven, Stamm, and Zhang}]{Hesthaven_2014}
Hesthaven, J.~S., B.~Stamm, and S.~Zhang (2014): \enquote{Efficient greedy
  algorithms for high-dimensional parameter spaces with applications to
  empirical interpolation and reduced basis methods,} \emph{{ESAIM}:
  Mathematical Modelling and Numerical Analysis}, 48, 259--283.

\bibitem[{Huynh et~al.(2007)Huynh, Rozza, Sen, and Patera}]{SCM_original}
Huynh, D., G.~Rozza, S.~Sen, and A.~Patera (2007): \enquote{A successive
  constraint linear optimization method for lower bounds of parametric
  coercivity and inf-sup stability constants,} \emph{Comptes Rendus
  Mathematique}, 345, 473--478.

\bibitem[{Huynh et~al.(2010)Huynh, Knezevic, Chen, Hesthaven, and
  Patera}]{Huynh2010}
Huynh, D. B.~P., D.~J. Knezevic, Y.~Chen, J.~S. Hesthaven, and A.~T. Patera
  (2010): \enquote{A natural-norm successive constraint method for inf-sup
  lower bounds,} \emph{Computer Methods in Applied Mechanics and Engineering},
  199, 1963--1975.

\bibitem[{Iapichino et~al.(2017)Iapichino, Ulbrich, and
  Volkwein}]{Iapichino2017}
Iapichino, L., S.~Ulbrich, and S.~Volkwein (2017): \enquote{Multiobjective
  {PDE}-constrained optimization using the reduced-basis method,}
  \emph{Advances in Computational Mathematics}, 43, 945--972.

\bibitem[{Ibanez et~al.(2018)Ibanez, Abisset-Chavanne, Aguado, Gonzalez, Cueto,
  and Chinesta}]{ibanez2018manifold}
Ibanez, R., E.~Abisset-Chavanne, J.~V. Aguado, D.~Gonzalez, E.~Cueto, and
  F.~Chinesta (2018): \enquote{A manifold learning approach to data-driven
  computational elasticity and inelasticity,} \emph{Archives of Computational
  Methods in Engineering}, 25, 47--57.

\bibitem[{Kolmogoroff(1936)}]{Kolmogoroff1936}
Kolmogoroff, A. (1936): \enquote{Uber die beste annaherung von funktionen einer
  gegebenen funktionenklasse,} \emph{The Annals of Mathematics}, 37, 107,
  \urlprefix\url{https://doi.org/10.2307/1968691}.

\bibitem[{Lam et~al.(2018)Lam, Zahm, Marzouk, and
  Willcox}]{lam2018multifidelity}
Lam, R., O.~Zahm, Y.~Marzouk, and K.~Willcox (2018): \enquote{Multifidelity
  dimension reduction via active subspaces,} \emph{arXiv preprint
  arXiv:1809.05567}.

\bibitem[{Lassila and Rozza(2010)}]{LassilaRozza2010}
Lassila, T. and G.~Rozza (2010): \enquote{Parametric free-form shape design
  with {PDE} models and reduced basis method,} \emph{Computer Methods in
  Applied Mechanics and Engineering}, 199, 1583--1592.

\bibitem[{Lee and Verleysen(2007)}]{lee2007nonlinear}
Lee, J.~A. and M.~Verleysen (2007): \emph{Nonlinear dimensionality reduction},
  Springer Science \& Business Media.

\bibitem[{Lombardi et~al.(2012)Lombardi, Parolini, Quarteroni, and
  Rozza}]{lombardi2012numerical}
Lombardi, M., N.~Parolini, A.~Quarteroni, and G.~Rozza (2012):
  \enquote{{Numerical simulation of sailing boats: Dynamics, {FSI}, and shape
  optimization},} in \emph{Variational Analysis and Aerospace Engineering:
  Mathematical Challenges for Aerospace Design}, Springer, 339.

\bibitem[{Lukaczyk et~al.(2014)Lukaczyk, Constantine, Palacios, and
  Alonso}]{lukaczyk2014active}
Lukaczyk, T.~W., P.~Constantine, F.~Palacios, and J.~J. Alonso (2014):
  \enquote{Active subspaces for shape optimization,} in \emph{10th AIAA
  multidisciplinary design optimization conference}, 1171.

\bibitem[{Maday and Mula(2013)}]{Maday2013}
Maday, Y. and O.~Mula (2013): \enquote{A generalized empirical interpolation
  method: Application of reduced basis techniques to data assimilation,} in
  \emph{Analysis and Numerics of Partial Differential Equations}, Springer
  Milan, 221--235.

\bibitem[{Maday et~al.(2016)Maday, Mula, and Turinici}]{Maday2016}
Maday, Y., O.~Mula, and G.~Turinici (2016): \enquote{Convergence analysis of
  the generalized empirical interpolation method,} \emph{{SIAM} Journal on
  Numerical Analysis}, 54, 1713--1731.

\bibitem[{Maday et~al.(2008)Maday, Nguyen, Patera, and Pau}]{Maday2008}
Maday, Y., N.~Nguyen, A.~Patera, and S.~Pau (2008): \enquote{A general
  multipurpose interpolation procedure: the magic points,} \emph{Communications
  on Pure and Applied Analysis}, 8, 383--404.

\bibitem[{Manzoni(2014)}]{manzoni_2014}
Manzoni, A. (2014): \enquote{An efficient computational framework for reduced
  basis approximation and a posteriori error estimation of parametrized
  navier--stokes flows,} \emph{ESAIM: Mathematical Modelling and Numerical
  Analysis}, 48, 1199--1226.

\bibitem[{Manzoni and Negri(2015)}]{Manzoni2015}
Manzoni, A. and F.~Negri (2015): \enquote{Heuristic strategies for the
  approximation of stability factors in quadratically nonlinear parametrized
  pdes,} \emph{Advances in Computational Mathematics}, 41, 1255--1288.

\bibitem[{Manzoni et~al.(2012)Manzoni, Quarteroni, and
  Rozza}]{manzoni2012model}
Manzoni, A., A.~Quarteroni, and G.~Rozza (2012): \enquote{Model reduction
  techniques for fast blood flow simulation in parametrized geometries,}
  \emph{International journal for numerical methods in biomedical engineering},
  28, 604--625.

\bibitem[{Meng et~al.(2018)Meng, Breitkopf, Le~Quilliec, Raghavan, and
  Villon}]{meng2018nonlinear}
Meng, L., P.~Breitkopf, G.~Le~Quilliec, B.~Raghavan, and P.~Villon (2018):
  \enquote{Nonlinear shape-manifold learning approach: concepts, tools and
  applications,} \emph{Archives of Computational Methods in Engineering}, 25,
  1--21.

\bibitem[{Meng et~al.(2015)Meng, Breitkopf, Raghavan, Mauvoisin, Bartier, and
  Hernot}]{meng2015identification}
Meng, L., P.~Breitkopf, B.~Raghavan, G.~Mauvoisin, O.~Bartier, and X.~Hernot
  (2015): \enquote{Identification of material properties using indentation test
  and shape manifold learning approach,} \emph{Computer Methods in Applied
  Mechanics and Engineering}, 297, 239--257.

\bibitem[{Morris et~al.(2008)Morris, Allen, and Rendall}]{morris2008cfd}
Morris, A., C.~Allen, and T.~Rendall (2008): \enquote{{CFD}-based optimization
  of aerofoils using radial basis functions for domain element parameterization
  and mesh deformation,} \emph{International Journal for Numerical Methods in
  Fluids}, 58, 827--860.

\bibitem[{Noor(1982)}]{Noor1982}
Noor, A. (1982): \enquote{On making large nonlinear problems small,}
  \emph{Comput. Meth. Appl. Mech. Engrg.}, 34, 955--985.

\bibitem[{Pinkus(2015)}]{pinkus2015ridge}
Pinkus, A. (2015): \emph{Ridge functions}, volume 205, Cambridge University
  Press.

\bibitem[{Prud'Homme et~al.(2002)Prud'Homme, Rovas, Veroy, Machiels, Maday,
  Patera, and Turinici}]{PrudHomme2002}
Prud'Homme, C., D.~V. Rovas, K.~Veroy, L.~Machiels, Y.~Maday, A.~T. Patera, and
  G.~Turinici (2002): \enquote{Reliable real-time solution of parametrized
  partial differential equations: Reduced-basis output bound methods,}
  \emph{Journal of Fluids Engineering}, 124, 70--80.

\bibitem[{PyGeM(2017)}]{pygem}
PyGeM (2017): \enquote{{Python Geometrical Morphing},}
  \urlprefix\url{https://github.com/mathLab/PyGeM}.

\bibitem[{Quarteroni(2017)}]{quarteroni2017}
Quarteroni, A. (2017): \emph{Numerical Models for Differential Problems},
  \emph{Modeling, Simulation and Applications}, volume~16, Springer
  International Publishing.

\bibitem[{Rebollo et~al.(2017)Rebollo, {\'A}vila, M{\'a}rmol, Ballarin, and
  Rozza}]{Enrique2017}
Rebollo, T., E.~{\'A}vila, M.~M{\'a}rmol, F.~Ballarin, and G.~Rozza (2017):
  \enquote{On a certified {S}magorinsky reduced basis turbulence model,}
  \emph{SIAM Journal on Numerical Analysis}, 55, 3047--3067.

\bibitem[{Renardy and Rogers(2004)}]{Renardy2004}
Renardy, M. and R.~Rogers (2004): \emph{An Introduction to Partial Differential
  Equations}, Springer-Verlag New York.

\bibitem[{Roweis and Saul(2000)}]{roweis2000nonlinear}
Roweis, S.~T. and L.~K. Saul (2000): \enquote{Nonlinear dimensionality
  reduction by locally linear embedding,} \emph{science}, 290, 2323--2326.

\bibitem[{Rozza et~al.(2008)Rozza, Huynh, and Patera}]{Patera:2008}
Rozza, G., D.~Huynh, and A.~T. Patera (2008): \enquote{Reduced {B}asis
  {A}pproximation and a {P}osteriori {E}rror {E}stimation for {A}ffinely
  {P}arametrized {E}lliptic {C}oercive {P}artial {D}ifferential {E}quations,}
  \emph{Archives of Computational Methods in Engineering}, 15, 229 -- 275.

\bibitem[{Rozza et~al.(2009)Rozza, Huynh, Nguyen, and Patera}]{RozzaHuynh2009}
Rozza, G., D.~B.~P. Huynh, N.~C. Nguyen, and A.~T. Patera (2009):
  \enquote{Real-time reliable simulation of heat transfer phenomena,} in
  \emph{ASME -American Society of Mechanical Engineers - Heat Transfer Summer
  Conference, paper HT2009-88212, volume 3}, 851--860.

\bibitem[{Rozza et~al.(2013)Rozza, Koshakji, and Quarteroni}]{rozza2013free}
Rozza, G., A.~Koshakji, and A.~Quarteroni (2013): \enquote{{Free Form
  Deformation techniques applied to {3D} shape optimization problems},}
  \emph{Communications in Applied and Industrial Mathematics}, 4, 1--26.

\bibitem[{Rozza et~al.(2018)Rozza, Malik, Demo, Tezzele, Girfoglio, Stabile,
  and Mola}]{rozza2018advances}
Rozza, G., M.~H. Malik, N.~Demo, M.~Tezzele, M.~Girfoglio, G.~Stabile, and
  A.~Mola (2018): \enquote{{Advances in Reduced Order Methods for Parametric
  Industrial Problems in Computational Fluid Dynamics},} Glasgow, UK: ECCOMAS
  ECCM - ECFD Conference Proceedings.

\bibitem[{Rudin(1976)}]{Rudin1976}
Rudin, W. (1976): \emph{Principles of Mathematical Analysis}, International
  Series in Pure and Applied Mathematics. McGraw-Hill.

\bibitem[{Russi(2010)}]{russi2010uncertainty}
Russi, T.~M. (2010): \emph{Uncertainty quantification with experimental data
  and complex system models}, Ph.D. thesis, UC Berkeley.

\bibitem[{Salmoiraghi et~al.(2016)Salmoiraghi, Ballarin, Corsi, Mola, Tezzele,
  and Rozza}]{salmoiraghi2016advances}
Salmoiraghi, F., F.~Ballarin, G.~Corsi, A.~Mola, M.~Tezzele, and G.~Rozza
  (2016): \enquote{Advances in geometrical parametrization and reduced order
  models and methods for computational fluid dynamics problems in applied
  sciences and engineering: Overview and perspectives,} \emph{ECCOMAS Congress
  2016 - Proceedings of the 7th European Congress on Computational Methods in
  Applied Sciences and Engineering}, 1, 1013--1031.

\bibitem[{Salmoiraghi et~al.(2018)Salmoiraghi, Scardigli, Telib, and
  Rozza}]{salmoiraghi2018}
Salmoiraghi, F., A.~Scardigli, H.~Telib, and G.~Rozza (2018):
  \enquote{Free-form deformation, mesh morphing and reduced-order methods:
  enablers for efficient aerodynamic shape optimisation,} \emph{International
  Journal of Computational Fluid Dynamics}, 32, 233--247.

\bibitem[{Sandwell(1987)}]{sandwell1987biharmonic}
Sandwell, D.~T. (1987): \enquote{Biharmonic spline interpolation of {GEOS-3}
  and {SEASAT} altimeter data,} \emph{Geophysical research letters}, 14,
  139--142.

\bibitem[{Schilders et~al.(2008)Schilders, van~der Vorst, and
  Rommes}]{Schilders2008}
Schilders, W.~H., H.~A. van~der Vorst, and J.~Rommes (2008): \emph{Model Order
  Reduction: Theory, Research Aspects and Applications}, Springer-Verlag Berlin
  Heidelberg.

\bibitem[{Sederberg and Parry(1986)}]{sederbergparry1986}
Sederberg, T. and S.~Parry (1986): \enquote{Free-{F}orm {D}eformation of solid
  geometric models,} in \emph{Proceedings of SIGGRAPH - Special Interest Group
  on GRAPHics and Interactive Techniques}, SIGGRAPH, 151--159.

\bibitem[{Shepard(1968)}]{shepard1968}
Shepard, D. (1968): \enquote{A two-dimensional interpolation function for
  irregularly-spaced data,} in \emph{Proceedings-1968 ACM National Conference},
  ACM, 517--524.

\bibitem[{Sieger et~al.(2015)Sieger, Menzel, and Botsch}]{sieger2015shape}
Sieger, D., S.~Menzel, and M.~Botsch (2015): \enquote{On shape deformation
  techniques for simulation-based design optimization,} in \emph{New Challenges
  in Grid Generation and Adaptivity for Scientific Computing}, Springer,
  281--303.

\bibitem[{Smith(1985)}]{Smith1985}
Smith, G.~D. (1985): \emph{Numerical Solution of Partial Differential
  Equations: Finite Difference Methods}, Clarendon Press, Oxford Applied
  Mathematics and Computing Science Series.

\bibitem[{Tenenbaum et~al.(2000)Tenenbaum, De~Silva, and
  Langford}]{tenenbaum2000global}
Tenenbaum, J.~B., V.~De~Silva, and J.~C. Langford (2000): \enquote{A global
  geometric framework for nonlinear dimensionality reduction,} \emph{science},
  290, 2319--2323.

\bibitem[{Tezzele et~al.(2018{\natexlab{a}})Tezzele, Ballarin, and
  Rozza}]{tezzele2018combined}
Tezzele, M., F.~Ballarin, and G.~Rozza (2018{\natexlab{a}}): \enquote{{Combined
  parameter and model reduction of cardiovascular problems by means of active
  subspaces and POD-Galerkin methods},} in D.~Boffi, L.~F. Pavarino, G.~Rozza,
  S.~Scacchi, and C.~Vergara, eds., \emph{Mathematical and Numerical Modeling
  of the Cardiovascular System and Applications}, Springer International
  Publishing, 185--207.

\bibitem[{Tezzele et~al.(2018{\natexlab{b}})Tezzele, Demo, Gadalla, Mola, and
  Rozza}]{tezzele2018model}
Tezzele, M., N.~Demo, M.~Gadalla, A.~Mola, and G.~Rozza (2018{\natexlab{b}}):
  \enquote{Model order reduction by means of active subspaces and dynamic mode
  decomposition for parametric hull shape design hydrodynamics,} in
  \emph{Technology and Science for the Ships of the Future: Proceedings of NAV
  2018: 19th International Conference on Ship \& Maritime Research}, IOS Press,
  569--576.

\bibitem[{Tezzele et~al.(2019)Tezzele, Demo, and Rozza}]{tezzele2019shape}
Tezzele, M., N.~Demo, and G.~Rozza (2019): \enquote{Shape optimization through
  proper orthogonal decomposition with interpolation and dynamic mode
  decomposition enhanced by active subspaces,} in \emph{Proceedings of MARINE
  2019: VIII International Conference on Computational Methods in Marine
  Engineering}, 122--133.

\bibitem[{Tezzele et~al.(2018{\natexlab{c}})Tezzele, Salmoiraghi, Mola, and
  Rozza}]{tezzele2018dimension}
Tezzele, M., F.~Salmoiraghi, A.~Mola, and G.~Rozza (2018{\natexlab{c}}):
  \enquote{Dimension reduction in heterogeneous parametric spaces with
  application to naval engineering shape design problems,} \emph{Advanced
  Modeling and Simulation in Engineering Sciences}, 5, 25.

\bibitem[{Vallagh{\'e} et~al.(2011)Vallagh{\'e}, Fouquembergh, Le~Hyaric, and
  Prud'Homme}]{Vallaghe2011}
Vallagh{\'e}, S., M.~Fouquembergh, A.~Le~Hyaric, and C.~Prud'Homme (2011):
  \enquote{A successive constraint method with minimal offline constraints for
  lower bounds of parametric coercivity constant,}
  \url{https://hal.archives-ouvertes.fr/hal-00609212}.

\bibitem[{Van Der~Maaten et~al.(2009)Van Der~Maaten, Postma, and Van~den
  Herik}]{van2009dimensionality}
Van Der~Maaten, L., E.~Postma, and J.~Van~den Herik (2009):
  \enquote{Dimensionality reduction: a comparative review,} \emph{J Mach Learn
  Res}, 10, 66--71.

\bibitem[{Verfuerth(2013)}]{Verfuerth2013}
Verfuerth, R. (2013): \emph{A posteriori error estimation techniques for finite
  element methods}, Oxford Univ. Press.

\bibitem[{Veroy and Patera(2005)}]{Veroy2005}
Veroy, K. and A.~T. Patera (2005): \enquote{Certified real-time solution of the
  parametrized steady incompressible {N}avier--{S}tokes equations: rigorous
  reduced-basis a posteriori error bounds,} \emph{International Journal for
  Numerical Methods in Fluids}, 47, 773--788.

\bibitem[{Veroy et~al.(2002)Veroy, Rovas, and Patera}]{Veroy2002}
Veroy, K., D.~V. Rovas, and A.~T. Patera (2002): \enquote{A posteriori error
  estimation for reduced-basis approximation of parametrized elliptic coercive
  partial differential equations : ``convex inverse'' bound conditioners,}
  \emph{ESAIM: Control, Optimisation and Calculus of Variations}, 8,
  1007--1028.

\bibitem[{Witteveen and Bijl(2009)}]{witteveenbijl2009}
Witteveen, J. and H.~Bijl (2009): \enquote{Explicit mesh deformation using
  {I}nverse {D}istance {W}eighting interpolation,} in \emph{19th AIAA
  Computational Fluid Dynamics}, AIAA.

\bibitem[{Yano(2014)}]{Yano2014}
Yano, M. (2014): \enquote{A space-time {P}etrov--{G}alerkin certified reduced
  basis method: Application to the boussinesq equations,} \emph{SIAM Journal on
  Scientific Computing}, 36, A232--A266.

\bibitem[{Yosida(1995)}]{Yosida1995}
Yosida, K. (1995): \emph{Functional Analysis}, Springer-Verlag Berlin
  Heidelberg.

\bibitem[{Zhang(2011)}]{Zhang2011}
Zhang, S. (2011): \enquote{Efficient greedy algorithms for successive
  constraints methods with high-dimensional parameters,} Tech. Report 23,
  \url{http://www.dam.brown.edu/people/shzhang/greedy_scm.pdf}.

\end{thebibliography}
\end{document}